\documentclass[11pt]{article}

\usepackage{graphicx, amsmath,amsthm,amssymb,url}

\usepackage[labelfont=bf]{subcaption}
\usepackage[labelfont=bf]{caption}
\usepackage{algorithm,algcompatible,mathtools}
\usepackage[multidot]{grffile}
\usepackage[normalem]{ulem} 

\usepackage{fullpage,etoolbox}
\usepackage{color}
\usepackage[square,numbers]{natbib}

\newtheorem{theorem}{Theorem}
\numberwithin{theorem}{section}
\newtheorem{lemma}[theorem]{Lemma}
\newtheorem{proposition}[theorem]{Proposition}
\newtheorem{coro}[theorem]{Corollary} 

\numberwithin{equation}{section}

\newtoggle{draft}
\togglefalse{draft}

\newcommand{\minimize}{\mbox{minimize}}

\newcommand{\st}{\mbox{subject to}}

\newcommand{\tf}[1]{\mathbf{#1}}
\newcommand{\Ah}{\widehat{A}}
\newcommand{\Ahat}{\Ah}
\newcommand{\Bhat}{\Bh}
\newcommand{\Bh}{\widehat{B}}
\newcommand{\Kh}{\widehat{K}}
\newcommand{\Jh}{\widehat{J}}

\newcommand{\Dh}{\hat{\tf{\Delta}}}
\newcommand{\trueA}{A}
\newcommand{\trueB}{B}
\newcommand{\trueK}{K_\star}

\DeclareMathOperator*{\Tr}{\mathbf{Tr}}

\newcommand{\R}{\ensuremath{\mathbb{R}}}

\newcommand{\norm}[1]{\lVert #1 \rVert}
\newcommand{\bignorm}[1]{\left\lVert #1 \right\rVert}
\newcommand{\twonorm}[1]{\lVert #1 \rVert_{2}}
\newcommand{\bigtwonorm}[1]{\left\lVert #1 \right\rVert_{2}}
\newcommand{\spectralnorm}[1]{\twonorm{#1}}
\newcommand{\bigspectralnorm}[1]{\bigtwonorm{#1}}

\newcommand{\E}{\mathbb{E}}

\renewcommand{\Pr}{\mathbb{P}}
\newcommand{\T}{*}
\newcommand{\Ncal}{\mathcal{N}}
\newcommand{\vecx}{{x}}
\newcommand{\vecw}{{w}}
\newcommand{\vecu}{{u}}
\newcommand{\PP}{\mathbb{P}}
\newcommand{\RR}{\mathbb{R}}
\newcommand{\Res}[1]{\mathfrak{R}_{#1}}

\newcommand{\statedim}{n}
\newcommand{\inputdim}{p}
\newcommand{\hinf}{\mathcal{H}_\infty}
\newcommand{\htwo}{\mathcal{H}_2}

\newcommand{\ltwonorm}[1]{\| #1 \|_2}
\newcommand{\hinfnorm}[1]{\| #1 \|_{\hinf}}
\newcommand{\bighinfnorm}[1]{\left \| #1 \right\|_{\hinf}}
\newcommand{\bightwonorm}[1]{\left \| #1 \right\|_{\htwo}}
\newcommand{\iid}{\stackrel{\mathclap{\text{\scriptsize{ \tiny i.i.d.}}}}{\sim}}

\def \basefigwidth{0.49}

\title{On the Sample Complexity of the Linear Quadratic Regulator}
\author{Sarah Dean$^\sharp$, Horia Mania$^\sharp$,  Nikolai Matni$^\dagger$,
Benjamin Recht$^\sharp$, and Stephen Tu$^\sharp$
\vspace{0.0625in}
\\
$^\sharp$ University of California, Berkeley \\
$\dagger$ California Institute of Technology}

\date{October 3, 2017, Revised: \today}

\begin{document}
\maketitle
\begin{abstract}
This paper addresses the optimal control problem known as the Linear Quadratic Regulator in the case when the dynamics are unknown.  We propose a multi-stage procedure, called \emph{Coarse-ID control}, that estimates a model from a few experimental trials, estimates the error in that model with respect to the truth, and then designs a controller using both the model and uncertainty estimate.  Our technique uses contemporary tools from random matrix theory to bound the error in the estimation procedure.  We also employ a recently developed approach to control synthesis called \emph{System Level Synthesis} that enables robust control design by solving a quasiconvex optimization problem.  We provide end-to-end bounds on the relative error in control cost that are optimal in the number of parameters and that highlight salient properties of the system to be controlled such as closed-loop sensitivity and optimal control magnitude.  We show experimentally that the Coarse-ID approach enables efficient computation of a stabilizing controller in regimes where simple control schemes that do not take the model uncertainty into account fail to stabilize the true system.
\end{abstract}

\section{Introduction}
\label{sec:intro}
Having surpassed human performance in video games~\cite{atari-nature} and Go~\cite{alphago}, there has been a renewed interest in applying machine learning techniques to planning and control.  In particular, there has been a considerable amount of effort in developing new techniques for \emph{continuous control} where an autonomous system interacts with a physical environment~\cite{duan2016benchmarking,Levine16}.  A tremendous opportunity lies in deploying these data-driven systems in more demanding interactive tasks including self-driving vehicles, distributed sensor networks, and agile robotics.  As the role of machine learning expands to more ambitious tasks, however, it is critical these new technologies be safe and reliable. Failure of such systems could have severe social and economic consequences including the potential loss of human life. How can we guarantee that our new data-driven automated systems are robust?

Unfortunately, there are no clean baselines delineating the possible control performance achievable given a fixed amount of data collected from a system. Such baselines would enable comparisons of different techniques and would allow engineers to trade off between data collection and action in scenarios with high uncertainty.  Typically, a key difficulty in establishing baselines is in proving \emph{lower bounds} that state the minimum amount of knowledge needed to achieve a particular performance, regardless of method. However, in the context of controls, even \emph{upper bounds} describing the worst-case performance of competing methods are exceptionally rare.  Without such estimates, we are left to compare algorithms on a case-by-case basis, and we may have trouble diagnosing whether poor performance is due to algorithm choice or some other error such as a software bug or a mechanical flaw.

In this paper, we attempt to build a foundation for a theoretical understanding of how machine learning interfaces with control by analyzing one of the most well-studied problems in classical optimal control, the \emph{Linear Quadratic Regulator} (LQR).  Here we assume that the system to be controlled obeys \emph{linear} dynamics, and we wish to minimize some \emph{quadratic} function of the system state and control action.  This problem has been studied for decades in control: it has a simple, closed form solution on the infinite time horizon and an efficient, dynamic programming solution on finite time horizons.  When the dynamics are unknown, however, there are far fewer results about achievable performance.

Our contribution is to analyze the LQR problem when the dynamics of the system are unknown, and we can measure the system's response to varied inputs.  A na{\"{i}}ve solution to this problem would be to collect some data of how the system behaves over time, fit a model to this data, and then solve the original LQR problem assuming this model is accurate. Unfortunately, while this procedure might perform well given sufficient data, it is difficult to determine how many experiments are necessary in practice. Furthermore, it is easy to construct examples where the procedure fails to find a stabilizing controller.

As an alternative, we propose a method that couples our uncertainty in estimation with the control design.  Our main approach uses the following framework of \emph{Coarse-ID control} to solve the problem of LQR with unknown dynamics:
\begin{enumerate}
\item Use supervised learning to learn a coarse model of the dynamical system to be controlled.  We refer to the system estimate as the \emph{nominal system}.
\item Using either prior knowledge or statistical tools like the bootstrap, build probabilistic guarantees about the distance between the nominal system and the true, unknown dynamics.
\item Solve a robust optimization problem over controllers that optimizes performance of the nominal system while penalizing signals with respect to the estimated uncertainty, ensuring stable and robust execution.
\end{enumerate}
We will show that for a sufficient number of observations of the system, this
approach is guaranteed to return a control policy with small relative cost.
In particular, it guarantees asymptotic stability of the closed-loop system. In
the case of LQR, step 1 of coarse-ID control simply requires solving a
linear least squares problem, step 2 uses a finite sample theoretical guarantee or a standard bootstrap technique, and step 3
requires solving a small semidefinite program. Analyzing this approach, on the
other hand, requires contemporary techniques in non-asymptotic
statistics and a novel parameterization of control problems
that renders nonconvex problems convex~\cite{virtual,SysLevelSyn1}.

We demonstrate the utility of our method on a simple simulation.  In the presented example, we show that simply using the nominal system to design a control policy frequently results in unstable closed-loop behavior, even when there is an abundance of data from the true system.  However, the Coarse-ID approach finds a stabilizing controller with few system observations.

\subsection{Problem Statement and Our Contributions}

The standard optimal control problem aims to find a control sequence that minimizes an expected cost.  We assume a dynamical system with \emph{state} $x_t\in\R^\statedim$ can be acted on by a \emph{control} $u_t \in \R^\inputdim$ and obeys the stochastic dynamics
\begin{align}
	x_{t+1} = f_t(x_t,u_t,w_t)
\end{align}
where $w_t$ is a random process with $w_t$ independent of $w_{t'}$ for all $t\neq t'$.  Optimal control then seeks to minimize
\begin{align}
\label{eq:general_control}
	\begin{array}{ll}
	\minimize & \E\left[ \frac{1}{T} \sum_{t=1}^T c_t(x_t,u_t)\right]\\
	\st & x_{t+1} = f_t(x_t,u_t,w_t)
	\end{array}\,.
\end{align}
Here, $c_t$ denotes the state-control cost at every time step, { and the input $u_t$ is allowed to depend on the current state $x_t$ and all previous states and actions. In this generality, problem~\eqref{eq:general_control} encapsulates many of the problems considered in the reinforcement learning literature.}

The simplest optimal control problem with continuous state is the Linear Quadratic Regulator (LQR), in which costs are a fixed quadratic function of state and control and the dynamics are linear and time-invariant:
\begin{align}\label{eq:lqr-classic}
	\begin{array}{ll}
	\minimize & \E\left[ \frac{1}{T} \sum_{t=1}^T x_t^* Q x_t + u_{t-1}^* R u_{t-1} \right]\\
	\st & x_{t+1} = A x_t + B u_t + w_t
	\end{array}\,.
\end{align}
Here $Q$ (resp. $R$) is a $\statedim \times \statedim$ (resp. $\inputdim\times \inputdim$) positive definite matrix, $A$ and $B$ are called the \emph{state transition matrices}, and $w_t\in \R^\statedim$ is Gaussian noise with zero-mean and covariance $\Sigma_w$.  Throughout, $M^\T$ denotes the Hermitian transpose of the matrix $M$.

In what follows, we will be concerned with the \emph{infinite time horizon}
variant of the LQR problem where we let the time horizon $T$ go to infinity and
minimize the average cost.  When the dynamics are known, this problem has a
celebrated closed form solution based on the solution of matrix Riccati
equations~\cite{ZDGBook}.  Indeed, the optimal solution sets $u_t = K x_t$ for
a fixed $\inputdim \times \statedim$ matrix $K$, and the corresponding
optimal cost will serve as our gold-standard baseline to which we will compare
the achieved cost of all algorithms.

In the case when the state transition matrices are unknown, fewer results have been established
about what cost is achievable.  We will assume that we can conduct experiments
of the following form: given some initial state $x_0$, we can evolve the
dynamics for $T$ time steps using any control sequence $\{u_0,\ldots,
u_{T-1}\}$, measuring the resulting output $\{x_1,\ldots, x_T\}$.  {If we run
$N$ such independent experiments}, what infinite time horizon control cost is
achievable using only the data collected? For simplicity of bookkeeping, in our analysis we further assume that we can
prepare the system in initial state $x_0 = 0$.

In what follows we will examine the performance of the Coarse-ID control framework in this scenario.  We will estimate the errors accrued by least squares estimates $(\Ah,\Bh)$ of the system dynamics.  This estimation error is not easily handled by standard techniques because the design matrix is highly correlated with the model to be estimated.  Regardless, for theoretical tractability,
 we can build a least squares estimate using only the final sample $(x_{T},x_{T-1},u_{T-1})$ of each of the $N$ experiments. Indeed, in Section~\ref{sec:estimation} we prove the following

\begin{proposition}
\label{prop:independent_estimation}
Define the matrices
\begin{align}
G_T = \begin{bmatrix}
\trueA^{T - 1} \trueB & \trueA^{T - 2}\trueB & \ldots & \trueB
\end{bmatrix} \quad \text{and} \quad
F_T = \begin{bmatrix}
\trueA^{T - 1}  & \trueA^{T - 2} & \ldots & I_\statedim
\end{bmatrix} \:. \label{eq:gramians}
\end{align}
Assume we collect data from the linear, time-invariant system initialized at
$x_0 = 0$, using inputs $u_t~\iid \mathcal{N}(0,\sigma_u^2 I_\inputdim)$ for $t=1,...,T$. Suppose
that the process noise is $w_t~\iid \mathcal{N}(0,\sigma_w^2 I_\statedim)$ and that
\begin{align*}
  N \geq 8(\statedim + \inputdim) + 16\log(4/\delta) \:.
\end{align*}
Then, with probability at least $1-\delta$,
the least squares estimator using only the final sample of each trajectory satisfies both
the inequality
\begin{align}\label{eq:A_error_bound}
  \ltwonorm{\widehat{A} - A} \leq \frac{16 \sigma_w}{\sqrt{\lambda_{\min}(\sigma_u^2 G_TG_T^\T + \sigma_w^2 F_TF_T^\T)}}\sqrt{\frac{(\statedim + 2\inputdim)\log(36/\delta)}{N}} \:,
\end{align}
 and the inequality
\begin{align}\label{eq:B_error_bound}
  \ltwonorm{\widehat{B} - B} \leq \frac{16\sigma_w}{\sigma_u}\sqrt{\frac{(\statedim + 2\inputdim)\log(36/\delta)}{N}}\:.
\end{align}
\end{proposition}

The  details of the estimation procedure are described in Section~\ref{sec:estimation} below.  Note that this estimation result seems to yield an optimal dependence in terms of the number of parameters:
$(A,B)$ together have $\statedim(\statedim + \inputdim)$ parameters to learn and each measurement consists of $\statedim$ values.  Moreover, this proposition further illustrates that not all linear systems are equally easy to estimate. The matrices $G_T G_T^*$ and $F_T F_T^*$ are finite time \emph{controllability Gramians} for the control and noise inputs, respectively .  These are standard objects in control: each eigenvalue/vector pair of such a Gramian characterizes how much input energy is required to move the system in that particular direction of the state-space.  Therefore $\lambda_{\min}\left(\sigma_u^2G_T G_T^* + \sigma_w^2F_T F_T^*\right)$ quantifies the least controllable, and hence most difficult to excite and estimate, mode of the system.  This property is captured nicely in our bound, which indicates that for systems for which all modes are easily excitable (i.e., all modes of the system amplify the applied inputs and disturbances), the identification task becomes easier.

While we cannot compute the operator norm error bounds~\eqref{eq:A_error_bound} and~\eqref{eq:B_error_bound} without knowing the true system matrices $(A,B)$, we present a data-dependent bound in Proposition~\ref{prop:data_dependent}.  Moreover, as we show in Section~\ref{sec:bootstrap}, a simple bootstrap procedure can efficiently upper bound the errors $\epsilon_A:=\ltwonorm{A-\Ah}$ and $\epsilon_B:=\ltwonorm{B-\Bh}$ from simulation.

With our estimates $(\Ah,\Bh)$ and error bounds $(\epsilon_A,\epsilon_B)$ in hand, we can turn to the problem of synthesizing a controller.
We can assert with high probability that $A = \Ah + \Delta_A$, and $B = \Bh + \Delta_B$, for $\ltwonorm{\Delta_A} \leq \epsilon_A$ and $\ltwonorm{\Delta_B} \leq \epsilon_B$, where the size of the error terms is determined by the number of samples $N$ collected.
In light of this, it is natural to pose the following robust variant of the standard LQR optimal control problem \eqref{eq:lqr-classic}, which computes a robustly stabilizing controller that seeks to minimize the worst-case performance of the system given the (high-probability) norm bounds on the perturbations $\Delta_A$ and $\Delta_B$:

\begin{equation}
\begin{array}{rl}
\minimize \:\: \displaystyle\sup\limits_{\substack{\|\Delta_A\|_2\leq \epsilon_A \\ \|\Delta_B\|_2\leq \epsilon_B}} &
	   \lim_{T\to \infty} \frac{1}{T}\sum_{t=1}^T	\E\left[ x_t^* Q x_t + u_{t-1}^* R u_{t-1} \right]\\
\st & x_{t+1} = (\hat{A} + \Delta_A) x_t + (\hat{B} + \Delta_B) u_t + w_t
\end{array} \,.
\label{eq:robust_lqr}
\end{equation}

Although classic methods exist for computing such controllers~\cite{feron1997analysis,paganini1995necessary,sznaier2002convex,wu1995optimal}, they typically require solving nonconvex optimization problems, and it is not readily obvious how to extract interpretable measures of controller performance as a function of the perturbation sizes $\epsilon_A$ and $\epsilon_B$.  To that end, we leverage the recently developed System Level Synthesis (SLS) framework \cite{SysLevelSyn1} to create an alternative robust synthesis procedure.  Described in detail in Section~\ref{sec:robust}, SLS lifts the system description into a higher dimensional space that enables efficient search for controllers. At the cost of some conservatism, we are able to guarantee robust stability of the resulting closed-loop system for all admissible perturbations and bound the performance gap between the resulting controller and the optimal LQR controller.  This is summarized in the following proposition.

\begin{proposition}\label{prop:sls-robust-stylized}
Let $(\Ah,\Bh)$ be estimated via the independent data collection scheme used in Proposition~\ref{prop:independent_estimation} and $\tf \Kh$ synthesized using robust SLS.  Let $\Jh$ denote the infinite time horizon LQR cost accrued by using the controller $\tf \Kh$ and $J_\star$ denote the optimal LQR cost achieved when $(A,B)$ are known.  Then the relative error in the LQR cost is bounded as
\begin{align}
\frac{\Jh - J_\star}{J_\star} \leq \mathcal{O}\left(\mathcal{C}_{\mathrm{LQR}}\sqrt{\frac{(\statedim + \inputdim )\log(1/\delta)}{N}} \right)
\end{align}
with probability $1-\delta$ provided $N$ is sufficiently large.
\end{proposition}

The complexity term $\mathcal{C}_{\mathrm{LQR}}$ depends on the rollout length $T$, the true dynamics, the matrices $(Q,R)$ which define the LQR cost, and the variances $\sigma_u^2$ and $\sigma_w^2$ of the control and noise inputs, respectively.  The $1-\delta$ probability comes from the probability of estimation error from Proposition~\ref{prop:independent_estimation}.  The particular form of $\mathcal{C}_{\mathrm{LQR}}$ and concrete requirements on $N$ are both provided in Section~\ref{sec:subopt}.

Though the optimization problem formulated by SLS is infinite dimensional, in Section~\ref{sec:computation} we provide two finite dimensional upper bounds on the optimization that inherit the stability guarantees of the SLS formulation. Moreover, we show via numerical experiments in Section \ref{sec:experiments} that the
controllers synthesized by our optimization do indeed provide stabilizing
controllers with small relative error.  We further show that settings exist
wherein a na{\"{i}}ve synthesis procedure that ignores the uncertainty in the
state-space parameter estimates produces a controller that performs poorly (or
has unstable closed-loop behavior) relative to the controller synthesized using the SLS
procedure.


\subsection{Related Work}   
\label{sec:related} 
We first describe related work in the \emph{estimation of unknown dynamical systems} and then turn to connections in the literature on \emph{robust control
with uncertain models}. We will end this review with a discussion of a few works from the
reinforcement learning community that have attempted to address the LQR problem and related variants.

\paragraph{Estimation of unknown dynamical systems.}

Estimation of unknown systems, especially linear dynamical systems, has a long history
in the system identification subfield of control theory.
While the text of Ljung~\cite{ljung99} covers the classical asymptotic results,
our interest is primarily in nonasymptotic results.
Early results~\cite{campi2002finite,vidyasagar2008learning} on nonasymptotic
rates for parameter identification featured conservative bounds which
are exponential in the system degree and other relevant quantities.
More recently, Bento et al.~\cite{bento10} show that when the $A$ matrix is
stable and induced by a sparse graph, then one can recover the support of $A$
from a single trajectory using $\ell_1$-penalized least squares.
Furthermore, Hardt et al.~\cite{hardt2016gradient}
provide the first polynomial time guarantee for identifying stable linear systems with outputs.
Their guarantees, however, are in terms of predictive output performance of the
model, and require
an assumption on the true system that is more stringent than stability.
 It is not clear how their statistical risk guarantee can be
used in a downstream robust synthesis procedure.

Next, we turn our attention to system identification of linear systems in the
frequency domain.
A comprehensive text on these methods (which
differ from the aforementioned state-space methods) is the work by Chen and Gu~\cite{chen00}.
For stable systems, Helmicki et al.~\cite{helmicki91}
propose to identify a finite impulse response (FIR) approximation by directly
estimating the first $r$ impulse response coefficients.
This method is analyzed in a non-adversarial probabilistic setting
by \cite{goldenshluger98,tu17}, who prove that a polynomial number of samples are sufficient to recover a
FIR filter which approximates the true system in both $\ell_p$-norm and $\hinf$-norm.
However, transfer function methods do not easily allow for
optimal control with state variables, since they only model the input/output
behavior of the system.

In parallel to the system identification community, identification
of auto-regressive time series models is a widely studied topic in the statistics
literature (see e.g. Box et al.~\cite{box08} for the classical results).
Goldenshluger and Zeevi~\cite{goldenshluger01}
show that the coefficients of
a stationary autoregressive model can be estimated from a single trajectory
of length polynomial in $1/(1-\rho)$ via least squares, where $\rho$ denotes
the stability radius of the process. They also prove that their rate is minimax
optimal.  More recently, several authors
\cite{kuznetsov17,mcdonald17,mohri2010stability} have studied
generalization bounds for non i.i.d.\ data, extending the standard
learning theory guarantees for independent data.
At the crux of these arguments lie various mixing assumptions~\cite{yu94},
which limits the analysis to only hold for stable dynamical systems.
Results in this line of research suggest that systems with smaller
mixing time (i.e. systems that are more stable) are easier to identify
(i.e. take less samples).
Our result in Proposition~\ref{prop:independent_estimation}, however, suggests instead that identification benefits from more easily excitable systems. While our analysis holds when we have access to full state observations,
empirical testing suggests that Proposition~\ref{prop:independent_estimation}
reflects reality more accurately than arguments based on mixing.
In follow up work we have begun to reconcile this issue for stable linear systems \cite{simchowitz18}.

\paragraph{Robust controller design.}

For end-to-end guarantees, parameter estimation is only half the picture.
Our procedure provides us with a family of system models
described by a nominal estimate and a set of unknown but bounded model errors.
It is therefore necessary to ensure that the computed controller has stability
and performance guarantees for any such admissible realization.  The problem of
robustly stabilizing such a family of systems is one with a rich history in the controls community. When modelling errors to the nominal system are allowed to be arbitrary
norm-bounded linear time-invariant (LTI) operators in feedback with the nominal plant, traditional
small-gain theorems and robust synthesis techniques can be applied to exactly
solve the problem \cite{dahleh1994control,ZDGBook}.  However, when the
errors are known to have more structure there are more sophisticated techniques based on structured
singular values and corresponding $\mu$-synthesis techniques
\cite{doyle-ssv,mu-mixed,mu-complex,mu-real} or integral quadratic constraints
(IQCs) \cite{megretski1997system}.  While theoretically
appealing and much less conservative than traditional small-gain approaches,
the resulting synthesis methods are both computationally intractable (although
effective heuristics do exist) and difficult to interpret analytically.  In
particular, we know of no results in the literature that bound the degradation in performance of controlling an uncertain system in terms of the size of the perturbations affecting it.

To circumvent this issue, we leverage a novel parameterization of robustly
stabilizing controllers based on the SLS framework for
controller synthesis \cite{SysLevelSyn1}.
We describe this framework in more detail in Section~\ref{sec:robust}.
Originally developed to allow for scaling optimal and robust controller synthesis
techniques to large-scale systems, the SLS framework can be
viewed as a generalization of the celebrated Youla parameterization
\cite{youla1976modern}. We show that SLS allows us to account for model uncertainty in a 
transparent and analytically tractable way.

\paragraph{PAC learning and reinforcement learning.}

Concerning end-to-end guarantees for LQR which couple estimation and control
synthesis, our work is most comparable to that of
Fiechter~\cite{fiechter1997pac}, who shows that the \emph{discounted} LQR
problem is PAC-learnable.  Fietcher analyzes an identify-then-control scheme
similar to the one we propose, but there are several key differences. First,
our probabilistic bounds on identification are much sharper, by leveraging
modern tools from high-dimensional statistics.
Second, Fiechter implicitly assumes that the true closed-loop system with the
estimated controller is not only stable but also contractive.  While this very
strong assumption is nearly impossible to verify in practice, contractive
closed-loop assumptions are actually pervasive throughout the literature, as we
describe below.  To the best of our knowledge, our work is the first to properly
lift this technical restriction.
Third, and most importantly, Fietcher proposes to directly solve the discounted
LQR problem with the identified model, and does not take into account any
uncertainty in the controller synthesis step. 
This is problematic for two reasons.
First, it is easy to construct an instance of a discounted LQR problem where
the \emph{optimal} solution does not stabilize the true system
(see e.g.~\cite{postoyan17}).
Therefore, even in the limit of infinite data, there is no guarantee
that the closed-loop system will be stable.
Second, even if the optimal solution does stabilize the underlying system, failing
to take uncertainty into account can lead to situations where the synthesized
controller does not. We will demonstrate this behavior in our experiments.

We are also particularly interested in the LQR problem as a baseline for more
complicated problems in reinforcement learning (RL).  LQR should be a
relatively easy problem in RL because on can learn the dynamics from anywhere
in the state space, vastly simplifying the problem of exploration. Hence, it is important to establish how well a pure exploration followed by exploitation strategy can fare on this simple baseline.

There are indeed some related efforts in RL and online learning.
Abbasi-Yadkori and Szepesvari~\cite{abbasi2011regret} propose to use the optimism in
the face of uncertainty (OFU) principle for the LQR problem, by maintaining
confidence ellipsoids on the true parameter, and using the controller which, in
feedback, minimizes the cost objective the most among all systems in the
confidence ellipsoid.  Ignoring the computational intractability of this
approach, their analysis reveals an exponential dependence in the system order
in their regret bound, and also makes the very strong assumption that
the optimal closed-loop systems are contractive
for every $A,B$ in the confidence ellipsoid.
The regret bound is improved by Ibrahimi et al.~\cite{ibrahimi12} to depend
linearly on the state dimension under additional sparsity constraints on the
dynamics.

In response to the computational intractability of the OFU principle,
researchers in RL and online learning have proposed the use of Thompson
sampling~\cite{russo17} for exploration. Abeille and Lazaric~\cite{abeille17}
show that the regret of a Thompson sampling approach for LQR scales as
$\widetilde{\mathcal{O}}(T^{2/3})$ and improve the result to
$\widetilde{\mathcal{O}}(\sqrt{T})$ in \cite{abeille18}, where $\widetilde{\mathcal{O}}(\cdot)$
hides poly-logarithmic factors.  However, their results are only valid for the
scalar $n=d=1$ setting.  Ouyang et al.~\cite{ouyang17} show that in a \emph{Bayesian}
setting, the expected regret can be bounded by
$\widetilde{\mathcal{O}}(\sqrt{T})$.
While this matches the bound of \cite{abbasi2011regret}, the Bayesian regret is with respect
to a particular Gaussian prior distribution over the true model, which differs
from the frequentist setting considered in \cite{abbasi2011regret,abeille17,abeille18}.
Furthermore, these works also make the same restrictive assumption that the
optimal closed-loop systems are uniformly contractive over some known set.

Jiang et al.~\cite{jiang17} propose a general exploration algorithm for
contextual decision processes (CDPs) and show that CDPs with low \emph{Bellman
rank} are PAC-learnable; in the LQR setting, they show the Bellman rank is
bounded by $\statedim^2$. While this result is appealing from an
information-theoretic standpoint, the proposed algorithm is computationally intractable for continuous problems.
Hazan et al.~\cite{hazan17,hazan18} study the problem of prediction in a linear dynamical system
via a novel spectral filtering algorithm. Their main result shows that one can compete in a regret
setting in terms of prediction error. As mentioned previously, converting prediction error bounds into
concrete bounds on sub-optimality of control performance is an open question.
Fazel et al.~\cite{Fazel18} show that randomized search algorithms similar to policy gradient
can learn the optimal controller with a polynomial number of samples in the noiseless case;
an explicit characterization of the dependence of the sample complexity on the parameters
of the true system is not given.


\section{System Identification through Least-Squares} 
\label{sec:estimation}
To estimate a coarse model of the unknown system dynamics, we turn to the
simple and classical method of linear least squares. By running experiments
in which the system starts at $x_0=0$ and the dynamics evolve with a given input, we can record the resulting state observations. The set of inputs
and outputs from each such experiment will be called a {\it rollout}.
For system estimation, we excite the system with Gaussian noise for $N$ rollouts, each
of length $T$. The resulting
dataset is $\{(\vecx_t^{(\ell)}, \vecu_t^{(\ell)}) ~:~ 1 \leq \ell \leq N, 0 \leq
t \leq T\}$, where $t$ indexes the time in one rollout and $\ell$ indexes
independent rollouts. Therefore, we can estimate the system dynamics by
\begin{align}
\label{eq:ls_problem}
(\Ahat, \Bhat) \in \arg \min_{(A,B)} \sum_{\ell = 1}^N \sum_{t = 0}^{T - 1} \frac{1}{2} \ltwonorm{A\vecx_t^{(\ell)} + B\vecu_t^{(\ell)} - \vecx_{t + 1}^{(\ell)}}^2.
\end{align}

For the Coarse-ID control setting, a good estimate of error is just as important as the estimate of the dynamics. Statistical theory and tools allow us to quantify the error of the least squares estimator. First, we present a theoretical analysis of the error in a simplified setting. Then, we describe a computational bootstrap procedure for error estimation from data alone.

\subsection{Least Squares Estimation as a Random Matrix Problem}
We begin by explicitly writing the form of the least squares estimator. First, fixing notation to simplify the presentation, let $\Theta := \begin{bmatrix} A & B \end{bmatrix}^* \in \R^{(\statedim + \inputdim) \times \statedim}$
and let $z_t := \begin{bmatrix} x_t \\ u_t \end{bmatrix} \in \R^{\statedim + \inputdim}$.
Then the system dynamics can be rewritten, for all $t \geq 0$,
\begin{align*}
  x_{t+1}^\T = z_t^\T \Theta + w_t^\T \:.
\end{align*}
Then in a single rollout, we will collect
\begin{align}\label{eq:single_trial}
  X := \begin{bmatrix} x_1^\T \\ x_2^\T \\ \vdots \\ x_T^\T \end{bmatrix} \:, \:\:
  Z := \begin{bmatrix} z_0^\T \\ z_1^\T \\ \vdots \\ z_{T-1}^\T \end{bmatrix} \:, \:\:
  W := \begin{bmatrix} w_0^\T \\ w_1^\T \\ \vdots \\ w_{T-1}^\T \end{bmatrix} \:.
\end{align}
The system dynamics give the identity $ X = Z \Theta + W $.
Resetting state of the system to $x_0=0$ each time,
we can perform $N$ rollouts and collect $N$ datasets like \eqref{eq:single_trial}. Having the ability to reset the system to a state independent of past observations will be important for the analysis in the following section, and it is also practically important for potentially unstable systems.
Denote the data for each rollout as $(X^{(\ell)}, Z^{(\ell)}, W^{(\ell)})$.
With slight abuse of notation, let
$X_N$ be composed of vertically stacked $X^{(\ell)}$, and similarly for
$Z_N$ and $W_N$. Then we have
\begin{align*}
  X_N = Z_N \Theta + W_N \:.
\end{align*}
The full data least squares estimator for $\Theta$ is (assuming for now invertibility of $Z_N^\T Z_N$),
\begin{align}\label{eq:LS_est}
  \widehat{\Theta} = (Z_N^\T Z_N)^{-1} Z_N^\T X_N = \Theta + (Z_N^\T Z_N)^{-1} Z_N^\T W_N \:.
\end{align}
Then the estimation error is given by
\begin{align}\label{eq:est_error}
   E := \widehat{\Theta} - \Theta = (Z_N^\T Z_N)^{-1} Z_N^\T W_N \:.
\end{align}
The magnitude of this error is the quantity of interest in determining confidence sets around estimates $(\widehat A,
\widehat B)$. However, since $W_N$ and $Z_N$ are not independent, this
estimator is difficult to analyze using standard methods. While this type of
analysis is an open problem of interest, in this paper we turn instead to a simplified estimator.

\subsection{Theoretical Bounds on Least Squares Error}
\label{sec:th_bounds}
In this section, we work out the statistical rate for the least squares estimator which uses just the last sample of each trajectory $(\vecx_T^{(\ell)}, \vecx_{T-1}^{(\ell)}, \vecu_{T-1}^{(\ell)})$. This estimation procedure is made precise in Algorithm~\ref{alg:independent_estimation}.
Our analysis ideas are analogous to those used to prove statistical rates for standard linear regression, and they leverage recent tools in nonasymptotic analysis of random matrices. The result is presented above in Proposition~\ref{prop:independent_estimation}.
\begin{center}
\begin{algorithm}[h!]
   \caption{Estimation of linear dynamics with independent data}
\begin{algorithmic}[1]
\FOR{$\ell$ from $1$ to $N$}
\STATE $\vecx_0^{(\ell)} = 0$
\FOR{$t$ from $0$ to $T - 1$}
\STATE $\vecx_{t + 1}^{(\ell)} = \trueA \vecx_t^{(\ell)} + \trueB \vecu_t^{(\ell)} + \vecw_t^{(\ell)}$ with $\vecw_t^{(\ell)} \iid \Ncal(0, \sigma_w^2 I_\statedim)$ and $\vecu_t^{(\ell)} \iid \Ncal(0, \sigma_u^2 I_\inputdim)$.
\ENDFOR
\ENDFOR
\STATE $(\Ahat, \Bhat) \in \arg\min_{(A,B)} \sum_{\ell = 1}^N \frac{1}{2} \ltwonorm{ A\vecx_{T - 1}^{(\ell)} + B\vecu_{T - 1}^{(\ell)} - \vecx_{T}^{(\ell)}}^2$
\end{algorithmic}
   \label{alg:independent_estimation}
 \end{algorithm}
\end{center}
In the context of Proposition~\ref{prop:independent_estimation}, a single data point from each $T$-step rollout is used.  We emphasize that this strategy results in independent data, which can be seen by defining the estimator matrix directly. The previous estimator \eqref{eq:LS_est} is amended as follows: the matrices defined in (\ref{eq:single_trial}) instead include only the final timestep of each trial,
$X_N = \begin{bmatrix} x_T^{(1)} & x_T^{(2)} & \hdots & x_T^{(N)} \end{bmatrix}^\T $, and similar modifications are made to $Z_N$ and $W_N$. The estimator~\eqref{eq:LS_est} uses these modified matrices, which now contain independent rows. To see this, recall the definition of $G_T$ and $F_T$ from \eqref{eq:gramians},
\begin{align*}
  G_T = \begin{bmatrix} A^{T-1} B & A^{T-2} B & ... & B \end{bmatrix} \:, \:\:
  F_T = \begin{bmatrix} A^{T-1} & A^{T-2} & ... & I_\statedim \end{bmatrix} \:.
\end{align*}
We can unroll the system dynamics and see that
\begin{align} \label{eq:unrolled_state_eq}
  x_T = G_T \begin{bmatrix} u_0 \\ u_1 \\ \vdots \\ u_{T-1} \end{bmatrix} + F_T \begin{bmatrix} w_0 \\ w_1 \\ \vdots \\ w_{T-1}\end{bmatrix} \:.
\end{align}
Using Gaussian excitation, $u_t\sim \Ncal(0,\sigma^2_u I_\inputdim)$ gives
\begin{align}\label{eq:row_distribution}
  \begin{bmatrix} x_T \\ u_T \end{bmatrix} \sim \Ncal\left(0,  \begin{bmatrix} \sigma^2_u G_TG_T^\T + \sigma^2_w F_TF_T^\T & 0 \\ 0 & \sigma^2_u I_\inputdim \end{bmatrix}  \right) \:.
\end{align}
Since $F_TF_T^\T \succ 0$, as long as both $\sigma_u,\sigma_w$ are positive, this is a non-degenerate distribution.

Therefore, bounding the estimation error can be achieved via proving a result on the error in random design linear regression with vector valued observations.
First, we present a lemma which bounds the spectral norm of the product of two independent Gaussian matrices.
\begin{lemma}
\label{lemma:product_gaussians}
Fix a $\delta \in (0, 1)$ and $N \geq 2 \log(1/\delta)$.
Let $f_k \in \R^m$, $g_k \in \R^n$ be independent random vectors
$f_k \sim \Ncal(0, \Sigma_f)$ and $g_k \sim \Ncal(0, \Sigma_g)$
for $1 \leq k \leq N$. With probability at least $1-\delta$,
\begin{align*}
  \bigspectralnorm{ \sum_{k=1}^{N} f_k g_k^\T } \leq 4 \spectralnorm{\Sigma_f}^{1/2}\spectralnorm{\Sigma_g}^{1/2} \sqrt{N(m+n)\log(9/\delta)} \:.
\end{align*}
\end{lemma}

We believe this bound to be standard, and include a proof in the appendix for completeness.
Lemma~\ref{lemma:product_gaussians} shows that if $X$ is $n_1 \times N$ with i.i.d. $\Ncal(0, 1)$ entries
and $Y$ is $N \times n_2$ with i.i.d. $\Ncal(0, 1)$ entries, and $X$ and $Y$ are independent,
then with probability at least $1-\delta$ we have
\begin{align*}
  \spectralnorm{ X Y } \leq 4 \sqrt{N(n_1 + n_2) \log(9/\delta)} \:.
\end{align*}
Next, we state a standard nonasymptotic bound on
the minimum singular value of a standard Wishart matrix
(see e.g. Corollary 5.35 of~\cite{vershynin10}).
\begin{lemma}
\label{lemma:wishart_lower_bound}
Let $X \in \R^{N \times n}$ have i.i.d. $\Ncal(0, 1)$ entries.
With probability at least $1-\delta$,
  \begin{align*}
    \sqrt{\lambda_{\min}(X^\T X)} \geq \sqrt{N} - \sqrt{n} - \sqrt{2\log(1/\delta)} \:.
  \end{align*}
\end{lemma}
We combine the previous lemmas into a statement on the error of random design regression.
\begin{lemma}
\label{lemma:random_design_bound}
Let $z_1, ..., z_N \in \R^{n}$ be i.i.d. from $\Ncal(0, \Sigma)$ with $\Sigma$ invertible.
Let $Z^\T := \begin{bmatrix} z_1 & ... & z_N \end{bmatrix}$.
Let $W \in \R^{N \times p}$ with each entry i.i.d. $\Ncal(0, \sigma_w^2)$ and independent of $Z$.
Let $E := (Z^\T Z)^{\dag} Z^\T W$, and
suppose that
\begin{align}
  N \geq 8n + 16 \log(2/\delta) \label{eq:N_assumption} \:.
\end{align}
For any fixed matrix $Q$, we have
with probability at least $1-\delta$,
\begin{align*}
  \spectralnorm{QE} \leq 16 \sigma_w \spectralnorm{Q \Sigma^{-1/2}} \sqrt{\frac{(n+p)\log(18/\delta)}{N}} \:.
\end{align*}
\end{lemma}
\begin{proof}
First, observe that $Z$ is equal in distribution to
$X \Sigma^{1/2}$, where $X \in \R^{N \times n}$ has i.i.d. $\Ncal(0, 1)$ entries.
By Lemma~\ref{lemma:wishart_lower_bound}, with probability at least $1-\delta/2$,
\begin{align*}
  \sqrt{\lambda_{\min}(X^\T X)} \geq \sqrt{N} - \sqrt{n} - \sqrt{2\log(2/\delta)} \geq \sqrt{N}/2 \:.
\end{align*}
The last inequality uses \eqref{eq:N_assumption} combined with the inequality
$(a+b)^2 \leq 2(a^2 + b^2)$.
Furthermore, by Lemma~\ref{lemma:product_gaussians} and \eqref{eq:N_assumption},
with probability at least $1-\delta/2$,
\begin{align*}
  \spectralnorm{X^\T W} \leq 4 \sigma_w \sqrt{N(n+p)\log(18/\delta)} \:.
\end{align*}
Let $\mathcal{E}$ denote the event which is the intersection
of the two previous events.
By a union bound, $\Pr(\mathcal{E}) \geq 1-\delta$.
We continue the rest of the proof assuming the event $\mathcal{E}$ holds.
Since $X^\T X$ is invertible,
\begin{align*}
  QE = Q(Z^\T Z)^{\dag} Z^\T W = Q (\Sigma^{1/2} X^\T X \Sigma^{1/2})^{\dag} \Sigma^{1/2} X^\T W = Q \Sigma^{-1/2} (X^\T X)^{-1} X^\T W \:.
\end{align*}
Taking operator norms on both sides,
\begin{align*}
  \spectralnorm{QE} \leq \spectralnorm{Q \Sigma^{-1/2}} \spectralnorm{(X^\T X)^{-1}} \spectralnorm{ X^\T W} = \spectralnorm{Q \Sigma^{-1/2}} \frac{\spectralnorm{X^\T W}}{\lambda_{\min}(X^\T X)} \:.
\end{align*}
Combining the inequalities above,
\begin{align*}
  \frac{\spectralnorm{X^\T W}}{\lambda_{\min}(X^\T X)} \leq 16\sigma_w \sqrt{\frac{(n+p)\log(18/\delta)}{N}} \:.
\end{align*}
The result now follows.
\end{proof}

Using this result on random design linear regression, we are now ready to analyze the estimation errors of the identification in Algorithm~\ref{alg:independent_estimation} and provide a proof of Proposition~\ref{prop:independent_estimation}.

\begin{proof}
Consider the least squares estimation error (\ref{eq:est_error}) with modified single-sample-per-rollout matrices.
Recall that rows of the design matrix $Z_N$ are distributed as independent normals, as in (\ref{eq:row_distribution}). Then applying Lemma~\ref{lemma:random_design_bound} with $Q_A = \begin{bmatrix} I_\statedim &
0 \end{bmatrix}$ so that $Q_A E$ extracts only the estimate for $A$, we conclude
that with probability at least $1 - \delta/2$,
\begin{align}
  \spectralnorm{\widehat{A} - A} \leq \frac{16 \sigma_w}{\sqrt{\lambda_{\min}(\sigma_u^2 G_TG_T^\T + \sigma_w^2 F_TF_T^\T)}}\sqrt{\frac{(\statedim + 2\inputdim)\log(36/\delta)}{N}} \:,
\end{align}
as long as $N \geq 8(\statedim + \inputdim) + 16\log(4/\delta)$.
Now applying Lemma~\ref{lemma:random_design_bound} under the same condition on $N$ with
$Q_B = \begin{bmatrix} 0 & I_\inputdim \end{bmatrix}$, we have
with probability at least $1 - \delta/2$,
\begin{align}
  \spectralnorm{\widehat{B} - B} \leq \frac{16\sigma_w}{\sigma_u}\sqrt{\frac{(\statedim + 2\inputdim)\log(36/\delta)}{N}} \:.
\end{align}
The result follows by application of the union bound.
\end{proof}

There are several interesting points to make about the guarantees offered by Proposition~\ref{prop:independent_estimation}.  First, as mentioned in the introduction, there are $\statedim(\statedim + \inputdim)$ parameters to learn and our bound states that we need  $O(\statedim + \inputdim)$ measurements,  each measurement providing $\statedim$ values.  Hence, this appears to be an optimal dependence with respect to the parameters $\statedim$ and $\inputdim$.   Second, note that intuitively, if the system amplifies the control and noise inputs in all directions of the state-space, as captured by the minimum eigenvalues of the control and disturbance Gramians $G_TG_T^*$ or $F_TF_T^*$, respectively, then the system has a larger ``signal-to-noise'' ratio and the system matrix $\trueA$ is easier to estimate.
On the other hand, this measure of the excitability of the system has no impact on learning $\trueB$. Unlike in Fiechter's work~\cite{fiechter1997pac}, we do not need to assume that $G_T G_T^\T$ is invertible. As long as the process noise is not degenerate, it will excite all modes of the system.

Finally, we note that the Proposition~\ref{prop:independent_estimation} offers a data independent guarantee for the estimation of the parameters $(A,B)$. We can also provide data dependent guarantees, which will be less conservative in practice. The next result shows how we can use the observed states and inputs to obtain more refined confidence sets than the ones offered by Proposition~\ref{prop:independent_estimation}. The proof is deferred to Appendix~\ref{app:data_dependent}.

\begin{proposition}
\label{prop:data_dependent}
Assume we have $N$ independent samples $(y^{(\ell)}, x^{(\ell)}, u^{(\ell)})$ such that
\begin{align*}
y^{(\ell)} = A x^{(\ell)} + B u^{(\ell)} + w^{(\ell)},
\end{align*}
where $w^{(\ell)}$ are i.i.d. $\mathcal{N}(0,\sigma_w^2 I_\statedim)$ and are independent from $x^{(\ell)}$ and $u^{(\ell)}$. Also, let us assume that $N \geq \statedim + \inputdim$. Then, with probability $1 - \delta$, we have

\begin{align*}
\begin{bmatrix}
(\widehat{A} - A)^\top \\
(\widehat{B} - B)^\top
\end{bmatrix}
\begin{bmatrix}
(\widehat{A} - A) & (\widehat{B} - B)
\end{bmatrix}
\preceq C(\statedim, \inputdim, \delta) \left(\sum_{\ell = 1}^N
\begin{bmatrix}
x^{(\ell)}\\
u^{(\ell)}
\end{bmatrix}
\begin{bmatrix}
(x^{(\ell)})^\top & (u^{(\ell)})^\top
\end{bmatrix}
\right)^{-1},
\end{align*}
where $C(\statedim, \inputdim, \delta) = \sigma_w^2 (\sqrt{\statedim + \inputdim} + \sqrt{\statedim} + \sqrt{2 \log(1/\delta)})^{2}$.
If the matrix on the right hand side has zero as an eigenvalue, we define the inverse of that eigenvalue to be infinity.
\end{proposition}

Proposition~\ref{prop:data_dependent} is a general result that does not require the inputs $u^{(\ell)}$ to be normally distributed and
it allows the states $x^{(\ell)}$ to be arbitrary as long as all the samples $(y^{(\ell)}, x^{(\ell)}, u^{(\ell)})$ are independent and the process noise
$w^{(\ell)}$ is normally distributed.  Nonetheless, both Propositions~\ref{prop:independent_estimation} and \ref{prop:data_dependent} require estimating $(A, B)$ from independent samples. In practice, one would collect rollouts from the system, which consist of many dependent measurements. In that case, using all the data is preferable. Since the guarantees offered in this section do not apply in that case, in the next section we study a different procedure for estimating the size of the estimation error.

\subsection{Estimating Model Uncertainty with the Bootstrap}\label{sec:bootstrap}

In the previous sections we offered theoretical guarantees on the performance of the least squares estimation of $\trueA$ and $\trueB$ from independent samples. However, there are two important limitations to using such guarantees in practice to offer upper bounds on $\epsilon_A = \spectralnorm{\trueA - \Ahat}$ and $\epsilon_B = \spectralnorm{\trueB - \Bhat}$. First, using only one sample per system rollout is empirically less efficient than using all available data for estimation. Second, even optimal statistical analyses often do not recover constant factors that match practice. For purposes of robust control, it is important to obtain upper bounds on $\epsilon_A$ and $\epsilon_B$
that are not too conservative.
Thus, we aim to find $\widehat{\epsilon}_A$ and $\widehat{\epsilon}_B$ such that $\epsilon_A \leq \widehat{\epsilon}_A$ and $\epsilon_B \leq \widehat{\epsilon}_B$ with high probability.

We propose a vanilla bootstrap method for estimating $\widehat{\epsilon}_A$ and $\widehat{\epsilon}_B$.
Bootstrap methods have had a profound impact in both theoretical and applied statistics since their introduction~\cite{efron79}. These methods are used to estimate statistical quantities (e.g. confidence intervals) by sampling synthetic data from an empirical distribution determined by the available data. For the problem at hand we propose the procedure described in Algorithm~\ref{alg:bootstrap}.\footnote{We assume that $\sigma_u$ and $\sigma_w$ are known. Otherwise they
can be estimated from data.}

\begin{center}
\begin{algorithm}[h!]
   \caption{Bootstrap estimation of $\epsilon_A$ and $\epsilon_B$}
\begin{algorithmic}[1]
\STATE \textbf{Input:} confidence parameter $\delta$, number of trials $M$, data $\{(\vecx_{t}^{(i)}, \vecu_t^{(i)})\}_{\substack{1 \leq i \leq N\\ 1 \leq t \leq T}}$, and $(\Ahat, \Bhat)$ a minimizer of
$
\sum_{\ell = 1}^N \sum_{t = 0}^{T - 1} \frac{1}{2}\spectralnorm{A\vecx_{t}^{(\ell)} + B \vecu_{t}^{(\ell)} - \vecx_{t + 1}^{(\ell)}}^2.
$
\FOR{$M$ trials}
\FOR{$\ell$ from $1$ to $N$}
\STATE $\widehat{\vecx}_0^{(\ell)} = \vecx_0^{(\ell)}$
\FOR{$t$ from $0$ to $T - 1$}
\STATE $\widehat{\vecx}_{t + 1}^{(\ell)} = \Ahat \widehat{\vecx}_t^{(\ell)} + \Bhat \widehat{\vecu}_t^{(\ell)} + \widehat{\vecw}_t^{(\ell)}$ with $\widehat{\vecw}_t^{(\ell)} \iid \Ncal(0, \sigma_w^2 I_\statedim)$ and $\widehat{\vecu}_t^{(\ell)} \iid \Ncal(0, \sigma_u^2 I_\inputdim)$.
\ENDFOR
\ENDFOR
\STATE $(\widetilde{A}, \widetilde{B}) \in \arg\min_{(A,B)} \sum_{\ell = 1}^N \sum_{t = 0}^{T - 1} \frac{1}{2} \ltwonorm{ A\widehat{\vecx}_{t}^{(\ell)} + B\widehat{\vecu}_{t}^{(\ell)} - \widehat{\vecx}_{t + 1}^{(\ell)}}^2$.
\STATE record $\widetilde{\epsilon}_A = \ltwonorm{\Ahat - \widetilde{A}}$ and $\widetilde{\epsilon}_B = \ltwonorm{\Bhat - \widetilde{B}}$.
\ENDFOR
\STATE \textbf{Output:} $\widehat{\epsilon}_A$ and $\widehat{\epsilon}_B$, the $100(1- \delta)$th percentiles of the $\widetilde{\epsilon}_A$'s and the $\widetilde{\epsilon}_B$'s.
\end{algorithmic}
   \label{alg:bootstrap}
 \end{algorithm}
\end{center}

For $\widehat{\epsilon}_A$ and $\widehat{\epsilon}_B$ estimated by Algorithm~\ref{alg:bootstrap} we intuitively have
\begin{align*}
\PP(\spectralnorm{A - \Ahat} \leq \widehat{\epsilon}_A) \approx 1 - \delta \quad \text{and} \quad \PP(\spectralnorm{B - \Bhat} \leq \widehat{\epsilon}_B) \approx 1 - \delta.
\end{align*}

There are many known guarantees for
the bootstrap, particularly for the parametric version we use. We do not
discuss these results here; for more details see texts by Van Der Vaart and
Wellner~\cite{van1996weak}, Shao and Tu~\cite{shao2012jackknife}, and
Hall~\cite{hall2013bootstrap}. Instead, in
Appendix~\ref{sec:bootstrap-experiments} we show empirically the performance of
the bootstrap for our estimation problem. For mission critical systems, where empirical validation is insufficient, the statistical error bounds presented in Section~\ref{sec:th_bounds} give guarantees on the size of $\epsilon_A$, $\epsilon_B$. In general, data dependent error guarantees will be less conservative. In follow up work we offer guarantees similar to the ones presented in Section~\ref{sec:th_bounds} for estimation of linear dynamics from dependent data \cite{simchowitz18}.


\section{Robust Synthesis}
\label{sec:robust}

With estimates of the system $(\Ah,\Bh)$ and operator norm error bounds $(\epsilon_A,\epsilon_B)$ in hand, we now turn to control design. In this section we introduce some useful tools from \emph{System Level Synthesis} (SLS), a recently developed approach to control design that relies on a particular parameterization of signals in a control system \cite{virtual,SysLevelSyn1}. We review the main SLS framework, highlighting the key constructions that we will use to solve the robust LQR problem.  As we show in this and the following section, using the SLS framework, as opposed to traditional techniques from robust control, allows us to (a) compute robust controllers using semidefinite programming, and (b) provide sub-optimality guarantees in terms of the size of the uncertainties on our system estimates.

\subsection{Useful Results from System Level Synthesis}

The SLS framework focuses on the \emph{system responses} of a closed-loop system. As a motivating example, consider linear dynamics under a fixed a static state-feedback control policy $K$, i.e., let $u_k = Kx_k$.  Then, the closed {loop map from} the disturbance process $\{w_0, w_1, \dots\}$ to the state $x_k$ and control input $u_k$ at time $k$ is given by
\begin{equation}
\begin{array}{rcl}
x_k &=& \sum_{t=1}^{k} (\trueA + \trueB K)^{k-t}w_{t-1} \:, \\
u_k &=& \sum_{t=1}^k K(\trueA + \trueB K)^{k-t}w_{t-1} \:.
\end{array}
\label{eq:impulse-response}
\end{equation}
Letting $\Phi_x(k) := (\trueA + \trueB K)^{k-1}$ and $\Phi_u(k) := K(\trueA + \trueB K)^{k-1}$, we can rewrite {Eq.~\eqref{eq:impulse-response}} as
\begin{equation}
\begin{bmatrix} x_k \\ u_k \end{bmatrix} =
\sum_{t=1}^k \begin{bmatrix}\Phi_x(k-t+1) \\ \Phi_u(k-t+1) \end{bmatrix}w_{t-1} \:,
\label{eq:phis}
\end{equation}
where $\{\Phi_x(k),\Phi_u(k)\}$ are called the \emph{closed-loop system response elements} induced by the static controller $K$.

Note that even when the control is a linear function of the state and its past history (i.e. a linear dynamic controller), the expression~\eqref{eq:phis} is valid. Though we conventionally think of the control policy as a function mapping states to input, whenever such a mapping is linear, both the control input and the state can be written as linear functions of the disturbance signal $w_t$. With such an identification, the dynamics require that the $\{\Phi_x(k),\Phi_u(k)\}$ must obey the constraints
\begin{equation}
\Phi_x(k+1) = \trueA \Phi_x(k) + \trueB \Phi_u(k) \:, \:\: \Phi_x(1) = I \:, \:\: \forall k \geq 1 \:,
\label{eq:time-achievability}
\end{equation}
As we describe in more detail below in Theorem \ref{thm:param}, these constraints are in fact both necessary and sufficient. Working with closed-loop system responses allows us to cast optimal control problems as optimization problems over elements $\{\Phi_x(k),\Phi_u(k)\}$, constrained to satisfy the affine equations~\eqref{eq:time-achievability}. Comparing equations \eqref{eq:impulse-response} and \eqref{eq:phis}, we see that the former is non-convex in the controller $K$, whereas the latter is affine in the elements $\{\Phi_x(k),\Phi_u(k)\}$.

As we work with infinite horizon problems, it is notationally more convenient to work with \emph{transfer function} representations of the above objects, which can be obtained by taking a $z$-transform of their time-domain representations. The frequency domain variable $z$ can be informally thought of as the time-shift operator, i.e., $z\{x_k,x_{k+1},\dots\} = \{x_{k+1},x_{k+2},\dots\}$, allowing for a compact representation of LTI dynamics. {We use boldface letters to denote such transfer functions signals in the frequency domain, e.g., $\tf{\Phi}_x(z) = \sum_{k = 1}^\infty\Phi_x(k) z^{-k}$. Then, the constraints \eqref{eq:time-achievability} can be rewritten as }
\begin{equation*}
\begin{bmatrix} zI - \trueA & - \trueB \end{bmatrix} \begin{bmatrix} \tf \Phi_x \\ \tf \Phi_u \end{bmatrix} = I \:,
\end{equation*}
{and the corresponding (not necessarily static) control law $\tf u = \tf K \tf x$ is given by $\tf K = \tf \Phi_u \tf \Phi^{-1}_x$.}
The relevant frequency domain connections for LQR are illustrated in Appendix~\ref{app:h2}.

We formalize our discussion by introducing notation that is common in the controls
literature. For a thorough introduction to the functional analysis
commonly used in control theory, see Chapters 2
and 3 of~\citet{ZDGBook}.
Let $\mathbb{T}$ (resp. $\mathbb{D}$) denote the unit circle (resp. open unit
disk) in the complex plane.
The restriction of the Hardy spaces $\hinf(\mathbb{T})$ and $\htwo(\mathbb{T})$ to matrix-valued
real-rational functions that are analytic on the complement of $\mathbb{D}$
will be referred to as $\mathcal{RH}_\infty$ and $\mathcal{RH}_2$,
respectively.
In controls parlance, this corresponds to (discrete-time) stable
matrix-valued transfer functions.
For these two function spaces, the $\hinf$ and $\htwo$ norms simplify to
\begin{align}
  \norm{\tf G}_{\hinf} = \sup_{z \in \mathbb{T}} \: \norm{G(z)}_2 \:, \:\:
  \norm{\tf G}_{\htwo} = \sqrt{ \frac{1}{2\pi} \int_{\mathbb{T}} \norm{G(z)}_F^2 \; dz } \:.
\end{align}
Finally, the notation $\frac{1}{z} \mathcal{RH}_\infty$ refers to the set
of transfer functions $\tf G$ such that $z \tf G \in \mathcal{RH}_\infty$.
Equivalently, $\tf G \in \frac{1}{z} \mathcal{RH}_\infty$
if $\tf G \in \mathcal{RH}_\infty$ and $\tf G$ is strictly proper.

The most important transfer function for the LQR problem is the map
from the state sequence to the control actions: the control policy.
Consider an arbitrary transfer function $\tf K$ denoting the map from state to control action, $\tf u = \tf K \tf x$. Then the closed-loop transfer matrices from the process noise $\tf w$ to the state $\tf x$ and control action $\tf u$ satisfy

\begin{equation}
\begin{bmatrix} \tf x \\ \tf u \end{bmatrix} = \begin{bmatrix} (zI - A-B\tf K)^{-1} \\ \tf K (zI-A-B \tf K)^{-1} \end{bmatrix} \tf w.
\label{eq:response}
\end{equation}
We then have the following theorem parameterizing the set of stable closed-loop transfer matrices, as described in equation \eqref{eq:response}, that are achievable by a given stabilizing controller $\tf K$.
\begin{theorem}[State-Feedback Parameterization~\cite{SysLevelSyn1}]
The following are true:
\begin{itemize}
\item The affine subspace defined by
\begin{equation}
\begin{bmatrix} zI - A & - B \end{bmatrix} \begin{bmatrix} \tf \Phi_x \\ \tf \Phi_u \end{bmatrix} = I, \ \tf \Phi_x, \tf \Phi_u \in \frac{1}{z}\mathcal{RH}_\infty
\label{eq:achievable}
\end{equation}
parameterizes all system responses \eqref{eq:response} from $\tf w$ to $(\tf x, \tf u)$, achievable by an internally stabilizing state-feedback controller $\tf K$.
\item For any transfer matrices $\{\tf \Phi_x, \tf \Phi_u\}$ satisfying \eqref{eq:achievable}, the controller $\tf K = \tf \Phi_u \tf \Phi_x^{-1}$ is internally stabilizing and achieves the desired system response \eqref{eq:response}.
\end{itemize}
\label{thm:param}
\end{theorem}

Note that in particular, $\{\tf \Phi_x, \tf \Phi_u\}=\{ (zI - A-B\tf K)^{-1} , \tf K (zI-A-B \tf K)^{-1} \}$ as in~\eqref{eq:response} are elements of the affine space defined by~\eqref{eq:achievable} whenever $\tf K$ is a causal stabilizing controller.

We will also make extensive use of a robust variant of Theorem \ref{thm:param}.
\begin{theorem}[Robust Stability~\cite{virtual}]
Suppose that the transfer matrices $\{\tf\Phi_x, \tf \Phi_u\} \in \frac{1}{z}\mathcal{RH}_\infty$ satisfy
\begin{equation}
\begin{bmatrix} zI - A & - B \end{bmatrix} \begin{bmatrix} \tf \Phi_x \\ \tf \Phi_u \end{bmatrix} = I + \tf \Delta.
\end{equation}
Then the controller $\tf K = \tf \Phi_u \tf \Phi_x^{-1}$ stabilizes the system described by $(A,B)$ if and only if $(I+\tf \Delta)^{-1} \in \mathcal{RH}_\infty$.  Furthermore, the resulting system response is given by
\begin{equation}
\begin{bmatrix} \tf x \\ \tf u \end{bmatrix} = \begin{bmatrix} \tf \Phi_x \\ \tf \Phi_u \end{bmatrix}(I+\tf \Delta)^{-1} \tf w.
\end{equation}
\label{thm:robust}
\end{theorem}
\begin{coro}
Under the assumptions of Theorem \ref{thm:robust}, if $\|\tf \Delta \| <1$ for any induced norm $\|\cdot \|$, then the controller $\tf K = \tf \Phi_u \tf \Phi_x^{-1}$ stabilizes the system described by $(A,B)$.
\label{coro:sufficient}
\end{coro}
\begin{proof}
Follows immediately from the small gain theorem, see for example Section 9.2 in \cite{ZDGBook}.
\end{proof}

\subsection{Robust LQR Synthesis}
We return to the problem setting where estimates $(\Ah, \Bh)$ of a true system $(A,B)$ satisfy
\[\|\Delta_A\|_2\leq\epsilon_A,~~\|\Delta_B\|_2\leq\epsilon_B\]
where $\Delta_A := \Ah-A$ and $\Delta_B := \Bh-B$ and where we wish to minimize the LQR cost for the worst instantiation of the parametric uncertainty.

Before proceeding, we must formulate the LQR problem in terms of the system responses $\{\Phi_x(k),\Phi_u(k)\}$. It follows from Theorem \ref{thm:param} and the standard equivalence between infinite horizon LQR and $\htwo$ optimal control that, for a disturbance process distributed as $w_t \overset{i.i.d.}{\sim{}} \mathcal{N}(0,\sigma_w^2 I)$, the standard LQR problem \eqref{eq:lqr-classic} can be equivalently written as
\begin{equation}
\min_{\tf\Phi_x, \tf\Phi_u} \sigma_w^2 \left\|\begin{bmatrix} Q^\frac{1}{2} & 0 \\ 0 & R^\frac{1}{2}\end{bmatrix}\begin{bmatrix} \tf \Phi_x \\ \tf \Phi_u \end{bmatrix}\right\|_{\htwo}^2 \text{ s.t. equation \eqref{eq:achievable}}.
\label{eq:lqr2}
\end{equation}
We provide a full derivation of this equivalence in Appendix~\ref{app:h2}.  Going forward, we drop the $\sigma_w^2$ multiplier in the objective function as it affects neither the optimal controller nor the sub-optimality guarantees that we compute in Section \ref{sec:subopt}.

We begin with a simple sufficient condition under which any controller $\tf K$ that stabilizes $(\Ah,\Bh)$ also stabilizes the true system $(A,B)$.  To state the lemma, we introduce one additional piece of notation.  For a matrix $M$, we let $\Res{M}$ denote the resolvent
\begin{equation}\label{eq:phi-def}
\Res{M} := (zI - M)^{-1}\,.
\end{equation}
We now can state our robustness lemma.

\begin{lemma}\label{lemma:robust-sls}
Let the controller $\tf K$ stabilize $(\Ah, \Bh)$ and $(\tf\Phi_x,\tf\Phi_u)$ be its corresponding system response \eqref{eq:response} on system $(\Ah,\Bh)$.  Then if $\tf K$ stabilizes $(A,B)$, it achieves the following LQR cost
\begin{equation}
 J(A,B,\tf K) := \left\|\begin{bmatrix} Q^\frac{1}{2} & 0 \\ 0 & R^\frac{1}{2}\end{bmatrix}\begin{bmatrix} \tf\Phi_x \\  \tf\Phi_u \end{bmatrix}\left(I+\begin{bmatrix}\Delta_A& \Delta_B\end{bmatrix}\begin{bmatrix} \tf\Phi_x \\  \tf\Phi_u \end{bmatrix}\right)^{-1}\right\|_{\htwo}\:.
\end{equation}

Furthermore, letting
\begin{equation}\label{eq:deltahat}
\tf{\hat\Delta} := \begin{bmatrix}\Delta_A& \Delta_B\end{bmatrix}\begin{bmatrix} \tf\Phi_x \\  \tf\Phi_u \end{bmatrix} = (\Delta_A + \Delta_B \tf K)\Res{\Ah+\Bh\tf K} \:.
\end{equation}
a sufficient condition for $\tf K$ to stabilize $(A,B)$ is that $\hinfnorm{\tf{\hat{\Delta}}} <1$.
\label{lem:sufficient}
\end{lemma}
\begin{proof}
Follows immediately from Theorems \ref{thm:param}, \ref{thm:robust} and Corollary \ref{coro:sufficient} by noting that for system responses $(\tf \Phi_x, \tf \Phi_u)$ satisfying
\[
\begin{bmatrix} zI - \Ah & - \Bh \end{bmatrix} \begin{bmatrix} \tf \Phi_x \\ \tf \Phi_u \end{bmatrix} = I,
\]
it holds that
\[
\begin{bmatrix} zI - A & - B \end{bmatrix} \begin{bmatrix} \tf \Phi_x \\ \tf \Phi_u \end{bmatrix} = I + \Dh
\]
for $\Dh$ as defined in equation \eqref{eq:deltahat}.
\end{proof}
We can therefore recast the robust LQR problem \eqref{eq:robust_lqr} in the following equivalent form
\begin{align}\label{eq:robustLQR}
\begin{split}
 &\min_{\tf\Phi_x, \tf \Phi_u} \sup\limits_{\substack{\|\Delta_A\|_2\leq \epsilon_A \\ \|\Delta_B\|_2\leq \epsilon_B}}  J(A,B,\tf K)\\
& \text{s.t.} \begin{bmatrix}zI-\Ah&-\Bh\end{bmatrix}\begin{bmatrix} \tf\Phi_x \\  \tf\Phi_u \end{bmatrix} = I,~~\tf\Phi_x, \tf \Phi_u  \in\frac{1}{z}\mathcal{RH}_\infty \:.
 \end{split}
\end{align}
The resulting robust control problem is one subject to real-parametric
uncertainty, a class of problems known to be computationally
intractable~\cite{braatz94}.  Although effective computational heuristics
(e.g., DK iteration \cite{ZDGBook}) exist, the performance of the resulting controller on the
true system is difficult to characterize analytically in terms of the size of
the perturbations.

To circumvent this issue, we take a slightly conservative approach and find an upper-bound to the cost $J(A,B,\tf K)$ that is independent of the uncertainties $\Delta_A$ and $\Delta_B$. First, note that if $\hinfnorm{\Dh} < 1$, we can write
\begin{align} \label{eq:upperbnd1}
J(A,B,\tf K) \leq \hinfnorm{(I+\Dh)^{-1}}J(\Ah,\Bh,\tf K) \leq \frac{1}{1-\hinfnorm{\Dh}}J(\Ah,\Bh,\tf K).
\end{align}

Because $J(\Ah,\Bh,\tf K)$ captures the performance of the controller $\tf K$ on the nominal system $(\Ah,\Bh)$, it is not subject to any uncertainty.  It therefore remains to compute a tractable bound for $\hinfnorm{\Dh}$, which we do using the following fact.
\begin{proposition}
	For any $\alpha \in (0,1)$ and $\tf{\hat{\Delta}}$ as defined in \eqref{eq:deltahat}
\begin{equation}\label{eq:ben-tri-bound}
	\|  \tf{\hat{\Delta}} \|_{\hinf}
	 \leq \left\|\begin{bmatrix} \tfrac{\epsilon_A}{\sqrt{\alpha}} \tf \Phi_x \\ \tfrac{\epsilon_B}{\sqrt{1-\alpha}} \tf\Phi_u \end{bmatrix} \right\|_{\hinf} =\colon H_\alpha(\tf\Phi_x,\tf\Phi_u) \:.
\end{equation}
\label{prop:bound}
\end{proposition}

\begin{proof}
Note that for any block matrix of the form $\begin{bmatrix} M_1 & M_2 \end{bmatrix}$, we have
\begin{equation}\label{eq:block-norm-bound}
	\left\|\begin{bmatrix} M_1 & M_2 \end{bmatrix}\right\|_2
\leq	\left(\left\| M_1 \right\|_2^2 + \left\| M_2 \right\|_2^2\right)^{1/2}\,.
\end{equation}
To verify this assertion, note that
\[
\left\|\begin{bmatrix} M_1 & M_2 \end{bmatrix}\right\|_2^2
= \lambda_{\mathrm{max}}(M_1 M_1^*+ M_2 M_2^*)
\leq \lambda_{\mathrm{max}}(M_1 M_1^*)+  \lambda_{\mathrm{max}}(M_2 M_2^*)
= \left\| M_1 \right\|_2^2 + \left\| M_2 \right\|_2^2\,.
\]

With~\eqref{eq:block-norm-bound} in hand, we have
\begin{align*} \left\|  \begin{bmatrix} \Delta_A & \Delta_B \end{bmatrix} \begin{bmatrix} \tf \Phi_x \\ \tf \Phi_u \end{bmatrix} \right\|_{\hinf}
&=\left\| \begin{bmatrix} \frac{\sqrt{\alpha}}{\epsilon_A}\Delta_A & \frac{\sqrt{1-\alpha}}{\epsilon_B}\Delta_B\ \end{bmatrix} \begin{bmatrix} \frac{\epsilon_A}{\sqrt{\alpha}} \tf\Phi_x \\ \frac{\epsilon_B}{\sqrt{1-\alpha}}\tf\Phi_u \end{bmatrix} \right\|_{\hinf} \\
&\leq \left\|  \begin{bmatrix} \frac{\sqrt{\alpha}}{\epsilon_A}\Delta_A & \frac{\sqrt{1-\alpha}}{\epsilon_B}\Delta_B\ \end{bmatrix}\right\|_2
\left\| \begin{bmatrix} \frac{\epsilon_A}{\sqrt{\alpha}} \tf\Phi_x \\ \frac{\epsilon_B}{\sqrt{1-\alpha}}\tf\Phi_u \end{bmatrix} \right\|_{\hinf}\leq
 \left\| \begin{bmatrix} \frac{\epsilon_A}{\sqrt{\alpha}} \tf\Phi_x \\ \frac{\epsilon_B}{\sqrt{1-\alpha}}\tf\Phi_u \end{bmatrix} \right\|_{\hinf} ,
\end{align*}
completing the proof.
 \end{proof}

The following corollary is then immediate.

\begin{coro}
Let the controller $\tf K$ and resulting system response $(\tf \Phi_x, \tf \Phi_u)$ be as defined in Lemma \ref{lem:sufficient}.  Then if $H_\alpha(\tf \Phi_x, \tf \Phi_u) < 1$, the controller $\tf K = \tf\Phi_u \tf\Phi_x^{-1}$ stabilizes the true system $(A,B)$.
\label{coro:stable}
\end{coro}

Applying Proposition \ref{prop:bound} in conjunction with the bound \eqref{eq:upperbnd1}, we arrive at the following upper bound to the cost function of the robust LQR problem \eqref{eq:robust_lqr}, which is independent of the perturbations $(\Delta_A,\Delta_B)$:
\begin{align}\label{eq:upperbnd}
 \sup\limits_{\substack{\|\Delta_A\|_2\leq \epsilon_A \\ \|\Delta_B\|_2\leq \epsilon_B}}  J(A,B,\tf K)  &\leq \left\|\begin{bmatrix} Q^\frac{1}{2} & 0 \\ 0 & R^\frac{1}{2}\end{bmatrix}\begin{bmatrix} \tf\Phi_x \\  \tf\Phi_u \end{bmatrix}\right\|_{\htwo}\frac{1}{1 - H_\alpha(\tf\Phi_x,\tf\Phi_u)} = \frac{J(\Ah,\Bh,\tf K)}{1 - H_\alpha(\tf\Phi_x,\tf\Phi_u)}\:.
\end{align}
The upper bound is only valid when $H_\alpha(\tf \Phi_x, \tf \Phi_u) < 1$, which guarantees the stability of the closed-loop system as in Corollary~\ref{coro:stable}.
We remark that Corollary~\ref{coro:stable} and the bound in \eqref{eq:upperbnd} are of interest independent of the synthesis procedure for $\tf K$. In particular, they can be applied to the optimal LQR controller $\Kh$ computed using the nominal system $(\Ah,\Bh)$.

As the next lemma shows, the right hand side of Equation~\eqref{eq:upperbnd} can be efficiently optimized by an appropriate decomposition.  The proof of the lemma is immediate.

\begin{lemma} \label{lem:innerouter}
For functions $f:\mathcal{X}\to\mathbb{R}$ and $g:\mathcal{X}\to\mathbb{R}$ and constraint set $C\subseteq \mathcal{X}$, consider
\begin{align*}
\min_{x\in C} \frac{f(x)}{1-g(x)} \:.
\end{align*}
Assuming that $f(x) \geq 0$ and $0 \leq g(x) < 1$ for all $x\in C$, this optimization problem can be reformulated as an outer single-variable problem and an inner constrained optimization problem (the objective value of an optimization over the emptyset is defined to be infinity):
\begin{align*}
\min_{x\in C} \frac{f(x)}{1-g(x)} = \min_{\gamma\in[0,1)} \tfrac{1}{1-\gamma} \min_{x\in C} \{ f(x) ~|~ g(x)\leq\gamma\}
\end{align*}
\end{lemma}

Then combining Lemma \ref{lem:innerouter} with the upper bound in \eqref{eq:upperbnd} results in the following optimization problem:
\begin{align}\label{eq:robustLQRbnd}
\begin{split}
 \minimize_{\gamma\in[0,1)}\frac{1}{1 - \gamma}&\min_{\tf\Phi_x, \tf \Phi_u} \left\|\begin{bmatrix} Q^\frac{1}{2} & 0 \\ 0 & R^\frac{1}{2}\end{bmatrix}\begin{bmatrix} \tf\Phi_x \\  \tf\Phi_u \end{bmatrix}\right\|_{\htwo}\\
& \text{s.t.} \begin{bmatrix}zI-\Ah&-\Bh\end{bmatrix}\begin{bmatrix} \tf\Phi_x \\  \tf\Phi_u \end{bmatrix} = I,~~\left\|\begin{bmatrix} \tfrac{\epsilon_A}{\sqrt{\alpha}} \tf \Phi_x \\ \tfrac{\epsilon_B}{\sqrt{1-\alpha}} \tf\Phi_u \end{bmatrix} \right\|_{\hinf}\leq \gamma\\
&\qquad \tf\Phi_x, \tf \Phi_u  \in\frac{1}{z}\mathcal{RH}_\infty.
 \end{split}
\end{align}
We note that this optimization objective is jointly quasi-convex in $(\gamma, \tf \Phi_x, \tf \Phi_u)$. Hence, as a function of $\gamma$ alone the objective is quasi-convex, and furthermore is smooth in the feasible domain. Therefore, the outer optimization with respect to $\gamma$ can effectively be solved with methods like golden section search. We remark that the inner optimization is a convex problem, though an infinite dimensional one.  We show in Section~\ref{sec:computation} that a simple finite impulse response truncation yields a finite dimensional problem with similar guarantees of robustness and performance.

We further remark that because $\gamma \in [0,1)$, any feasible solution $(\tf \Phi_x, \tf \Phi_u)$ to optimization problem \eqref{eq:robustLQRbnd} generates a controller $\tf K = \tf \Phi_u \tf \Phi_x^{-1}$ satisfying the conditions of Corollary \ref{coro:stable}, and hence stabilizes the true system $(A,B)$.  Therefore, even if the solution is approximated, as long as it is feasible, it will be stabilizing.  As we show in the next section, for sufficiently small estimation error bounds $\epsilon_A$ and $\epsilon_B$, we can further bound the sub-optimality of the performance achieved by our robustly stabilizing controller relative to that achieved by the optimal LQR controller $\trueK$.

\section{Sub-optimality Guarantees}
\label{sec:subopt}
We now return to analyzing the Coarse-ID control problem. We upper bound the performance of the
 controller synthesized using the optimization \eqref{eq:robustLQRbnd} in terms of the size of the perturbations $(\Delta_A$, $\Delta_B)$ and a measure of complexity of the LQR problem defined by $A$, $B$, $Q$, and $R$. The following result is one of our main contributions.

\begin{theorem}
\label{thm:lqr_cost}
Let $J_\star$ denote the minimal LQR cost achievable by any controller for the dynamical system with transition matrices $(A,B)$, and let $\trueK$ denote the optimal contoller. Let $(\Ah,\Bh)$ be estimates of the transition matrices such that $\ltwonorm{\Delta_A} \leq \epsilon_A$, $\ltwonorm{\Delta_B} \leq \epsilon_B$. Then, if $\tf K$ is synthesized via \eqref{eq:robustLQRbnd} with $\alpha = 1/2$, the relative error in the LQR cost is
\begin{align} \label{eq:lqr_bound}
\frac{J(\trueA, \trueB, \tf K) - J_\star }{J_\star} \leq 5 (\epsilon_A + \epsilon_B\ltwonorm{\trueK})\hinfnorm{\Res{A+B\trueK}} \:,
\end{align}
as long as $(\epsilon_A + \epsilon_B\ltwonorm{\trueK})\|\Res{A+B\trueK}\|_{\hinf}\leq 1/5$.
\end{theorem}

This result offers a guarantee on the performance of the SLS synthesized
controller regardless of the estimation procedure used to estimate the
transition matrices. Together with our result (Proposition~\ref{prop:independent_estimation})
on system identification from
independent data, Theorem~\ref{thm:lqr_cost} yields a sample complexity upper bound
on the performance of the robust SLS controller $\tf K$ when $(A,B)$ are not
known. We make this guarantee precise in Corollary~\ref{coro:lqr_cost_iid}
below. The rest of the section is dedicated to proving
Theorem~\ref{thm:lqr_cost}.

Recall that $\trueK$ is the optimal LQR static state feedback matrix for the true dynamics $(A,B)$, and let $
\label{eq:delta-opt} \tf \Delta := -\left[\Delta_A + \Delta_B\trueK\right]\Res{A+B\trueK}$. We begin with a technical result.

\begin{lemma} \label{lemma:feasibility}
Define $\zeta := (\epsilon_A + \epsilon_B\ltwonorm{\trueK})\hinfnorm{\Res{A + B\trueK}}$,
and suppose that $\zeta < (1 +\sqrt{2})^{-1}$. Then $(\gamma_0, \tilde{\tf \Phi}_x, \tilde{\tf \Phi}_u)$ is a feasible solution of \eqref{eq:robustLQRbnd} with $\alpha = 1/2$, where
\begin{align}\label{eq:feasible_sol}
\gamma_0 = \frac{\sqrt{2}\zeta}{1 - \zeta}\text{, } \quad \tilde{\tf \Phi}_x = \Res{A + B\trueK} (I + \tf\Delta)^{-1}\text{, }\quad \tilde{\tf \Phi}_u = \trueK \Res{A + B\trueK} (I +\tf \Delta)^{-1}.
\end{align}
\end{lemma}
\begin{proof}
By construction $\tilde{\tf \Phi}_x, \tilde{ \tf \Phi}_u \in \frac{1}{z}\mathcal{RH}_\infty$. Therefore, we are left to check three conditions:
\begin{align}\label{eq:conds_to_check}
 \gamma_0 < 1\text{,} \quad \begin{bmatrix}zI-\Ah&-\Bh\end{bmatrix}\begin{bmatrix} \tilde{\tf\Phi}_x \\  \tilde{\tf\Phi}_u \end{bmatrix} = I\;\text{, and} \quad \bignorm{\begin{bmatrix}\tfrac{\epsilon_A}{\sqrt{\alpha}}
 \tilde{\tf \Phi}_x\\
 \tfrac{\epsilon_B}{\sqrt{1-\alpha}} \tilde{\tf \Phi}_u\end{bmatrix}}_{\hinf} \leq \frac{\sqrt{2}\zeta}{1 - \zeta}.
\end{align}
The first two conditions follow by simple algebraic computations.
Before we check the last condition, note that
$\hinfnorm{{\tf \Delta}} \leq (\epsilon_A + \epsilon_B \ltwonorm{\trueK})\hinfnorm{\Res{A+B\trueK}} = \zeta < 1$.
Now observe that,
\begin{align*}
 \bignorm{\begin{bmatrix}\tfrac{\epsilon_A}{\sqrt{\alpha}}
 \tilde{\tf \Phi}_x\\
 \tfrac{\epsilon_B}{\sqrt{1-\alpha}} \tilde{\tf \Phi}_u\end{bmatrix}}_{\hinf} &= \sqrt{2}\bignorm{\begin{bmatrix}\epsilon_A \Res{A+B\trueK}\\
 \epsilon_B \trueK \Res{A+B\trueK}\end{bmatrix}(I + {\tf\Delta})^{-1}}_{\hinf} \\
& \leq \sqrt{2}\hinfnorm{(I + \tf \Delta)^{-1}}  \bignorm{\begin{bmatrix}\epsilon_A \Res{A+B\trueK}\\
 \epsilon_B \trueK \Res{A+B\trueK}\end{bmatrix}}_{\hinf}\\
 & \leq \frac{\sqrt{2}}{1-\hinfnorm{\tf \Delta}}\bignorm{\begin{bmatrix}\epsilon_A I\\
 \epsilon_B \trueK \end{bmatrix}\Res{A+B\trueK}}_{\hinf} \\
&\leq \frac{\sqrt{2}(\epsilon_A + \epsilon_B \ltwonorm{\trueK})\hinfnorm{\Res{A + B\trueK}}}{1-\hinfnorm{\tf \Delta}} \leq \frac{\sqrt{2}\zeta}{1-\zeta} \:.
\end{align*}
\end{proof}

\begin{proof}[Proof of Theorem~\ref{thm:lqr_cost}]

Let $(\gamma_\star, {\tf \Phi}_x^\star, {\tf \Phi}_u^\star)$ be an optimal solution to problem~\eqref{eq:robustLQRbnd} and let $\tf K = \tf \Phi_u^\star ({\tf\Phi}_x^\star)^{-1}$.
We can then write
\begin{align*}
J(A,B,\tf K)\leq \frac{1}{1-\hinfnorm{\Dh}}J(\Ah,\Bh,\tf K) \leq \frac{1}{1-\gamma_\star} J(\Ah,\Bh,\tf K),
\end{align*}
where the first inequality follows from the bound \eqref{eq:upperbnd1}, and the second follows from the fact that $\hinfnorm{\Dh} \leq \gamma_\star$ due to  Proposition \ref{prop:bound} and the constraint in optimization problem~\eqref{eq:robustLQRbnd}.

From Lemma~\ref{lemma:feasibility} we know that  $(\gamma_0, \tilde{\tf \Phi}_x, \tilde{\tf \Phi}_u)$
defined in equation~\eqref{eq:feasible_sol} is also a feasible solution. Therefore, because $\trueK = \tilde{\tf \Phi}_u \tilde{\tf \Phi}_x^{-1} $, we have by optimality,
\begin{align*}
  \frac{1}{1-\gamma_\star} J(\Ah,\Bh,\tf K) \leq \frac{1}{1-\gamma_0} J(\Ah,\Bh,\trueK) \leq  \frac{J(A,B,\trueK)}{(1 - \gamma_0)(1- \hinfnorm{\tf \Delta})}  = \frac{J_\star}{(1 - \gamma_0)(1- \hinfnorm{\tf \Delta})} \:,
\end{align*}
where the second inequality follows by the argument used to derive \eqref{eq:upperbnd1} with the true and estimated transition matrices switched. Recall that $\hinfnorm{\tf \Delta} \leq \zeta$ and that $\gamma_0 = \sqrt{2}\zeta/(1 + \zeta)$. Therefore
\begin{align*}
\frac{J(A,B,\tf K) - J_\star}{J_\star} \leq \frac{1}{1 - (1 + \sqrt{2})\zeta} - 1 = \frac{(1 + \sqrt{2})\zeta}{1 - (1 + \sqrt{2})\zeta} \leq 5\zeta \:,
\end{align*}
where the last inequality follows because $\zeta < 1/5 < 1/(2 + 2\sqrt{2})$.
The conclusion follows.
\end{proof}

With this suboptimality result in hand, we are now ready to give an end-to-end performance guarantee for our procedure when the independent data estimation scheme is used.

\begin{coro}\label{coro:lqr_cost_iid}
Let $\lambda_G = \lambda_{\min}(\sigma_u^2 G_TG_T^\T + \sigma_w^2 F_TF_T^\T)$,
where $F_T, G_T$ are defined in \eqref{eq:gramians}.
Suppose the independent data estimation procedure
described in Algorithm~\ref{alg:independent_estimation}
is used to produce estimates $(\Ah,\Bh)$ and $\tf K$ is synthesized via \eqref{eq:robustLQRbnd} with $\alpha = 1/2$.  Then there are universal constants $C_0$ and $C_1$ such that the relative error in the LQR cost satisfies
\begin{align} \label{eq:samplecomplexity}
&\frac{J(\trueA, \trueB, \tf K) - J_\star }{J_\star} \leq C_0 \sigma_w \|\Res{\trueA + \trueB\trueK}\|_{\hinf}\left(\frac{1}{\sqrt{\lambda_G}} + \frac{\ltwonorm{\trueK}}{\sigma_u}\right)\sqrt{\frac{(\statedim + \inputdim)\log(1/\delta)}{N}}
\end{align}
with probability $1 - \delta$,
as long as $N \geq C_1(\statedim + \inputdim)\sigma_w^2 \hinfnorm{\Res{A + B\trueK}}^2(1 / \lambda_G +  \ltwonorm{\trueK}^2/\sigma_u^2) \log(1/\delta)$.
\end{coro}
\begin{proof}
Recall from Proposition \ref{prop:independent_estimation} that for the independent data estimation scheme, we have
\begin{align} \label{eq:estimation_rate}
\epsilon_A \leq \frac{16\sigma_w}{\sqrt{\lambda_G}}\sqrt{\frac{(\statedim + 2\inputdim)\log(32/\delta)}{N}}\; \text{, and}\quad \epsilon_B \leq \frac{16\sigma_w}{\sigma_u}\sqrt{\frac{(\statedim + 2\inputdim)\log(32/\delta)}{N}},
\end{align}
with probability $1 - \delta$, as long as $N \geq 8 (\statedim + \inputdim) + 16 \log(4/\delta)$.

To apply Theorem~\ref{thm:lqr_cost} we need $(\epsilon_A + \epsilon_B\ltwonorm{\trueK})\hinfnorm{\Res{A + B\trueK}} < 1/5$, which will hold as long as $N \geq \mathcal{O}\left\{(\statedim + \inputdim)\sigma_w^2 \hinfnorm{\Res{A + B\trueK}}^2(1 / \lambda_G +  \ltwonorm{\trueK}^2/\sigma_u^2) \log(1/\delta)\right\}$. A direct plug in of \eqref{eq:estimation_rate} in \eqref{eq:lqr_bound} yields the conclusion.
\end{proof}

This result fully specifies the complexity term $\mathcal{C}_{\mathrm{LQR}}$ promised in the introduction:
\begin{align*}
\mathcal{C}_{\mathrm{LQR}} := C_0 \sigma_w \left(\frac{1}{\sqrt{\lambda_G}} + \frac{\|\trueK\|_2}{\sigma_u}\right)\hinfnorm{\Res{\trueA + \trueB\trueK}}.
\end{align*}

Note that $\mathcal{C}_{\mathrm{LQR}}$ decreases as the minimum eigenvalue of the sum of the input and noise controllability Gramians increases. This minimum eigenvalue tends to be larger for systems that amplify inputs in all directions of the state-space.  $\mathcal{C}_{\mathrm{LQR}}$ increases as function of the operator norm of the gain matrix $\trueK$ and the $\hinf$ norm of the transfer function from disturbance to state of the closed-loop system.  These two terms tend to be larger for systems that are ``harder to control.'' The dependence on $Q$ and $R$ is implicit in this definition since the optimal control matrix $\trueK$ is defined in terms of these two matrices. Note that when $R$ is large in comparison to $Q$, the norm of the controller $\trueK$ tends to be smaller because large inputs are more costly. However, such a change in the size of the controller could cause an increase in
the $\hinf$ norm of the closed-loop system.
Thus, our upper bound suggests an odd balance.  Stable and highly damped systems are easy to control but hard to estimate, whereas unstable systems are easy to estimate but hard to control. Our theorem suggests that achieving a small relative LQR cost requires for the system to be somewhere in the middle of these two extremes.

Finally, we remark that the above analysis holds more generally when we apply additional constraints to the controller in the synthesis problem~\eqref{eq:robustLQRbnd}.  In this case, the suboptimality bounds presented in Theorem \ref{thm:lqr_cost} and Corrollary \ref{coro:lqr_cost_iid} are true with respect to the minimal cost achievable by the constrained controller with access to the true dynamics. In particular, the bounds hold unchanged if the search is restricted to static controllers, i.e. $u_t = K x_t$. This is true because the optimal controller is static and therefore feasible for the constrained synthesis problem.

\section{Computation}      
\label{sec:computation}

As posed, the main optimization problem~\eqref{eq:robustLQRbnd} is a semi-infinite program, and we are not aware of a way to solve this problem efficiently.  In this section we describe two alternative formulations that provide upper bounds to the optimal value and that can be solved in polynomial time.

\subsection{Finite impulse response approximation}\label{sec:approx-FIR}
An elementary approach to reducing the aforementioned semi-infinite program to a finite dimensional one is to only optimize over the first $L$ elements of the transfer functions $\tf \Phi_x$ and $\tf \Phi_u$, effectively taking a finite impulse response (FIR) approximation. Since these are both stable maps, we expect the effects of such an approximation to be negligible as long as the optimization horizon $L$ is chosen to be sufficiently large -- in what follows, we show that this is indeed the case.

By restricting our optimization to FIR approximations of $\tf \Phi_x$ and $\tf \Phi_u$, we can cast the $\htwo$ cost as a second order cone constraint.  The only difficulty arises in posing the $\hinf$ constraint as a semidefinite program.  Though there are several ways to cast $\hinf$ constraints as linear matrix inequalities, we use the formulation in Theorem 5.8 of Dumitrescu's text to take advantage of the FIR structure in our problem~\cite{dumitrescu2007positive}. We note that using Dumitrescu's formulation, the resulting problem is affine in $\alpha$ when $\gamma$ is fixed, and hence we can solve for the optimal value of $\alpha$.
Then the resulting system response elements can be cast as a dynamic feedback controller using Theorem 2 of~\citet{anderson17}.

\subsubsection{Sub-optimality guarantees}
In this subsection we show that optimizing over FIR approximations incurs only a small degradation in performance relative to the solution to the infinite-horizon problem.  In particular, this degradation in performance decays exponentially in the FIR horizon $L$, where the rate of decay is specified by the decay rate of the spectral elements of the optimal closed loop system response $\Res{\trueA + \trueB\trueK}$.

Before proceeding, we introduce additional concepts and notation needed to formalize guarantees in the FIR setting.  A linear-time-invariant transfer function is stable if and only if it is exponentially stable, i.e., $\tf \Phi = \sum_{t=0}^\infty z^{-t}\Phi(t) \in \mathcal{RH}_\infty$ if and only if there exists positive values $C$ and $\rho \in [0,1)$ such that for every spectral element $\Phi(t)$, $t\geq 0$, it holds that
\begin{equation}
\twonorm{\Phi(t)} \leq C \rho^t.
\label{eq:exp_decay}
\end{equation}
In what follows, we pick $C_\star$ and $\rho_\star$ to be any such constants satisfying $\twonorm{\Res{\trueA + \trueB \trueK}(t)} \leq C_\star \rho_\star^t$ for all $t\geq 0$.

We introduce a version of the optimization problem~\eqref{eq:robustLQR} with a finite number of decision variables:

\begin{align}\label{eq:robustFIRbnd}
\begin{split}
 \minimize_{\gamma\in[0,1)}\frac{1}{1 - \gamma}&\min_{\tf\Phi_x, \tf \Phi_u, V} \left\|\begin{bmatrix} Q^\frac{1}{2} & 0 \\ 0 & R^\frac{1}{2}\end{bmatrix}\begin{bmatrix} \tf\Phi_x \\  \tf\Phi_u \end{bmatrix}\right\|_{\htwo}\\
& \text{s.t.} \begin{bmatrix}zI-\Ah&-\Bh\end{bmatrix}\begin{bmatrix} \tf\Phi_x \\  \tf\Phi_u \end{bmatrix} = I + \frac{1}{z^L}V, \\ &\left\|\begin{bmatrix} \tfrac{\epsilon_A}{\sqrt{\alpha}} \tf \Phi_x \\ \tfrac{\epsilon_B}{\sqrt{1-\alpha}} \tf\Phi_u \end{bmatrix} \right\|_{\hinf} + \twonorm{V} \leq \gamma\\
&\tf\Phi_x = \sum_{t=1}^L \frac{1}{z^t}\Phi_x(t), \, \tf\Phi_u = \sum_{t=1}^L \frac{1}{z^t}\Phi_u(t).
 \end{split}
\end{align}
In this optimization problem we search over finite response transfer functions $\tf \Phi_x$ and $\tf \Phi_u$.
Given a feasible solution $\tf \Phi_x $, $ \tf \Phi_u$ of problem \eqref{eq:robustFIRbnd}, we can implement the controller $\tf K_L = \tf \Phi_u \tf \Phi_x^{-1}$ with an equivalent state-space representation $(A_K, B_K, C_K, D_K)$ using the
response elements $\{ \Phi_x(k) \}_{k=1}^{L}$ and
$\{ \Phi_u(k) \}_{k=1}^{L}$  via Theorem 2 of \cite{anderson17}.

The slack term $V$ accounts for the error introduced by truncating the infinite response transfer functions of problem \eqref{eq:robustLQR}.
Intuitively, if the truncated tail is sufficiently small, then the effects of this approximation should be negligible on performance. The next result formalizes this intuition.

\begin{theorem}\label{thm:FIR_subopt}
Set $\alpha = 1/2$ in \eqref{eq:robustFIRbnd} and let $C_\star > 0$ and $\rho_\star \in [0,1)$ be such that $\twonorm{\Res{(\trueA + \trueB \trueK)}(t)} \leq C_\star \rho_\star^t$ for all $t \geq 0$.  Then, if $\tf K_L$ is synthesized via \eqref{eq:robustFIRbnd}, the relative error in the LQR cost is
\begin{align*}
\frac{J(\trueA, \trueB, \tf K_L) - J_\star}{J_\star} \leq 10 (\epsilon_A + \epsilon_B \twonorm{\trueK}) \hinfnorm{\Res{\trueA + \trueB \trueK}},
\end{align*}
as long as
\begin{align*}
\epsilon_A + \epsilon_B \twonorm{\trueK} \leq  \frac{1 - \rho_\star}{10C_\star} \; \text{ and }\; L \geq \frac{4 \log\left(\frac{C_\star}{(\epsilon_A + \epsilon_B \twonorm{\trueK})\hinfnorm{\Res{\trueA + \trueB \trueK}}} \right)}{1 - \rho_\star}.
\end{align*}
\end{theorem}

The proof of this result, deferred to Appendix \ref{app:FIR}, is conceptually the same as that of the infinite horizon setting. The main difference is that care must be taken to ensure that the approximation horizon $L$ is sufficiently large so as to ensure stability and performance of the resulting controller.  From the theorem statement,  we see that for such an appropriately chosen FIR approximation horizon $L$, our performance bound is the same, up to universal constants, to that achieved by the solution to the infinite horizon problem.  Furthermore, the approximation horizon $L$ only needs to grow logarithmically with respect to one over the estimation rate in order to preserve the same statistical rate as the controller produced by the infinite horizon problem. Finally, an end-to-end sample complexity result analogous to that stated in Corollary \ref{coro:lqr_cost_iid} can be easily obtained by simply substituting in the sample-complexity bounds on $\epsilon_A$ and $\epsilon_B$ specified in Proposition \ref{prop:independent_estimation}.

\subsection{Static controller and a common Lyapunov approximation}\label{sec:approx-CL}

As we have reiterated above, when the dynamics are known, the optimal LQR control law takes the form $u_t = K x_t$ for properly chosen static gain matrix $ K$.
We can reparameterize the optimization problem~\eqref{eq:robustLQRbnd} to restrict our attention to such static control policies:
\begin{align}\label{eq:robustLQRbnd-static}
\begin{split}
 \minimize_{\gamma\in[0,1)}\frac{1}{1 - \gamma}&\min_{\tf\Phi_x, \tf \Phi_u,K} \left\|\begin{bmatrix} Q^\frac{1}{2} & 0 \\ 0 & R^\frac{1}{2}\end{bmatrix}\begin{bmatrix} \tf\Phi_x \\  \tf\Phi_u \end{bmatrix}\right\|_{\htwo}\\
& \text{s.t.} \begin{bmatrix}zI-\Ah&-\Bh\end{bmatrix}\begin{bmatrix} \tf\Phi_x \\  \tf\Phi_u \end{bmatrix} = I,~~\left\|\begin{bmatrix} \tfrac{\epsilon_A}{\sqrt{\alpha}} \tf \Phi_x \\ \tfrac{\epsilon_B}{\sqrt{1-\alpha}} \tf\Phi_u \end{bmatrix} \right\|_{\hinf}\leq \gamma\\
&\qquad \tf\Phi_x, \tf \Phi_u  \in\frac{1}{z}\mathcal{RH}_\infty\,,~~ K=\tf \Phi_u \tf \Phi_x^{-1}.
 \end{split}
\end{align}
Under this reparameterization, the problem is no longer convex.  Here we present a simple application of the \emph{common Lyapunov relaxation} that allows us to find a controller $K$ using semidefinite programming.

Note that the equality constraints imply:
\begin{align*}
I=\begin{bmatrix} zI-\Ah&-\Bh\end{bmatrix}\begin{bmatrix} \tf\Phi_x \\  \tf\Phi_u \end{bmatrix} =\begin{bmatrix} zI-\Ah&-\Bh\end{bmatrix}\begin{bmatrix} I \\  K \end{bmatrix} \tf \Phi_x
=(zI-\Ah-\Bh K) \tf \Phi_x\,,
\end{align*}
revealing that we must have
\begin{align*}
	\tf \Phi_x = (zI - \Ah-\Bh K)^{-1} ~~\mbox{and}~~\tf \Phi_u= K(zI - \Ah-\Bh K)^{-1}\,.
\end{align*}
With these identifications,~\eqref{eq:robustLQRbnd-static} can be reformulated as
\begin{align}\label{eq:robustLQRbnd-static2}
\begin{split}
 \minimize_{\gamma\in[0,1)}\frac{1}{1 - \gamma}&\min_{K} \left\|\begin{bmatrix} Q^\frac{1}{2} & 0 \\ 0 & R^\frac{1}{2}K\end{bmatrix}(zI - \Ah-\Bh K)^{-1} \right\|_{\htwo}\\
& \text{s.t.} \left\|\begin{bmatrix} \tfrac{\epsilon_A}{\sqrt{\alpha}}  \\ \tfrac{\epsilon_B}{\sqrt{1-\alpha}} K\end{bmatrix} (zI - \Ah -\Bh K)^{-1}\right\|_{\hinf}\leq \gamma
 \end{split}
\end{align}

Using standard techniques from the robust control literature, we can upper bound this problem via the semidefinite program
\begin{align}\label{eq:robustLQRbnd-common-lyap}
\begin{array}{ll}
\operatorname{minimize}\limits_{X,Z,W,\alpha,\gamma} &\frac{1}{(1-\gamma)^2}
\left\{ \operatorname{Trace}(Q W_{11}) + \operatorname{Trace}(R W_{22})  \right\}\\
\mbox{subject to}
& \begin{bmatrix} X &  X  & Z^* \\
                   X & W_{11} & W_{12} \\
                  Z & W_{21} & W_{22} \end{bmatrix} \succeq 0\\
&\begin{bmatrix} X-I &  (\Ah+\Bh K)X  & 0 & 0 \\
    X(\Ah+\Bh K)^* & X & \epsilon_A X & \epsilon_B Z^* \\
    0 & \epsilon_A X & \alpha\gamma^2 I & 0\\
    0 & \epsilon_B Z & 0 & (1-\alpha)\gamma^2 I \end{bmatrix} \succeq 0\,.
\end{array}
\end{align}
Note that this optimization problem is affine in $\alpha$ when $\gamma$ is fixed. Hence, in practice we can find the optimal value of $\alpha$ as well.  A static controller can then be extracted from this optimization problem by setting $K=Z X^{-1}$.  A full derivation of this relaxation can be found in Appendix~\ref{appendix:common_lyap}.  Note that this compact SDP is simpler to solve than the truncated FIR approximation.  As demonstrated experimentally in the following section, the cost of this simplification is that the common Lyapunov approach provides a controller with slightly higher LQR cost.


\section{Numerical Experiments}    
\label{sec:experiments}  

We illustrate our results on estimation, controller synthesis, and LQR performance with numerical experiments of the end-to-end Coarse-ID control scheme. The least squares estimation procedure~\eqref{eq:ls_problem} is carried out on a simulated system in Python, and the bootstrapped error estimates are computed in parallel using PyWren~\cite{Pywren}.

All of the synthesis and performance experiments are run in MATLAB.  We make use of the YALMIP package for prototyping convex optimization~\cite{Lofberg2004} and use the MOSEK solver under an academic license~\cite{mosek}.   In particular, when using the FIR approximatsion described in Section~\ref{sec:approx-FIR}, we find it effective to make use of YALMIP's \texttt{dualize} function, which considerably reduces the computation time.

\subsection{Estimation of Example System}
We focus experiments on a particular example system. Consider the LQR problem instance specified by
\begin{align} \label{eq:exampledynamics}
\trueA  = \begin{bmatrix} 1.01 & 0.01 & 0\\
0.01 & 1.01 & 0.01\\
0 & 0.01 & 1.01\end{bmatrix}, ~~ \trueB = I, ~~ Q = 10^{-3} I, ~~ R =I \:.
\end{align}

The dynamics correspond to a marginally unstable graph Laplacian system where adjacent nodes are weakly connected,  each node receives direct input, and input size is penalized relatively more than state.  Dynamics described by graph Laplacians arise naturally in consensus and distributed averaging problems. For this system, we perform the full data identification procedure in \eqref{eq:ls_problem}, using inputs with variance $\sigma_u^2=1$ and noise with variance $\sigma_w^2=1$. The errors are estimated via the bootstrap (Algorithm \ref{alg:bootstrap}) using $M=2,000$ trials and confidence parameter $\delta = 0.05$. 

The behavior of the least squares estimates and the bootstrap error estimates are illustrated in Figure \ref{fig:eps_v_rollout}. The rollout length is fixed to $T=6$, and the number of rollouts used in the estimation is varied. As expected, increasing the number of rollouts corresponds to decreasing errors. For large enough $N$, the bootstrapped error estimates are of the same order of magnitude as the true errors. In Appendix~\ref{app:eps_v_trial_figs} we show plots for the setting in which the number of rollouts is fixed to $N=6$ while the rollout length is varied.

\begin{figure}[h!]
\centering
\begin{subfigure}[b]{\basefigwidth\textwidth}
\caption{\small Least Squares Estimation Errors}
\centerline{\includegraphics[width=\columnwidth]{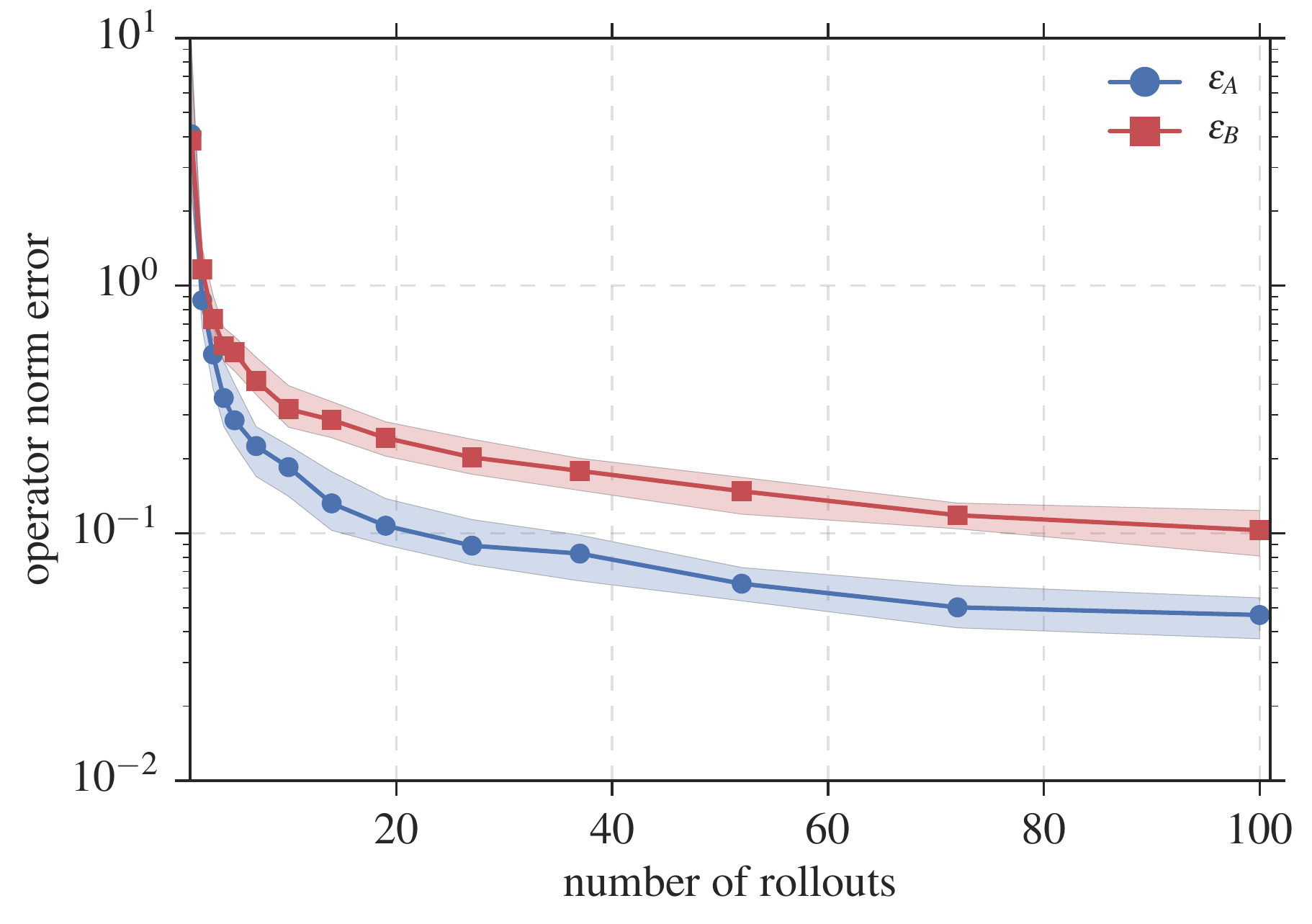}}
\end{subfigure}
\begin{subfigure}[b]{\basefigwidth\textwidth}
\caption{\small Accuracy of Bootstrap Error Estimates}
\centerline{\includegraphics[width=\columnwidth]{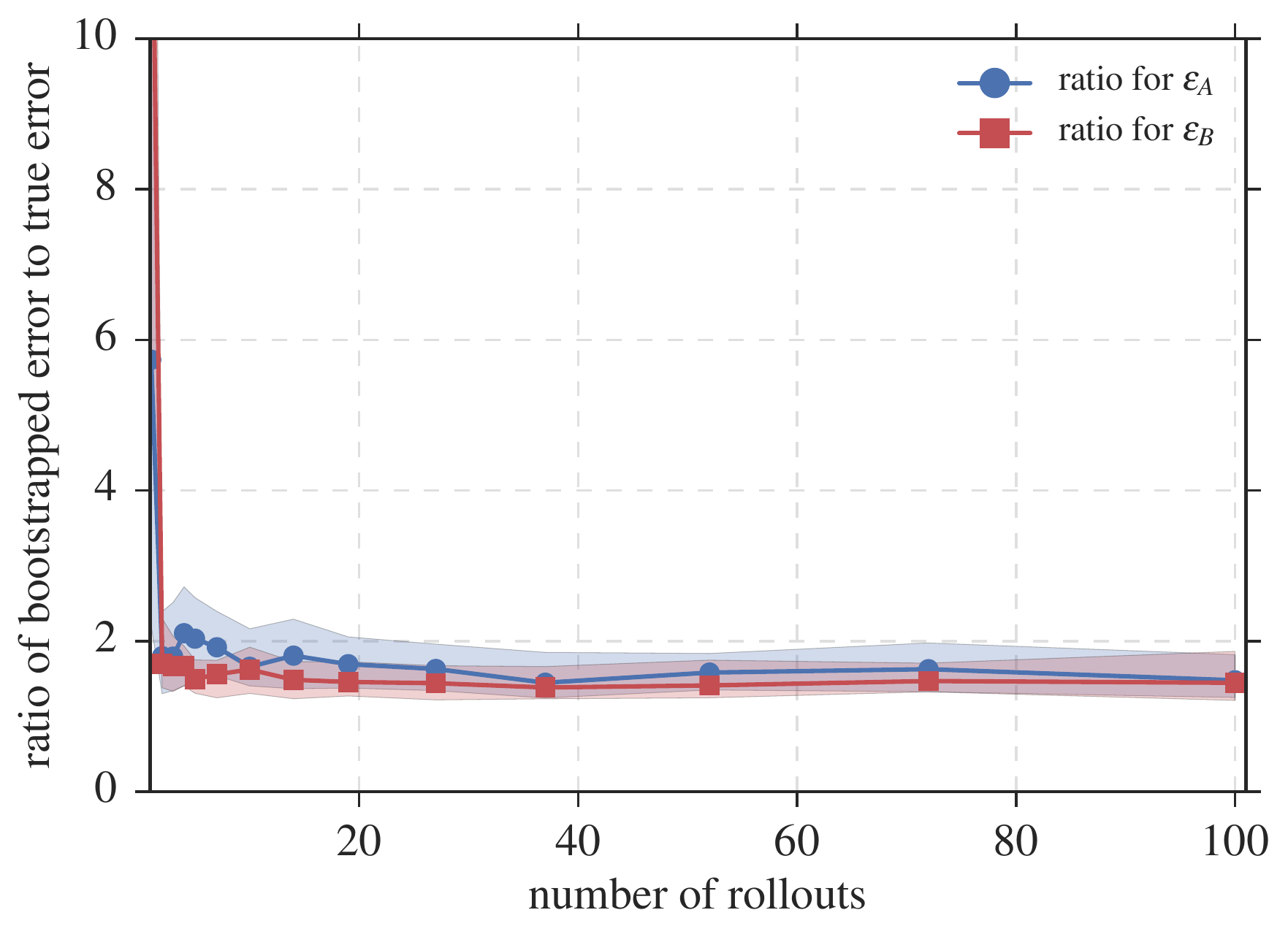}}
\end{subfigure}
\caption{\small The resulting errors from 100 repeated least squares identification experiments with rollout length $T=6$ is plotted against the number of rollouts. In (a), the median of the least squares estimation errors decreases with $N$. In (b), the ratio of the bootstrap estimates to the true estimates hover at 2. Shaded regions display quartiles.}
\label{fig:eps_v_rollout}
\end{figure}

\subsection{Controller Synthesis on Estimated System}
\label{sec:synthesis-experiments}

Using the estimates of the system in~\eqref{eq:exampledynamics}, we synthesize controllers using two robust control schemes: the convex problem in~\ref{eq:robustFIRbnd} with filters of length $L=32$ and $V$ set to $0$, and the common Lyapunov (CL) relaxation of the static synthesis problem~\eqref{eq:robustLQRbnd-static}.
Once the FIR responses $\{ \Phi_x(k) \}_{k=1}^{F}$ and
$\{ \Phi_u(k) \}_{k=1}^{F}$ are found, we need a way to implement the system responses as a controller.
We represent the dynamic controller $\tf K = \tf \Phi_u \tf \Phi_x^{-1}$ by finding
an equivalent state-space realization $(A_K, B_K, C_K, D_K)$ via Theorem 2 of \cite{anderson17}. In what follows, we compare the performance of these controllers with the nominal LQR controller (the solution to \eqref{eq:lqr-classic} with $\Ah$ and $\Bh$ as model parameters), and explore the trade-off between robustness, complexity, and performance.

The relative performance of the nominal controller is compared with robustly synthesized controllers in Figure \ref{fig:perf_v_rollout}. For both robust synthesis procedures, two controllers are compared: one using the true errors on $\trueA$ and $\trueB$, and the other using the bootstrap estimates of the errors. The robust static controller generated via the common Lyapunov approximation performs slightly worse than the more complex FIR controller, but it still achieves reasonable control performance.  Moreover, the conservative bootstrap estimates also result in worse control performance, but the degradation of performance is again modest.

Furthermore, the experiments show that the nominal controller often outperforms the robust controllers \emph{when it is stabilizing.} On the other hand, the nominal controller is not guaranteed to stabilize the true system, and as shown in Figure \ref{fig:perf_v_rollout}, it only does so in roughly 80 of the 100 instances after $N=60$ rollouts. It is also important to note a distinction between stabilization for nominal and robust controllers. When the nominal controller is not stabilizing, there is no indication to the user (though sufficient conditions for stability can be checked using our result in Corollary~\ref{lemma:robust-sls} or structured singular value methods~\cite{qiu1995formula}). On the other hand, the robust synthesis procedure will return as infeasible, alerting the user by default that the uncertainties are too high. We observe similar results when we fix the number of trials but vary the rollout length.  These figures are provided in Appendix~\ref{app:eps_v_trial_figs}.

Figure \ref{fig:FIR} explores the trade-off between performance and complexity for the computational approximations, both for FIR truncation and the common Lyapunov relaxation. We examine the tradeoff both in terms of the bound on the LQR cost (given by the value of the objective) as well as the actual achieved value. It is interesting that for smaller numbers of rollouts (and therefore larger uncertainties), the benefit of using more complex FIR models is negligible, both in terms of the actual costs and the upper bound. This trend makes sense: as uncertainties decrease to zero, the best robust controller should approach the nominal controller, which is associated with infinite impulse response (IIR) transfer functions. Furthermore, for the experiments presented here, FIR length of $L = 32$ seems to be sufficient to characterize the performance of the robust synthesis procedure in~\eqref{eq:robustLQRbnd}. Additionally, we note that static controllers are able to achieve costs of a similar magnitude.

The SLS framework guarantees a stabilizing controller for the true system provided that the computational approximations are feasible for \emph{any} value of $\gamma$ between 0 and 1, as long as the system errors $(\epsilon_A,\epsilon_B)$ are upper bounds on the true errors. Figure \ref{fig:perf_v_rollout_relaxed} displays the controller performance for robust synthesis when $\gamma$ is set to 0.999. Simply ensuring a stable model and neglecting to optimize the nominal cost yields controllers that perform nearly an order of magnitude better than those where we search for the optimal value of $\gamma$.  This observation aligns with common practice in robust control: constraints ensuring stability are only active when the cost tries to drive the system up against a safety limit.  We cannot provide end-to-end sample complexity guarantees for this method and leave such bounds as an enticing challenge for future work.

\begin{figure}[h!]
\centering
\begin{subfigure}[b]{\basefigwidth\textwidth}
\caption{\small LQR Cost Suboptimality}
\centerline{\includegraphics[width=\columnwidth]{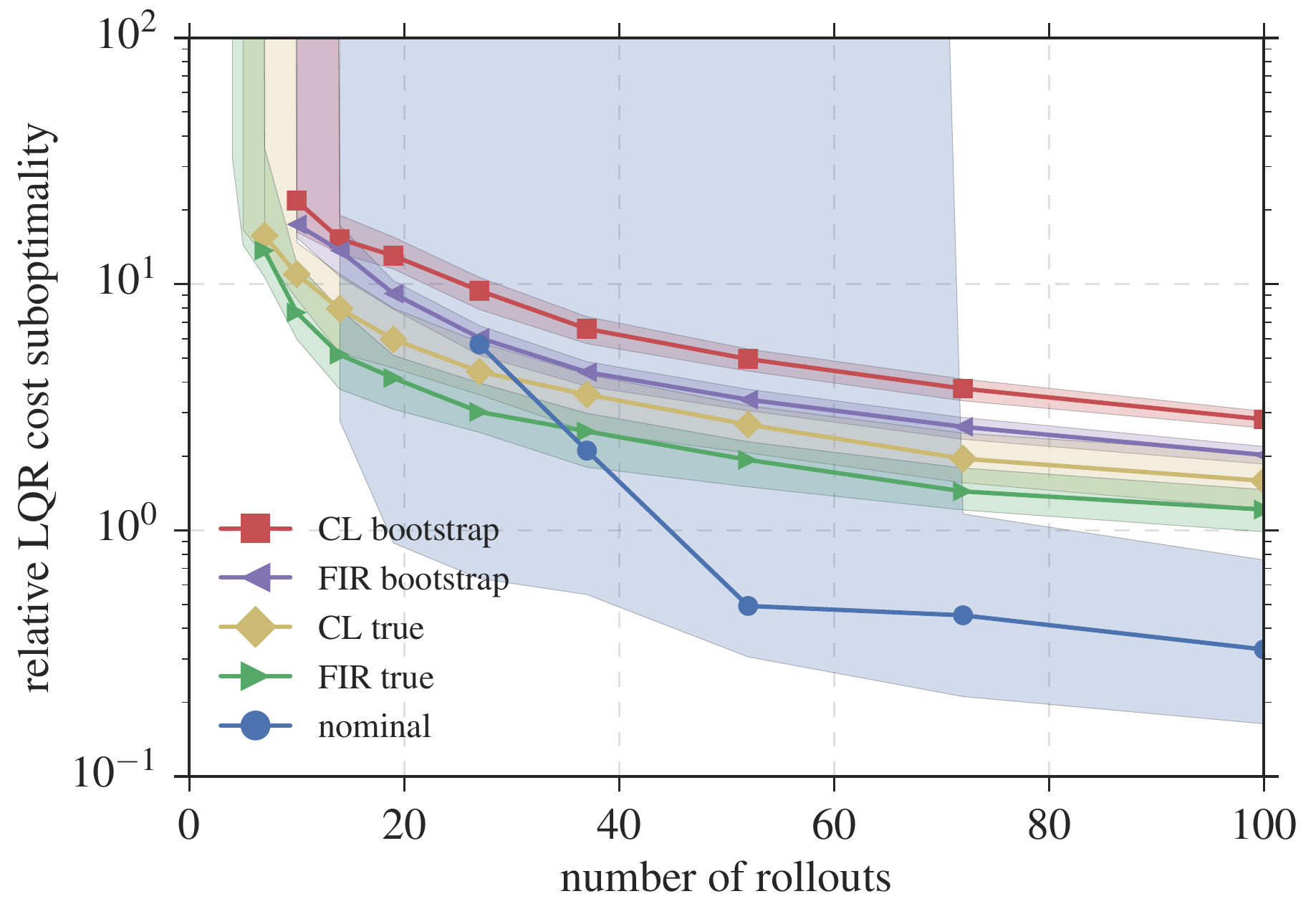}}
\end{subfigure}
\begin{subfigure}[b]{\basefigwidth\textwidth}
\caption{\small Frequency of Finding Stabilizing Controller}
\centerline{\includegraphics[width=\columnwidth]{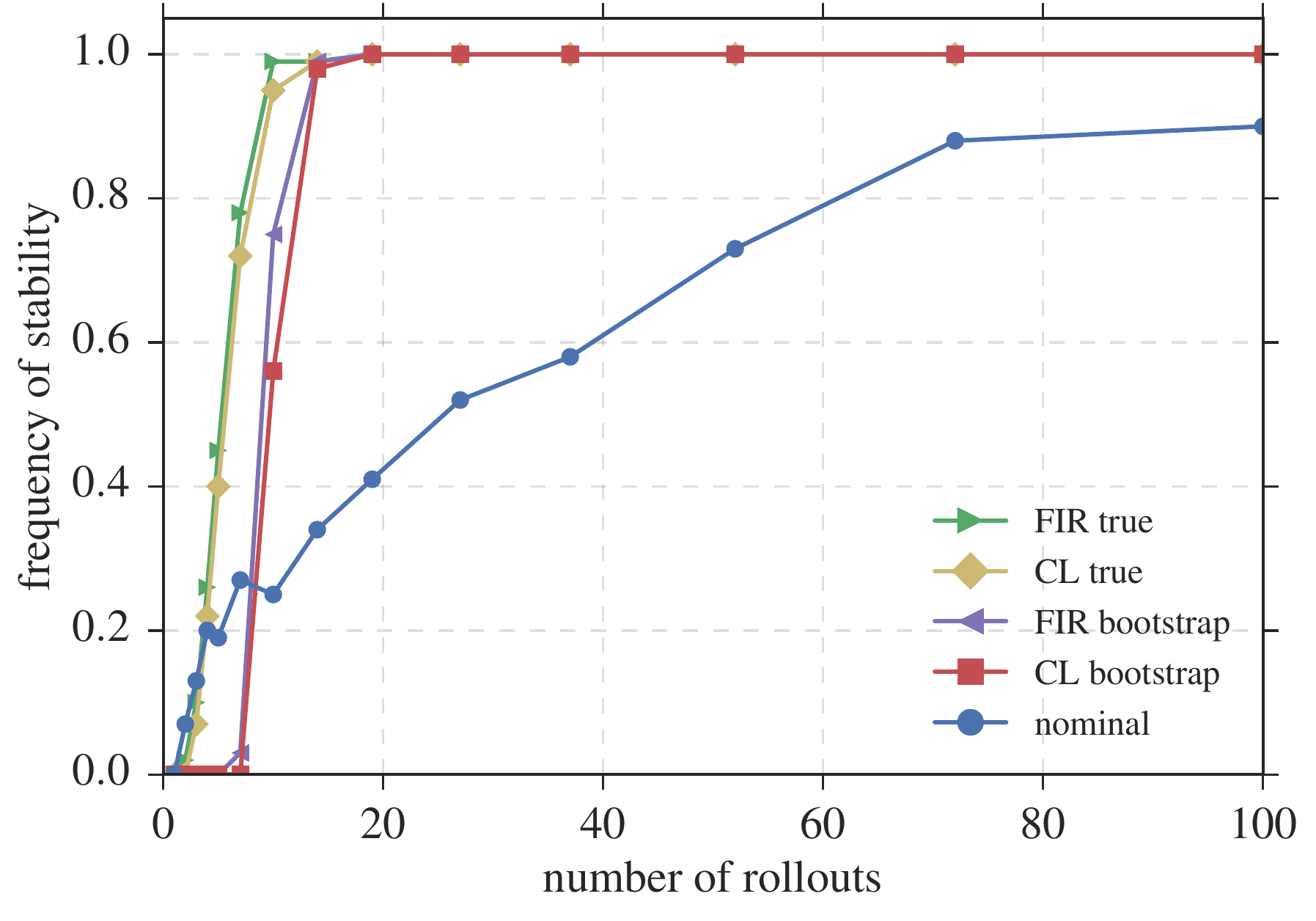}}
\end{subfigure}
\caption{\small The performance of controllers synthesized on the results of the 100 identification experiments is plotted against the number of rollouts. Controllers are synthesis nominally, using FIR truncation, and using the common Lyapunov (CL) relaxation. In (a), the median suboptimality of nominal and robustly synthesized controllers are compared, with shaded regions displaying quartiles, which go off to infinity in the case that a stabilizing controller was not found. In (b), the frequency that the synthesis methods found stabilizing controllers.}
\label{fig:perf_v_rollout}
\end{figure}

\begin{figure}[h!]
\centering

\begin{subfigure}[b]{\basefigwidth\textwidth}
\caption{\small LQR Cost Suboptimality Bound}
\centerline{\includegraphics[width=\columnwidth]{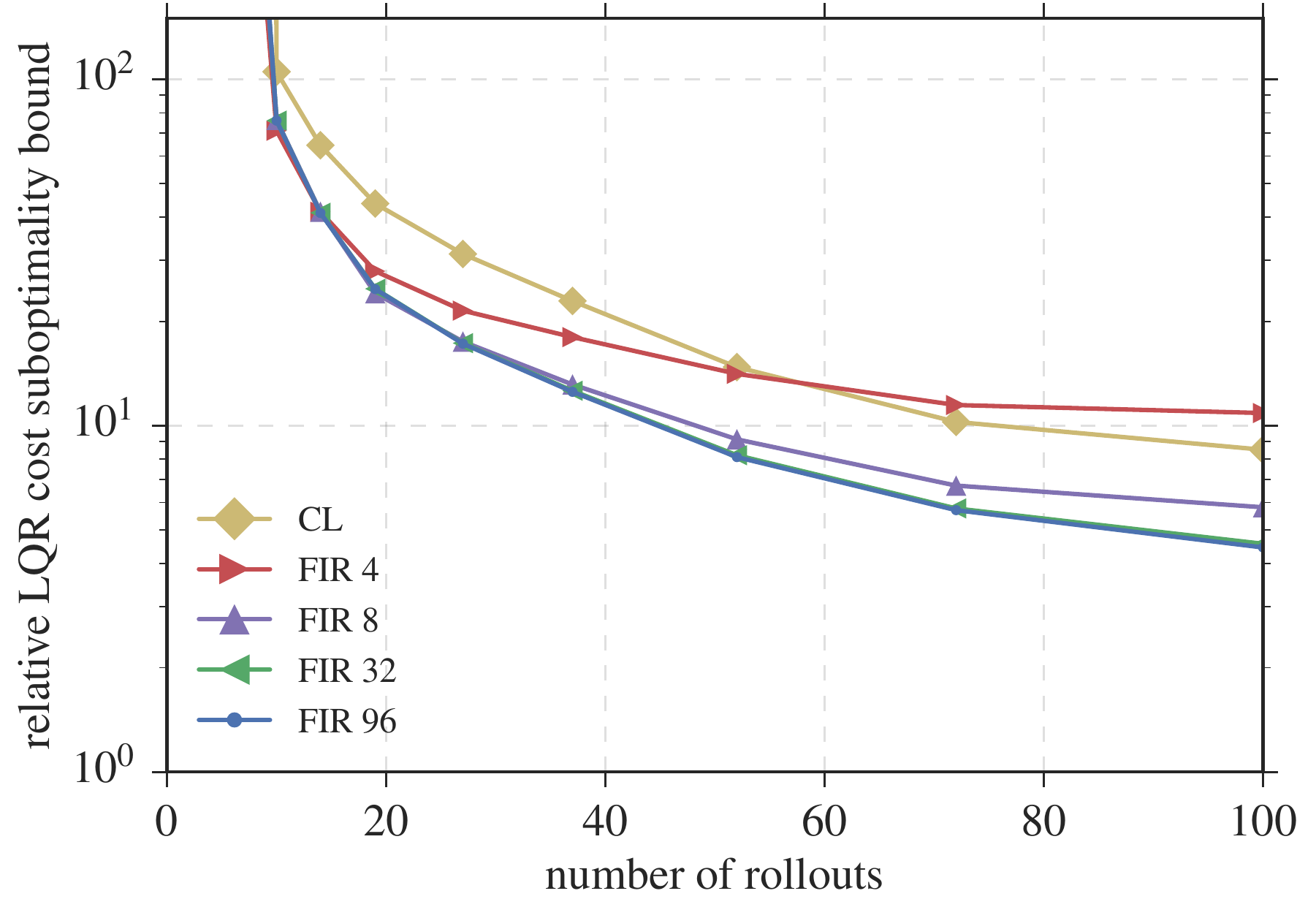}}
\end{subfigure}
\begin{subfigure}[b]{\basefigwidth\textwidth}
\caption{\small LQR Cost Suboptimality}
\centerline{\includegraphics[width=\columnwidth]{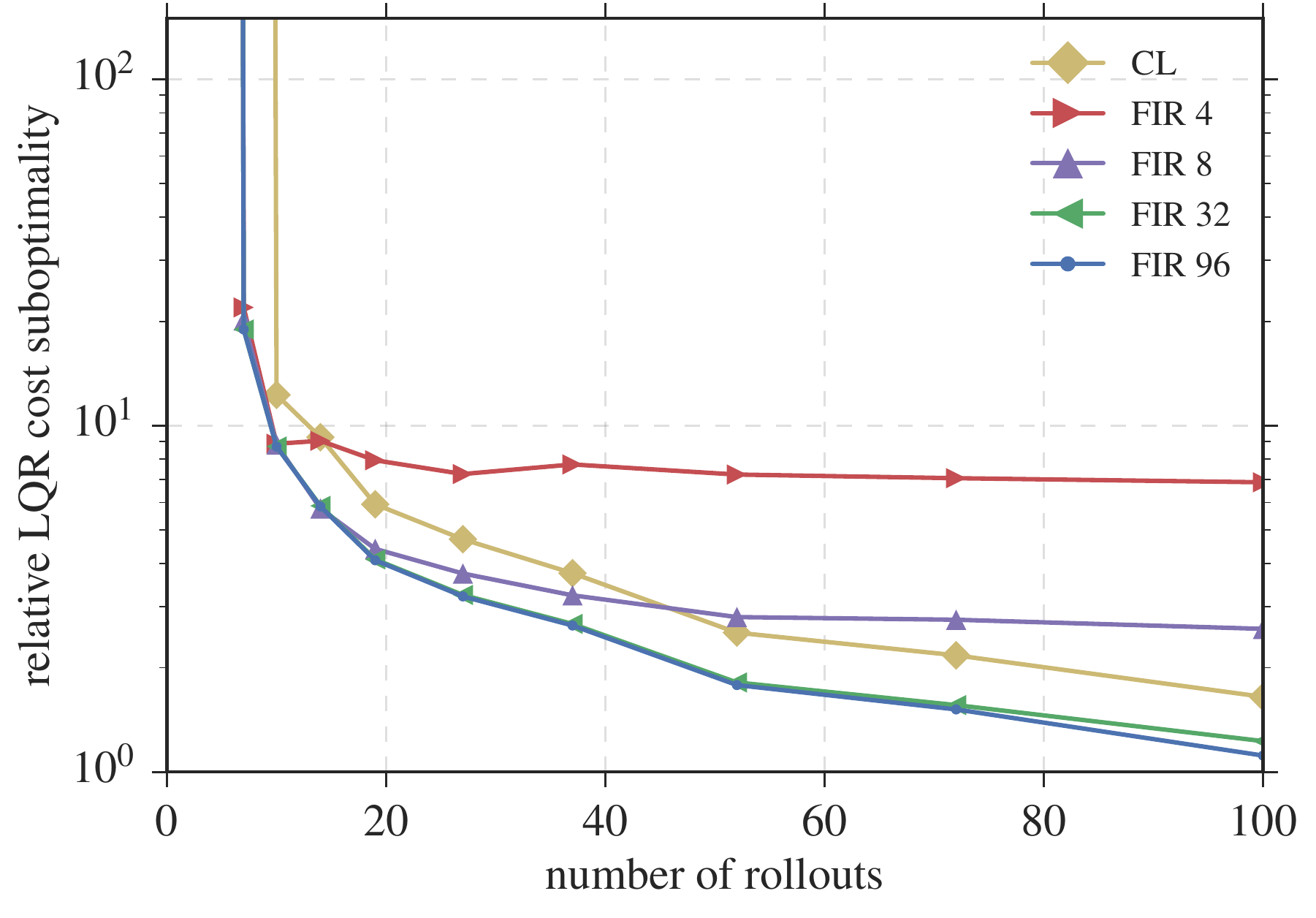}}
\end{subfigure}
\caption{\small The performance of controllers synthesized with varying FIR filter lengths on the results of 10 of the identification experiments using true errors. The median suboptimality of robustly synthesized controllers does not appear to change for FIR lengths greater than 32, and the common Lyapunov (CL) synthesis tracks the performance in both upper bound and actual cost.}
\label{fig:FIR}
\end{figure}

\begin{figure}[h!]
\centering
\begin{subfigure}[b]{\basefigwidth\textwidth}
\centerline{\small LQR Cost Suboptimality}

\centerline{\includegraphics[width=\columnwidth]{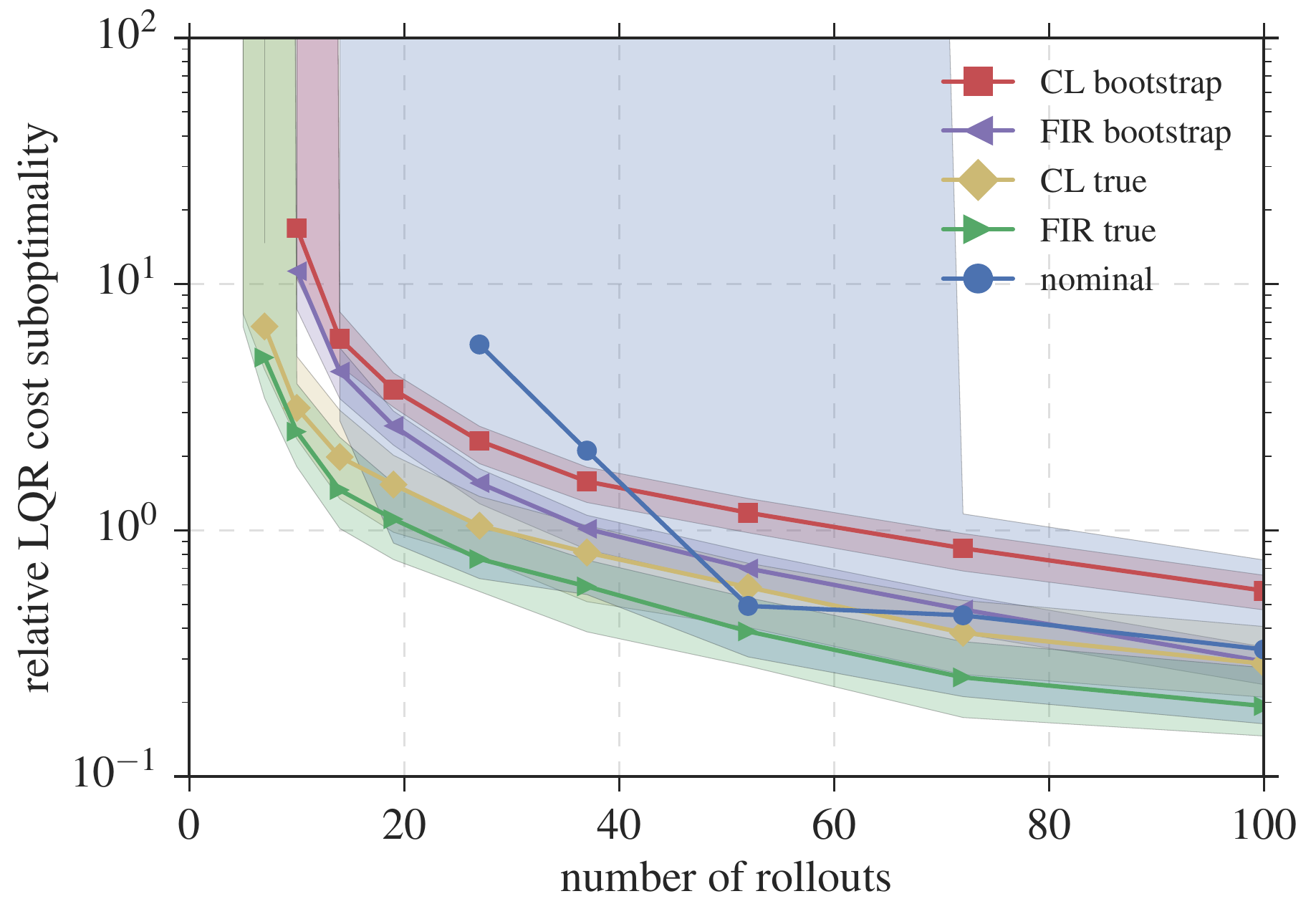}}
\end{subfigure}
\caption{\small The performance of controllers synthesized on the results of 100 identification experiments is plotted against the number of rollouts. The plot compares the median suboptimality of nominal controllers with fixed-$\gamma$ robustly synthesized controllers ($\gamma = 0.999$).}
\label{fig:perf_v_rollout_relaxed}
\end{figure}


\section{Conclusions and Future Work} 
\label{sec:conclusions}
Coarse-ID control provides a straightforward approach to merging nonasymptotic methods from system identification with contemporary Systems Level Synthesis approaches to robust control.  Indeed, many of the principles of Coarse-ID control were well established in the 90s~\cite{chen00,chen93,helmicki91}, but fusing together an end-to-end result required contemporary analysis of random matrices and a new perspective on controller synthesis.  These results can be extended in a variety of directions, and we close this paper with a discussion of some of the short-comings of our approach and of several possible applications of the Coarse-ID framework to other control settings.

\paragraph{Other performance metrics.} Though we focused exclusively on LQR in this paper, we note that all of our results on robust synthesis and end-to-end performance analysis extend to other metrics popular in control. Indeed, \emph{any} norm on the system responses $\{\Phi_x(k),\Phi_u(k)\}$ can be solved robustly using our approach; Lemma~\ref{lemma:robust-sls} holds for any norm. In turn, we can mimic the derivation in Section~\ref{sec:robust} to yield a constrained optimization problem with respect to the nominal dynamics and a norm on the uncertainty $\tf{\hat\Delta}$. This means that our suboptimality bound in Corollary~\ref{coro:lqr_cost_iid} holds true when we replace $\htwo(\mathbb{T})$ with $\hinf(\mathbb{T})$. Furthermore, similar results can be derived for other norms, so long as care is placed on the associated submultiplicative properties of the norms in question. For example, in follow up work we analyze robustness under the $\mathcal{L}_1$ norm in the context of 
constraints on the states and control signals~\cite{dean18safe}.

\paragraph{Improving the end-to-end analysis.} There are several places where
our analysis could be substantially improved.  The most obvious is that
in our estimator for the
state-transition matrices, our algorithm only uses the final time step of each
rollout.  This strategy is data inefficient, and empirically,
accuracy only improves when including all of the data.  Analyzing the full
least squares estimator is non-trivial because the design
matrix strongly depends on data to be estimated.  This poses a
challenging problem in random matrix theory that has applications in a
variety of control and reinforcement learning settings. 
In follow up work we have begun to address this issue for stable linear systems \cite{simchowitz18}.

In the context of SLS, we use a very coarse characterization of the plant uncertainty to bound the quantity in Lemma~\ref{lemma:robust-sls} and to yield a tractable optimization problem.  Indeed, the only property we use about the error between our nominal system and the true system is that the maps
\[
	x\mapsto (A-\hat{A})x~~\mbox{and}~~u\mapsto (B-\hat{B})u
\]
are contractions.  Nowhere do we use the fact that these are linear operators, or even the fact that they are the same operator from time-step to time-step.  Indeed, there are stronger bounds that could be engineered using the theory of Integral Quadratic Constraints~\cite{megretski1997system} that would take into account these additional properties.  Such tighter bounds could yield considerably less conservative control schemes in both theory and practice.

Additionally, it would be of interest to understand the loss in performance incurred by the common Lyapunov relaxation we use in our experiments.  Empirically, we see that the approximation leads to good performance, suggesting that it does not introduce much conservatism into the synthesis task. Further, our numerical experiments suggest that optimizing a nominal cost subject to robust stability constraints, as opposed to directly optimizing the SLS upper bound, leads to better empirical performance.  Future work will seek to understand whether this is a phenomenological observation specific to the systems used in our experiments, or if there is a deeper principle at play that leads to tighter sub-optimality guarantees.

\paragraph{Lower bounds.} Finding lower bounds for control problems when the model is unknown is an open question. Even for LQR, it is not at all clear how well the system $(A,B)$ needs to be known in order to attain good control performance.  While we produce reasonable worst-case upper bounds for this problem, we know of no lower bounds.  Such bounds would offer a reasonable benchmark for how well one could ever expect to do with no priors on the linear system dynamics.

\paragraph{Integrating Coarse-ID control in other control paradigms.} The
end-to-end Coarse-ID control framework should be applicable in a variety of settings.  For example, in Model Predictive Control (MPC), controller synthesis
problems are approximately solved on finite time horizons, one step is taken,
and then this process is repeated~\cite{BorrelliMPCBook}.  MPC is an effective
solution which substitutes fast optimization solvers for clever, complex
control design.  We believe it will be straightforward to extend the
Coarse-ID paradigm to MPC, using a similar perturbation argument as in
Section~\ref{sec:robust}.    The main challenges
in MPC lie in how to guarantee that safety constraints are maintained
throughout execution without too much conservatism in control costs.

Another interesting investigation lies in the area of adaptive control, where we
could investigate how to incorporate new data into coarse models to further refine
constraint sets and costs.  Indeed, some work has already been done in this
space. We propose to investigate
how to operationalize and extend the notion of \emph{optimistic exploration}
proposed in the context of continuous control by Abbasi-Yadkori and
Szepesvari~\cite{abbasi2011regret}.  The idea behind optimistic improvement is
to select the model that would give the best optimization cost if the current
model was true.  In this way, we fail fast, either receiving a good cost or learning
quickly that our model is incorrect.
It would be
worth investigating whether the Coarse-ID framework can make it simple to
update a least squares estimate for the system parameters and then provide an
efficient mechanism for choosing the next optimistic control.

Finally, Coarse-ID control could be relevant to nonlinear control applications.  In nonlinear control, iterative LQR
schemes are remarkably effective~\cite{li2004iterative}.  Hence, it would be
interesting to understand how parametric model errors can be estimated and
mitigated in a control loop that employs iterative LQR or similar dynamic
programming methods.

\paragraph{Sample complexities of reinforcement learning for continuous control.}
Finally, we imagine that the analysis in this paper may be useful for
understanding popular reinforcement learning algorithms that are also being
tested for continuous control.  Reinforcement learning directly attacks a
control cost in question without resorting to any specific identification
scheme.  While this suffers from the drawback that generally speaking, no
parameter convergence can be guaranteed, it is ideally suited to ignoring modes
that do not affect control performance.  For instance, it might not be
important to get a good estimate of very stable modes or of lightly damped
modes that do not substantially affect the performance.

There are two parallel problems here.  First, it would be of interest to determine system identification algorithms that are tuned to particular control tasks.  In the Coarse-ID control approach, the estimation and control are completely decoupled.  However, it may be beneficial to inform the identification algorithm about the desired cost, resulting in improved sample complexity.

From a different perspective, Policy Gradient and Q-Learning methods applied to LQR could yield important insights about the pros and cons of such methods.  There are classic papers~\cite{bradtke1994adaptive} on Q-Learning for LQR, but these use asymptotic analysis.  Recently, the first such analysis for Policy Gradient has appeared, though the precise scaling with respect to system parameters is not yet understood~\cite{Fazel18}.  Providing clean nonasymptotic bounds here could help provide a rapprochement between machine learning and adaptive control, with optimization negotiating the truce.

\section*{Acknowledgements}
We thank Ross Boczar, Qingqing Huang, Laurent Lessard, Michael Littman, Manfred Morari, Andrew Packard, Anders Rantzer, Daniel Russo, and Ludwig Schmidt for many helpful comments and suggestions. We also thank the anonymous referees for making several suggestions that have significantly improved the paper and its presentation. SD is supported by an NSF Graduate Research Fellowship under Grant No.  DGE 1752814.  NM is generously funded by grants from the AFOSR and NSF, and by gifts from Huawei and Google. BR is generously supported by NSF award CCF-1359814, ONR awards N00014-14-1-0024 and N00014-17-1-2191, the DARPA Fundamental Limits of Learning (Fun LoL) Program,  a Sloan Research Fellowship, and a Google Faculty Award.

\begin{small}
\bibliographystyle{abbrvnat}  
\bibliography{lqr} 
\end{small} 

\appendix
\section{Proof of Lemma~\ref{lemma:product_gaussians}}
First, recall Bernstein's lemma.
Let $X_1, ..., X_p$ be zero-mean independent r.v.s
satisfying the Orlicz norm bound $\norm{X_i}_{\psi_1} \leq K$.
Then as long as $p \geq 2 \log(1/\delta)$, with probability at least $1-\delta$,
\begin{align*}
  \sum_{i=1}^{p} X_i \leq K \sqrt{2n \log(1/\delta)} \:.
\end{align*}
Next, let $Q$ be an $m \times n$ matrix.  Let $u_1, ..., u_{M_\varepsilon}$ be a
$\varepsilon$-net for the $m$-dimensional $\ell_2$ ball, and similarly let
$v_1, ..., v_{N_\varepsilon}$ be a $\varepsilon$ covering for the $n$-dimensional $\ell_2$ ball.
For each $\norm{u}_2=1$ and $\norm{v}_2=1$, let $u_i$, $v_j$ denote the
elements in the respective nets such that $\norm{u - u_i}_2 \leq \varepsilon$
and $\norm{v - v_j}_2 \leq \varepsilon$.  Then,
\begin{align*}
  u^\T Q v &= (u - u_i + u_i)^\T Q v = (u - u_i)^\T Q v + u_i^\T Q ( v - v_j + v_j )  \\
  &= (u - u_i)^\T Q v + u_i^\T Q ( v - v_j ) + u_i^\T Q v_j \:.
\end{align*}
Hence,
\begin{align*}
  u^\T Q v \leq 2 \varepsilon \norm{Q}_2 + u_i^\T Q v_j \leq  2 \varepsilon \norm{Q}_2 +\max_{1 \leq i \leq M_\varepsilon, 1 \leq j \leq N_\varepsilon} u_i^\T Q v_j  \:.
\end{align*}
Since $u,v$ are arbitrary on the sphere,
\begin{align*}
  \norm{Q}_2  \leq \frac{1}{1-2\varepsilon} \max_{1 \leq i \leq M_\varepsilon, 1 \leq j \leq N_\varepsilon} u_i^\T Q v_j \:.
\end{align*}
Now we study the problem at hand.  Choose $\varepsilon = 1/4$.
By a standard volume comparison argument,
we have that $M_\varepsilon \leq
9^m$ and $N_\varepsilon \leq 9^n$, and that
\begin{align*}
  \bignorm{ \sum_{k=1}^{N} f_k g_k^\T }_2 \leq 2 \max_{1 \leq i \leq M_\varepsilon, 1 \leq j \leq N_\varepsilon} \sum_{k=1}^{N} (u_i^\T f_k) (g_k^\T v_j) \:.
\end{align*}
Note that $u_i^\T f_k \sim N(0, u_i^\T \Sigma_f u_i)$ and $g_k^\T v_j \sim N(0,
v_j^\T \Sigma_g v_j)$.  By independence of $f_k$ and $g_k$, $(u_i^\T
f_k)(g_k^\T v_j)$ is a zero mean sub-Exponential random variable, and therefore $\norm{(u_i^\T
f_k)(g_k^\T v_j)}_{\psi_1} \leq \sqrt{2}
  \norm{\Sigma_f}_2^{1/2} \norm{\Sigma_g}_2^{1/2}$.
Hence, for each pair $u_i, v_j$ we have with probability at least $1 - \delta/9^{m+n}$,
\begin{align*}
  \sum_{k=1}^{N} (u_i^\T f_k) (g_k^\T v_j) \leq 2\norm{\Sigma_f}_2^{1/2}\norm{\Sigma_g}_2^{1/2} \sqrt{N (m + n) \log(9/\delta)} \:.
\end{align*}
Taking a union bound over all pairs in the $\varepsilon$-net yields the claim.

\section{Proof of Proposition~\ref{prop:data_dependent}}
\label{app:data_dependent}

For this proof we need a lemma similar to Lemma~\ref{lemma:product_gaussians}. The following is a standard result in high-dimensional statistics \cite{wainwright2019high}, and we state it here without proof.

\begin{lemma}
\label{lem:operator_norm}
Let $W \in \R^{N \times \statedim}$ be a matrix with each entry i.i.d. $\Ncal(0, \sigma_w^2)$. Then, with probability $1 - \delta$, we have
\begin{align*}
\|W\|_2 \leq \sigma_w(\sqrt{N} + \sqrt{\statedim} + \sqrt{2 \log(1/\delta)}).
\end{align*}
\end{lemma}

As before we use $Z$ to denote the $N \times (\statedim + \inputdim)$ matrix with rows equal to $z_\ell^\top = \begin{bmatrix}
(x^{(\ell)})^\top & (u^{(\ell)})^\top
\end{bmatrix}$. Also, we denote by $W$ the $N \times \statedim$ matrix with columns equal to $w^{(\ell)}$. Therefore, the error matrix for the ordinary least squares estimator satisfies
\begin{align*}
E = \begin{bmatrix}
(\Ahat - A)^\top \\
(\Bhat - B)^\top
\end{bmatrix} = (Z^\top Z)^{-1} Z^\top W,
\end{align*}
when the matrix $Z$ has rank $\statedim + \inputdim$. Under the assumption that $N \geq \statedim + \inputdim$ we consider the singular value decomposition  $Z = U \Lambda V^\top$, where $V, \Lambda \in \R^{(\statedim + \inputdim) \times (\statedim + \inputdim) }$ and $U \in \R^{N \times (\statedim + \inputdim)}$. Therefore, when $\Lambda$ is invertible,
\begin{align*}
E = V (\Lambda^\top \Lambda)^{-1} \Lambda^\top U^\top W = V \Lambda^{-1} U^\top W.
\end{align*}

This implies that
\begin{align*}
E E^\top &= V \Lambda^{-1} U^\top W W^\top U \Lambda^{-1} V^\top \preceq \| U^\top W\|_2^2 V \Lambda^{-2} V^\top= \| U^\top W\|_2^2 (Z^\top Z)^{-1}.
\end{align*}

Since the columns of $U$ are orthonormal, it follows that the entries of $U^\top W$ are i.i.d. $\Ncal(0, \sigma_w^2)$. Hence, the conclusion follows by Lemma~\ref{lem:operator_norm}.

\section{Derivation of the LQR cost as an $\htwo$ norm}\label{app:h2}
In this section, we consider the transfer function description of the infinite horizon LQR optimal control problem. In particular,
we show how it can be recast as an equivalent $\mathcal{H}_2$ optimal control problem in terms of the system response variables defined in
Theorem \ref{thm:param}.

Recall that stable and achievable system responses $(\tf \Phi_x,\tf \Phi_u)$, as characterized in equation \eqref{eq:achievable}, describe the closed-loop map from disturbance signal $\tf w$ to the state and control action $(\tf x, \tf u)$ achieved by the controller $\tf K = \tf \Phi_u \tf \Phi_x^{-1}$, i.e.,
\begin{equation*}
\begin{bmatrix} \tf x \\ \tf u \end{bmatrix} = \begin{bmatrix} \tf \Phi_x \\ \tf \Phi_u \end{bmatrix} \tf w.
\end{equation*}
Letting $\tf \Phi_x  = \sum_{t=1}^\infty \Phi_x(t) z^{-t}$ and $\tf \Phi_u = \sum_{t=1}^\infty \Phi_u(t) z^{-t}$, we can then equivalently write for any $t \geq 1$
\begin{equation}
\begin{bmatrix}  x_t \\ u_t \end{bmatrix} = \sum_{k=1}^t\begin{bmatrix}  \Phi_x(k) \\  \Phi_u(k) \end{bmatrix}  w_{t-k}.
\label{eq:time_response}
\end{equation}
For a disturbance process distributed as $w_t \iid \mathcal{N}(0,\sigma_w^2 I_\statedim)$, it follows from equation \eqref{eq:time_response} that
\begin{align*}
    \E\left[ x_t^\T Q x_t\right] &= \sigma_w^2\sum_{k=1}^t \Tr(\Phi_x(k)^\T Q\Phi_x(k)) \:, \\
    \E\left[ u_t^\T R u_t\right] &= \sigma_w^2\sum_{k=1}^t \Tr(\Phi_u(k)^\T R\Phi_u(k)) \:.
\end{align*}
We can then write
\begin{align*}
    \lim_{T\to \infty} \frac{1}{T} \sum_{t=1}^T \E\left[ x_t^\T Q x_t  +  u_t^\T R u_t\right]  &=  \sigma_w^2\left[\sum_{t=1}^\infty \Tr(\Phi_x(t)^\T Q\Phi_x(t)) + \Tr (\Phi_u(t)^\T R\Phi_u(t))\right] \\
&= \sigma_w^2\sum_{t=1}^\infty \bignorm{\begin{bmatrix} Q^\frac{1}{2} & 0 \\ 0 & R^\frac{1}{2} \end{bmatrix}\begin{bmatrix} \Phi_x(t) \\ \Phi_u(t) \end{bmatrix}}_F^2 \\
&= \frac{\sigma_w^2}{2\pi} \int_{\mathbb{T}} \bignorm{\begin{bmatrix} Q^\frac{1}{2} & 0 \\ 0 & R^\frac{1}{2} \end{bmatrix}\begin{bmatrix} \tf \Phi_x \\ \tf \Phi_u \end{bmatrix}}_F^2 \; dz  \\
&=\sigma_w^2 \bignorm{\begin{bmatrix} Q^\frac{1}{2} & 0 \\ 0 & R^\frac{1}{2} \end{bmatrix}\begin{bmatrix} \tf \Phi_x \\ \tf\Phi_u \end{bmatrix}}_{\mathcal{H}_2}^2 \:,
\end{align*}
where the second to last equality is due to Parseval's Theorem.

\section{Proof of Theorem \ref{thm:FIR_subopt}}
\label{app:FIR}

To understand the effect of restricting the optimization to FIR transfer functions we need to understand the decay of the transfer functions $\Res{\Ahat+\Bhat\trueK}$ and $\trueK \Res{\Ahat+\Bhat\trueK}$. To this end we consider $C_\star  > 0$ and $\rho_\star \in (0, 1)$ such that $\| (\trueA + \trueB \trueK)^t \|_2 \leq C_\star \rho_\star^t$ for all $t \geq 0$. Such $C_\star$ and $\rho_\star$ exist because $\trueK$ stabilizes the system $(\trueA, \trueB)$. The next lemma quantifies how well $\trueK$ stabilizes the system $(\Ahat, \Bhat)$ when the estimation error is small. 
 
\begin{lemma}
\label{lem:simple_perturbation}
Suppose $\epsilon_A + \epsilon_B \|\trueK\|_2 \leq \frac{1 - \rho_\star}{2 C_\star}$. Then, 
\begin{align*}
\| (\Ahat + \Bhat\trueK)^t \|_2 \leq C_\star \left( \frac{1 + \rho_\star}{2} \right)^t \;  \text{, for all } \; t \geq 0. 
\end{align*}
\end{lemma}

\begin{proof}
The claim is obvious when $t = 0$. Fix an integer $t \geq 1$ and denote $M = \trueA + \trueB \trueK$. Then, 
if $\Delta = \Delta_A + \Delta_B \trueK$, we have $\Ahat + \Bhat \trueK = M + \Delta$. 

Consider the expansion of $(M+\Delta)^t$ into $2^k$ terms.
Label all these terms as $T_{i,j}$ for $i=0, ..., t$ and $j=1, ..., {t \choose i}$
where $i$ denotes the degree of $\Delta$ in the term.
Using the fact that $\norm{M^t}_2 \leq C_\star \rho_\star^t$ for all $t \geq 0$, we have 
 $\norm{T_{i,j}}_2 \leq C^{i+1} \rho^{t-i} \norm{\Delta}_2^i$.
Hence by triangle inequality:
\begin{align*}
  \norm{(M + \Delta)^t}_2 &\leq \sum_{i=0}^{t} \sum_{j} \norm{T_{i,j}}_2 \\
  &\leq \sum_{i=0}^{t} {t \choose i} C_\star^{i+1} \rho_\star^{t-i} \norm{\Delta}^i_2 \\
  &= C_\star  \sum_{i=0}^{t} {t \choose i} (C_\star \norm{\Delta}_2)^i \rho_\star^{t-i} \\
  &= C_\star (C_\star \norm{\Delta}_2 + \rho_\star)^t \\
  &\leq C_\star \left(\frac{1 + \rho_\star}{2} \right)^t \:,
\end{align*}
where the last inequality uses the fact $\norm{\Delta}_2 \leq \epsilon_A + \epsilon_B \norm{\trueK}_2 \leq  \frac{ 1 -\rho_\star}{2C_\star}$.
\end{proof}

For the remainder of this discussion, we use the following notation to denote the restriction of a system response to its first $L$ time-steps:
\begin{equation}
\tf\Phi_{x}(1:L) = \sum_{t=1}^L \frac{1}{z^t}\Phi_x(t), \ \tf\Phi_{u}(1:L) = \sum_{t=1}^L \frac{1}{z^t}\Phi_u(t).
\label{eq:FIR_restriction}
\end{equation}

To prove Theorem \ref{thm:FIR_subopt} we must relate the optimal controller $\trueK$ with the optimal solution of the optimization problem  \eqref{eq:robustFIRbnd}. In the next lemma we use $\trueK$ to construct a feasible solution for  problem  \eqref{eq:robustFIRbnd}. As before, we denote $\zeta = (\epsilon_A + \epsilon_B \norm{\trueK}_2) \hinfnorm{\Res{\trueA + \trueB \trueK}}$.

\begin{lemma}\label{lem:FIR_feasibility}
Set $\alpha = 1/2$ in problem \eqref{eq:robustFIRbnd}, and assume that $\epsilon_A + \epsilon_B \|\trueK\|_2 \leq \frac{1 - \rho_\star}{2 C_\star}$, $\zeta < 1/5$, and 
\begin{equation}
L \geq \frac{4 \log\left(\frac{C_\star}{\zeta} \right)}{1 - \rho_\star}.
\label{eq:L-feasible}
\end{equation}
Then, optimization problem \eqref{eq:robustFIRbnd} is feasible, and the following is one such feasible solution:
\begin{equation}
\widetilde{\tf\Phi}_x = \Res{\Ah+\Bh\trueK}(1:L),~~ \widetilde{\tf \Phi}_u = \trueK\Res{\Ah+\Bh\trueK}(1:L),~~\widetilde{V}= - \Res{\Ah+\Bh\trueK}(L+1),~~\tilde \gamma = \frac{4 \zeta}{1 - \zeta}.
\end{equation}
\end{lemma}
\begin{proof}

From Lemma~\ref{lem:simple_perturbation} and the assumption on $\zeta$ we have that $\norm{(\Ahat + \Bhat \trueK)^t}_2 \leq C_\star \left(\frac{1 + \rho_\star}{2} \right)^t$ for all $t \geq 0$. In particular, since $\Res{\Ahat + \Bhat \trueK} (L + 1) = (\Ahat + \Bhat \trueK)^{L} $, we have  $\norm{\widetilde V} = \norm{(\Ahat + \Bhat \trueK)^L } \leq C_\star  \left(\frac{1 + \rho_\star}{2} \right)^L \leq \zeta$. The last inequality is true because we assumed $L$ is sufficiently large. 

Once again, since $\Res{\Ahat + \Bhat \trueK} (L + 1) = (\Ahat + \Bhat \trueK)^{L} $, it can be easily seen that our choice of $\widetilde {\tf \Phi}_x$, $\widetilde {\tf \Phi}_u$, and $\widetilde V$ satisfy the linear constraint of problem \eqref{eq:robustFIRbnd}. It remains to prove that
\begin{align*}
\sqrt{2}\bighinfnorm{\begin{bmatrix}{\epsilon_A}{\tf\Phi_x}\\ {\epsilon_B}{\tf\Phi_u}\end{bmatrix}} + \twonorm{\widetilde V} \leq \tilde \gamma < 1.
\end{align*}

The second inequality holds because of our assumption on $\zeta$. We already know that $ \twonorm{\widetilde V} \leq \zeta$. Now, we bound:
\begin{align*}
\bighinfnorm{\begin{bmatrix}{\epsilon_A}{\widetilde{\tf\Phi}_x}\\ {\epsilon_B}{ \widetilde{\tf\Phi}_u}\end{bmatrix}} &\leq (\epsilon_A +\epsilon_B\twonorm{\trueK})\hinfnorm{\Res{\Ahat + \Bhat\trueK}(1:L)} \\
&\leq  (\epsilon_A +\epsilon_B \twonorm{\trueK})(\hinfnorm{\Res{\Ahat + \Bhat\trueK}} +  \hinfnorm{\Res{\Ahat + \Bhat\trueK} (L + 1 : \infty)}).
\end{align*}
These inequalities follow from the definition of $(\widetilde{\tf\Phi}_x, \widetilde{\tf \Phi}_u)$ and the triangle inequality.

Now, we recall that $\Res{\Ahat + \Bhat \trueK} = \Res{\trueA + \trueB \trueK} (I + \tf \Delta)^{-1}$, where $\tf \Delta = - (\Delta_A + \Delta_B \trueK) \Res{\trueA + \trueB \trueK}$. Then, since $\hinfnorm{\tf \Delta} \leq \zeta$ (due to  Proposition \ref{prop:bound}), we have $\hinfnorm{\Res{\Ahat + \Bhat \trueK}} \leq \frac{1}{1 - \zeta} \hinfnorm{\Res{\trueA + \trueB \trueK}}$. 

We can upper bound 

\begin{align*}
\hinfnorm{\Res{\Ahat + \Bhat\trueK}(L + 1 : \infty)} \leq \sum_{t = L + 1}^\infty \twonorm{\Res{\Ahat + \Bhat \trueK}({t})} \leq C_\star \left(\frac{1 + \rho_\star}{2}\right)^L \sum_{t = 0}^\infty \left( \frac{1 + \rho_\star}{2}\right)^t = \frac{2 C_\star}{1 - \rho_\star} \left(\frac{1 + \rho_\star}{2}\right)^L. 
\end{align*}

Then, since we assumed that $\epsilon_A$ and $\epsilon_B$ are sufficiently small  and that $L$ is sufficiently large, we obatin 
\begin{align*}
(\epsilon_A +\epsilon_B \twonorm{\trueK}) \hinfnorm{\Res{\Ahat + \Bhat\trueK}(L + 1 : \infty)} \leq \zeta. 
\end{align*}

Therefore, 
\begin{align*}
\bighinfnorm{\begin{bmatrix}{\epsilon_A}{\widetilde{\tf\Phi}_x}\\ {\epsilon_B}{ \widetilde{\tf\Phi}_u}\end{bmatrix}} &\leq \frac{\zeta}{1 - \zeta} + \zeta \leq \frac{2\zeta}{1 - \zeta}. 
\end{align*}

The conclusion follows. 
\end{proof}

\begin{proof}[Proof of Theorem \ref{thm:FIR_subopt}]

As all of the assumptions of Lemma \ref{lem:FIR_feasibility} are satisfied, optimization problem \eqref{eq:robustFIRbnd} is feasible. We denote $(\tf \Phi_x^\star, \tf \Phi_u^\star, V_\star, \gamma_\star)$ the optimal solution of problem \eqref{eq:robustFIRbnd}. We denote
\[
\Dh :=  \Delta_A \tf \Phi_{x}^\star + \Delta_B \tf \Phi_{u}^\star + \frac{1}{z^L}V_\star.
\]
Then, we have 
\begin{align*}
\begin{bmatrix}
zI - \trueA & - \trueB
\end{bmatrix}
\begin{bmatrix}
\tf \Phi_{x}^\star \\ \tf \Phi_{u}^\star
\end{bmatrix} = I + \Dh. 
\end{align*}

Applying the triangle inequality, and leveraging Proposition \ref{prop:bound}, we can verify that
\[
\hinfnorm{\Dh} \leq \sqrt2\bighinfnorm{\begin{bmatrix}\epsilon_A \tf\Phi_{x}^\star \\ \epsilon_B \tf \Phi_{u}^\star \end{bmatrix}} + \twonorm{V_\star} \leq \gamma_\star < 1,
\]
where the last two inequalities are true because the optimal solution is a feasible point of the optimization problem \eqref{eq:robustFIRbnd}.  

We now apply Lemma \ref{lemma:robust-sls} to characterize the response achieved by the FIR approximate controller $\tf K_L$ on the true system $(\trueA,\trueB)$:
\begin{align*}
J(\trueA,\trueB,\tf K_L) &= \bightwonorm{\begin{bmatrix} Q^\frac{1}{2} & 0 \\ 0 & R^\frac{1}{2} \end{bmatrix} \begin{bmatrix}\tf \Phi_{x}^\star\\ \tf \Phi_{u}^\star \end{bmatrix}(I+\Dh)^{-1} }\\
& \leq \frac{1}{1-\gamma_\star}\bightwonorm{\begin{bmatrix} Q^\frac{1}{2} & 0 \\ 0 & R^\frac{1}{2} \end{bmatrix} \begin{bmatrix}{\tf \Phi_{x}^\star}\\{\tf \Phi_{u}^\star}\end{bmatrix}}.
\end{align*}

Denote by $(\widetilde{\tf\Phi}_x, \widetilde{\tf\Phi}_u, \widetilde{V}, \tilde{\gamma})$ the feasible solution constructed in Lemma \ref{lem:FIR_feasibility}, and let $J_L( \Ahat , \Bhat, \trueK)$ denote the truncation of the LQR cost achieved by controller $\trueK$ on system $(\Ahat, \Bhat)$ to its first $L$ time-steps.

Then,
\begin{align*}
 \frac{1}{1-\gamma_\star}\bightwonorm{\begin{bmatrix} Q^\frac{1}{2} & 0 \\ 0 & R^\frac{1}{2} \end{bmatrix} \begin{bmatrix}{\tf \Phi_{x}^\star}\\{\tf \Phi_{u}^\star}\end{bmatrix}}
  &\leq \frac{1}{1- \tilde{\gamma}}\bightwonorm{\begin{bmatrix} Q^\frac{1}{2} & 0 \\ 0& R^\frac{1}{2} \end{bmatrix} \begin{bmatrix}{\widetilde{\tf \Phi}_x}\\{ \widetilde{\tf \Phi}_u}\end{bmatrix}}\\
 &= \frac{1}{1- \tilde \gamma} {J_L(\Ah,\Bh,\trueK)} \\
 & \leq \frac{1}{1- \tilde \gamma} {J(\Ah,\Bh,\trueK)} \\
 & \leq \frac{1}{1- \tilde \gamma} \frac{1}{1-\hinfnorm{\tf \Delta}}{J_\star},
\end{align*}
where $\tf \Delta = - (\Delta_A + \Delta_B \trueK) \Res{\trueA + \trueB \trueK}$. The first inequality follows from the optimality of  $(\tf \Phi_{x}^\star,\tf\Phi_{u}^\star ,V_\star, \gamma_\star)$, the equality and second inequality from the fact that $(\widetilde{\tf\Phi}_x, \widetilde{\tf\Phi}_u)$ are truncations of the response of $\trueK$ on $(\Ah,\Bh)$ to the first $L$ time steps, and the final inequality by following similar arguments to the proof of Theorem \ref{thm:lqr_cost}, and in applying Theorem \ref{thm:robust}.

Noting that
\[
\hinfnorm{\tf \Delta} = \bighinfnorm{({\Delta_A}+{\Delta_B}\trueK)\Res{\trueA + \trueB \trueK}} \leq \zeta < 1,
\]
we then have that
\[
{J(\trueA,\trueB,\tf K_L)} \leq \frac{1}{1- \tilde \gamma} \frac{1}{1-\zeta}{J_\star},
\]
Recalling that $\tilde \gamma = \frac{4 \zeta}{1 - \zeta}$, we obtain 
\begin{align*}
\frac{J(\trueA,\trueB,\tf K_L) - J_\star}{J_\star} \leq \frac{1 - \zeta}{1- 5\zeta} \frac{1}{1-\zeta} - 1 = \frac{5 \zeta}{(1 - 5\zeta)} \leq 10 \zeta,
\end{align*}
where the last equality is true when $\zeta \leq 1/10$. The conclusion follows.

\end{proof}


\section{A Common Lyapunov Relaxation for Proportional Control}
\label{appendix:common_lyap}

We unpack each of the norms in~\eqref{eq:robustLQRbnd-static2} as linear matrix inequalities.  First, by the KYP Lemma, the $\hinf$ constraint is satisfied if and only if there exists a matrix $P_\infty$ satisfying
\begin{align*}
\begin{bmatrix}
(\Ah+\Bh K)^* P_\infty(\Ah+\Bh K)-P_\infty & (\Ah+\Bh K)^*P_\infty\\
P_\infty (\Ah+\Bh K) & P_\infty
\end{bmatrix}+
\begin{bmatrix}
	\gamma^{-2} \begin{bmatrix} \tfrac{\epsilon_A}{\sqrt{\alpha}}  \\ \tfrac{\epsilon_B}{\sqrt{1-\alpha}} K\end{bmatrix}^* \begin{bmatrix} \tfrac{\epsilon_A}{\sqrt{\alpha}}  \\ \tfrac{\epsilon_B}{\sqrt{1-\alpha}} K\end{bmatrix} & 0\\
	0 & - I
\end{bmatrix} \preceq 0 \:.
\end{align*}
Applying the Schur complement Lemma, we can reformulate this as the equivalent matrix inequality
\begin{align*}
\begin{bmatrix}
-P_\infty^{-1} & 0 & 0 & (\Ah+\Bh K) & I\\
0 & -\gamma^2 I & 0  & \tfrac{\epsilon_A}{\sqrt{\alpha}} I & 0\\
0 & 0 & -\gamma^2 I & \tfrac{\epsilon_B}{\sqrt{1-\alpha}} K & 0\\
(\Ah + \Bh K)^* &\tfrac{\epsilon_A}{\sqrt{\alpha}} I & \tfrac{\epsilon_B}{\sqrt{1-\alpha}} K^*  & -P_\infty & 0\\
I & 0 & 0 & 0 & -I
\end{bmatrix} \preceq 0 \:.
\end{align*}
Then, conjugating by the matrix $\operatorname{diag}(I,I,P_\infty^{-1},I)$ and setting $X_\infty = P_\infty^{-1}$, we are left with
\begin{align*}
\begin{bmatrix}
-X_\infty & 0 & 0 & (\Ah+\Bh K)X_\infty & I\\
0 & -\gamma^2 I & 0  & \tfrac{\epsilon_A}{\sqrt{\alpha}} X_\infty & 0\\
0 & 0 & -\gamma^2 I & \tfrac{\epsilon_B}{\sqrt{1-\alpha}} KX_\infty & 0\\
X_\infty(\Ah + \Bh K)^* &\tfrac{\epsilon_A}{\sqrt{\alpha}}  X_\infty & \tfrac{\epsilon_B}{\sqrt{1-\alpha}} X_\infty K^*  & -X_\infty & 0\\
I & 0 & 0 & 0 & -I
\end{bmatrix} \preceq 0 \:.
\end{align*}
Finally, applying the Schur complement lemma again gives the more compact inequality
\begin{align*}
\begin{bmatrix}
-X_\infty+I & 0 & 0 & (\Ah+\Bh K)X_\infty\\
0 & -\gamma^2 I & 0  & \tfrac{\epsilon_A}{\sqrt{\alpha}} X_\infty \\
0 & 0 & -\gamma^2 I & \tfrac{\epsilon_B}{\sqrt{1-\alpha}} KX_\infty \\
X_\infty(\Ah + \Bh K)^* &\tfrac{\epsilon_A}{\sqrt{\alpha}}  X_\infty & \tfrac{\epsilon_B}{\sqrt{1-\alpha}} X_\infty K^*  & -X_\infty \\
\end{bmatrix} \preceq 0 \:.
\end{align*}
For convenience, we permute the rows of this inequality and conjugate by $\operatorname{diag}(I,I,\sqrt{\alpha}I,\sqrt{1-\alpha} I )$ and use the equivalent form
 \begin{align*}
 \begin{bmatrix} -X_\infty+I &  (\Ah+\Bh K)X_\infty  & 0 & 0 \\
    X_\infty(\Ah+\Bh K)^* & -X_\infty & \epsilon_A X_\infty & \epsilon_B X_\infty K^* \\
    0 & \epsilon_A X_\infty  & -\alpha\gamma^2 I & 0\\
    0 & \epsilon_B KX_\infty  & 0 & -(1-\alpha)\gamma^2 I \end{bmatrix} \preceq 0 \:.
\end{align*}
For the $\htwo$ norm, we have that under proportional control $K$, the average cost is given by
$ \operatorname{Trace}((Q +K^*R K)X_2)$ where $X_2$ is the steady state covariance. That is, $X_2$ satisfies the Lyapunov equation
\begin{align*}
	X_2 = (\Ah+\Bh K) X_2(\Ah+\Bh K)^* +I \,.
\end{align*}
But note that we can relax this expression to a matrix inequality
\begin{align}\label{eq:htwo-lyap}
	X_2 \succeq (\Ah+\Bh K) X_2(\Ah+\Bh K)^* +I \:,
\end{align}
and $ \operatorname{Trace}((Q +K^*R K)X_2)$  will remain an upper bound on the squared $\htwo$ norm.  Rewriting this matrix inequality with Schur complements and combining with our derivation for the $\hinf$ norm, we can reformulate~\eqref{eq:robustLQRbnd-static2} as a nonconvex semidefinite program
\begin{align}\label{eq:robustLQRbnd-static-nc}
\begin{array}{ll}
\operatorname{minimize}\limits_{X_2,X_\infty,K,\gamma} &\frac{1}{(1-\gamma)^2}  \operatorname{Trace}((Q +K^*R K)X_2) \\
\mbox{subject to} & \begin{bmatrix} X_2 -I & (\Ah+\Bh K)X_2 \\ X_2(\Ah+\Bh K)^*  & X_2\end{bmatrix}\succeq 0\\
&\begin{bmatrix} X_\infty-I &  (\Ah+\Bh K)X_\infty  & 0 & 0 \\
    X_\infty(\Ah+\Bh K)^* & X_\infty & \epsilon_A X_\infty & \epsilon_B X_\infty K^* \\
    0 & \epsilon_A X_\infty  & \alpha\gamma^2 I & 0\\
    0 & \epsilon_B KX_\infty  & 0 & (1-\alpha)\gamma^2 I \end{bmatrix} \succeq 0 \:.
\end{array}
\end{align}
The common Lyapunov relaxation simply imposes that $X_2 = X_\infty$.  Under this identification, we note that the first LMI becomes redundant and we are left with the SDP
\begin{align*}
\begin{array}{ll}
\operatorname{minimize}\limits_{X,K,\gamma} &\frac{1}{(1-\gamma)^2}  \operatorname{Trace}((Q +K^*R K)X) \\
\mbox{subject to}
&\begin{bmatrix} X-I &  (\Ah+\Bh K)X  & 0 & 0 \\
    X(\Ah+\Bh K)^* & X  & \epsilon_A X & \epsilon_B X K^* \\
    0 & \epsilon_A X & \alpha\gamma^2 I & 0\\
    0 & \epsilon_B KX & 0 & (1-\alpha)\gamma^2 I \end{bmatrix} \succeq 0 \:.
\end{array}
\end{align*}
Now though this appears to be nonconvex, we can perform the standard variable substitution $Z=KX$ and rewrite the cost to yield~\eqref{eq:robustLQRbnd-common-lyap}.

\section{Numerical Bootstrap Validation}
\label{sec:bootstrap-experiments}

We evaluate the efficacy of the bootstrap
procedure introduced in Algorithm~\ref{alg:bootstrap}.  Recall that even
though we provide theoretical bounds in
Proposition~\ref{prop:independent_estimation}, for practical purposes and for
handling dependent data, we want bounds that are the least conservative
possible.

For given state dimension $\statedim$, input dimension $\inputdim$, and scalar $\rho$, we generate upper triangular
matrices $\trueA \in \RR^{\statedim \times \statedim}$ with all diagonal
entries equal to $\rho$ and the upper triangular entries i.i.d. samples from
$\Ncal(0,1)$, clipped at magnitude $1$. By construction, matrices will have spectral radius $\rho$.
The entries of $\trueB\in \RR^{\statedim \times \inputdim}$ were sampled
i.i.d. from $\Ncal(0,1)$, clipped  at magnitude $1$. The variance
 terms  $\sigma_u^2$ and $\sigma_w^2$ were fixed to be $1$.

Recall from Section~\ref{sec:bootstrap} that $M$ represents the number of
trials used for the bootstrap estimation, and $\widehat \epsilon_A$, $\widehat \epsilon_B$ are the bootstrap estimates for $\epsilon_A$, $\epsilon_B$. To check the validity of the
bootstrap procedure we empirically estimate the fraction of time $\trueA$
 and $\trueB$ lie in the balls $B_{\widehat{A}}(\widehat{\epsilon}_A)$ and $B_{\widehat{B}}(\widehat{\epsilon}_B)$, where $B_{X}(r) = \{X'\colon \ltwonorm{X' - X} \leq r\}$.

Our findings are
summarized in
Figures~\ref{fig:bootstrap_AB1}
and~\ref{fig:bootstrap_AB2}.
Although not plotted, the theoretical bounds found
in Section~\ref{sec:estimation} would be orders of magnitude larger
than the true $\epsilon_A$ and $\epsilon_B$, while the bootstrap bounds offer a
good approximation.

\def \sbasefigwidth{0.38}
\def \abasefigwidth{0.49}

\begin{figure}[h]
\centering
\begin{subfigure}[b]{\sbasefigwidth\textwidth}
  \caption{\footnotesize  Estimation Error in $A$}
\centerline{\includegraphics[width=\columnwidth]{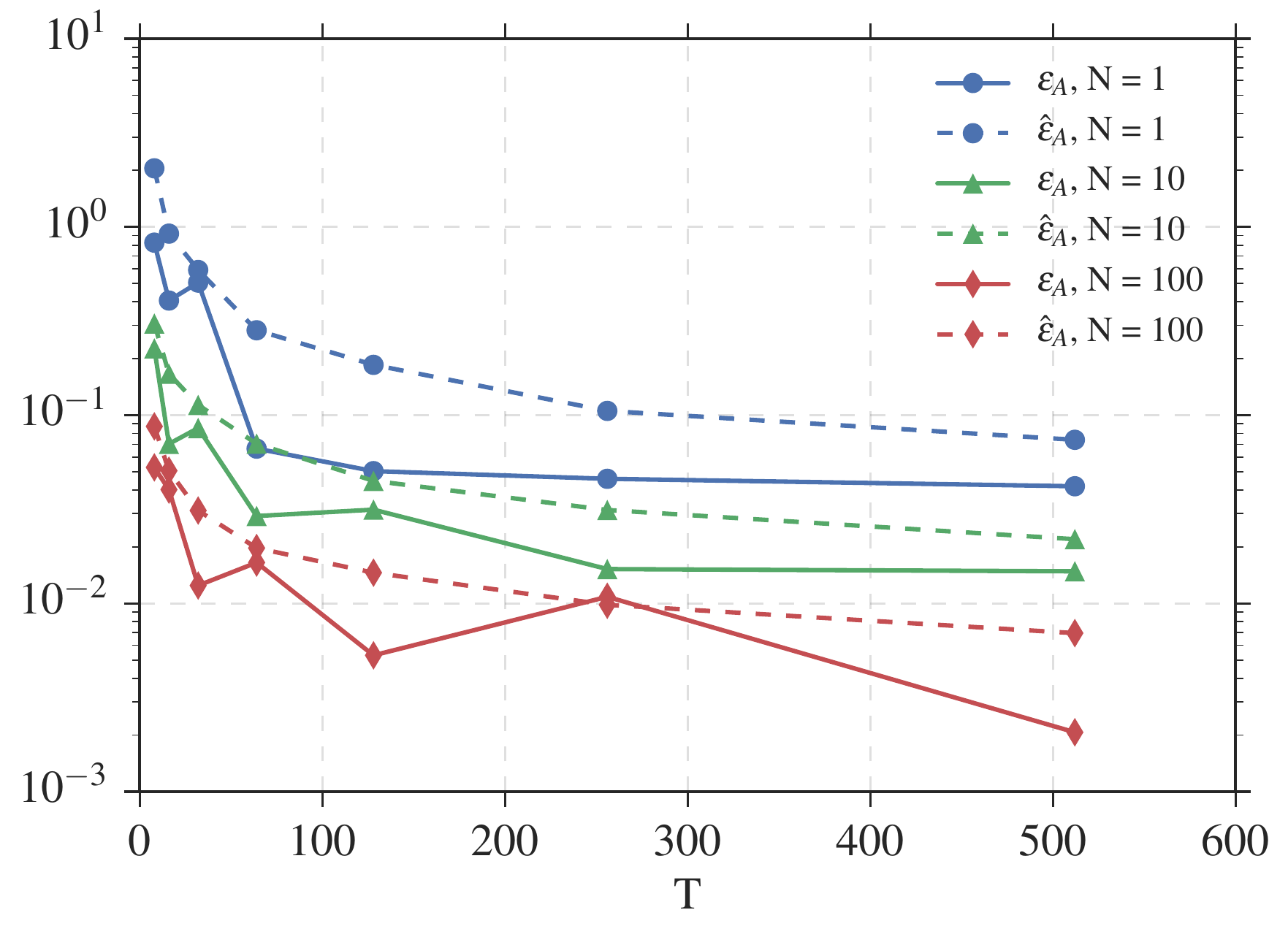}}
\end{subfigure}
\begin{subfigure}[b]{\sbasefigwidth\textwidth}
  \caption{\footnotesize  Correctness of Bootstrap Estimate}
\centerline{\includegraphics[width=\columnwidth]{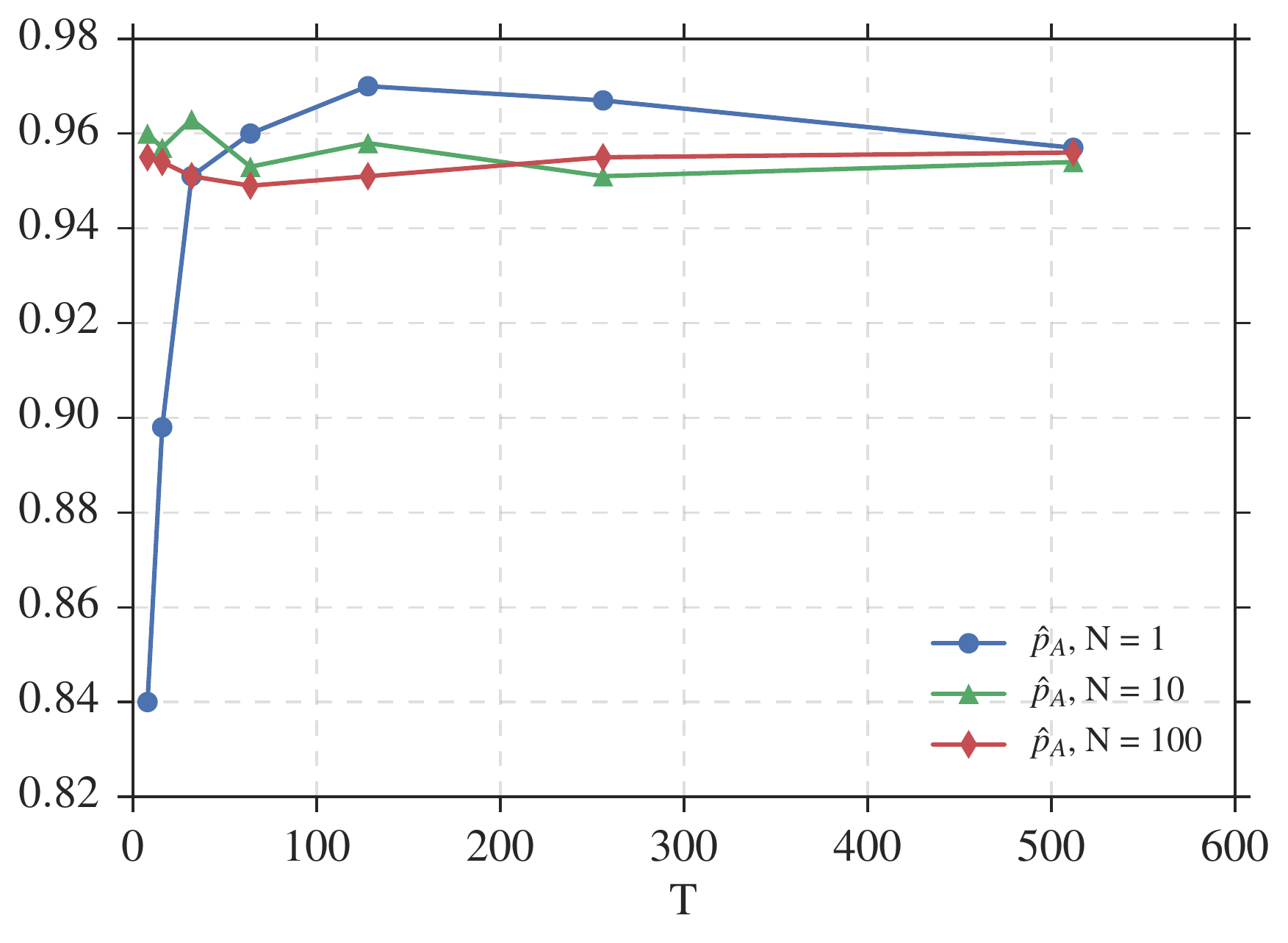}}
\end{subfigure}
\begin{subfigure}[b]{\sbasefigwidth\textwidth}
\caption{\footnotesize  Estimation Error in $B$}
\centerline{\includegraphics[width=\columnwidth]{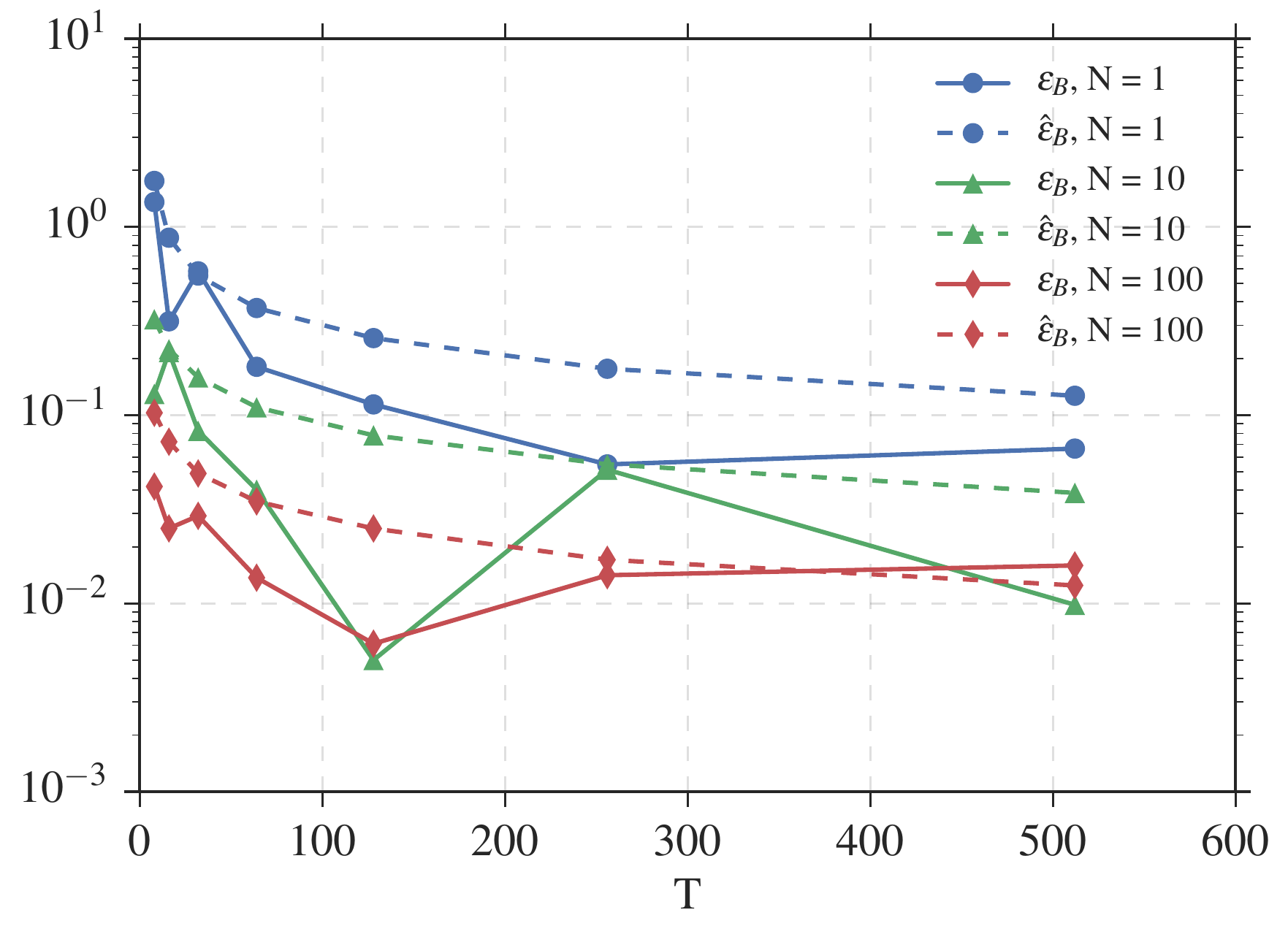}}
\end{subfigure}
\begin{subfigure}[b]{\sbasefigwidth\textwidth}
\caption{\footnotesize  Correctness of Bootstrap Estimate}
\centerline{\includegraphics[width=\columnwidth]{{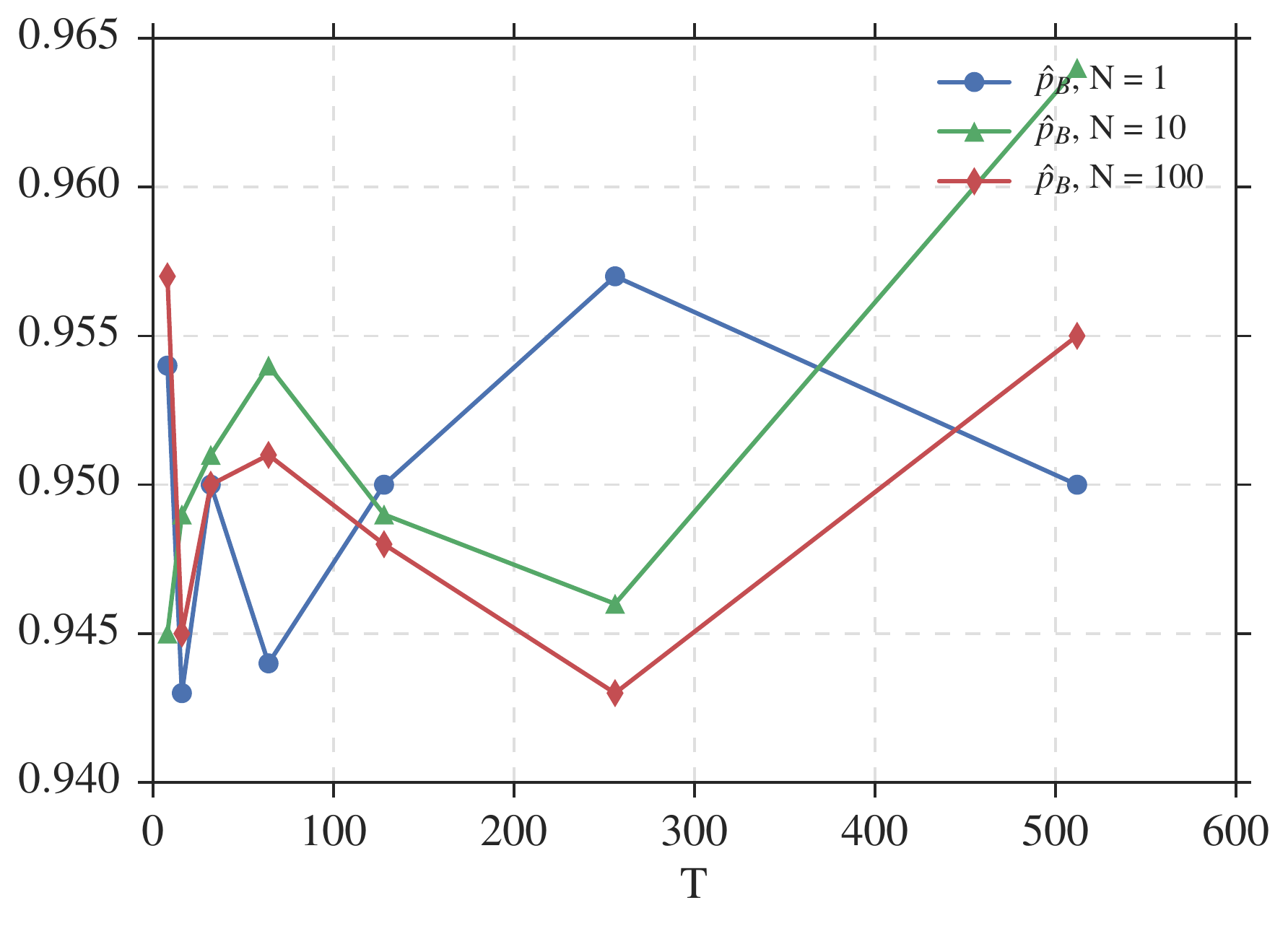}}}
\end{subfigure}
\caption{\small  In these simulations: $\statedim = 3$, $\inputdim = 1$, $\rho = 0.9$, and $M = 2000$. In (a), the spectral distances to $\trueA$ (shown in the solid lines) are compared with
  the bootstrap estimates (shown in the dashed lines). In (b), the probability $\trueA$ lies in $B_{\widehat{A}}(\widehat{\epsilon}_A)$ estimated from $2000$ trials.
 In (c), the spectral distances to $B_*$ are compared with the bootstrap estimates. In (d), the probability $\trueB$ lies in $B_{\widehat{B}}(\widehat{\epsilon}_B)$ estimated from $2000$ trials.}
\label{fig:bootstrap_AB1}
\end{figure}

\begin{figure}[h]
\centering
\begin{subfigure}[b]{\sbasefigwidth\textwidth}
\caption{\footnotesize  Estimation Error in $A$}
\centerline{\includegraphics[width=\columnwidth]{{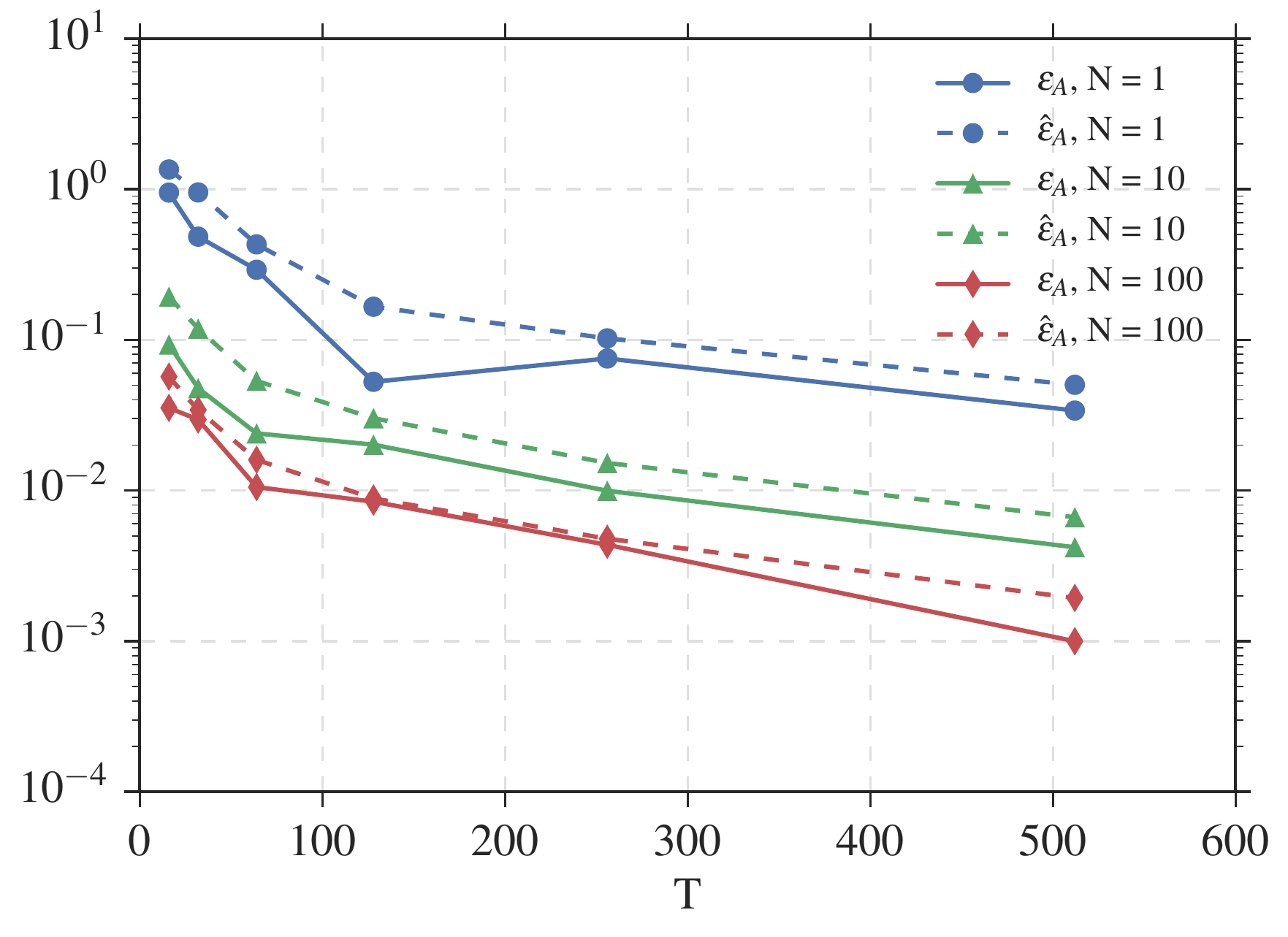}}}
\end{subfigure}
\begin{subfigure}[b]{\sbasefigwidth\textwidth}
\caption{\footnotesize Correctness of Bootstrap Estimate}
\centerline{\includegraphics[width=\columnwidth]{{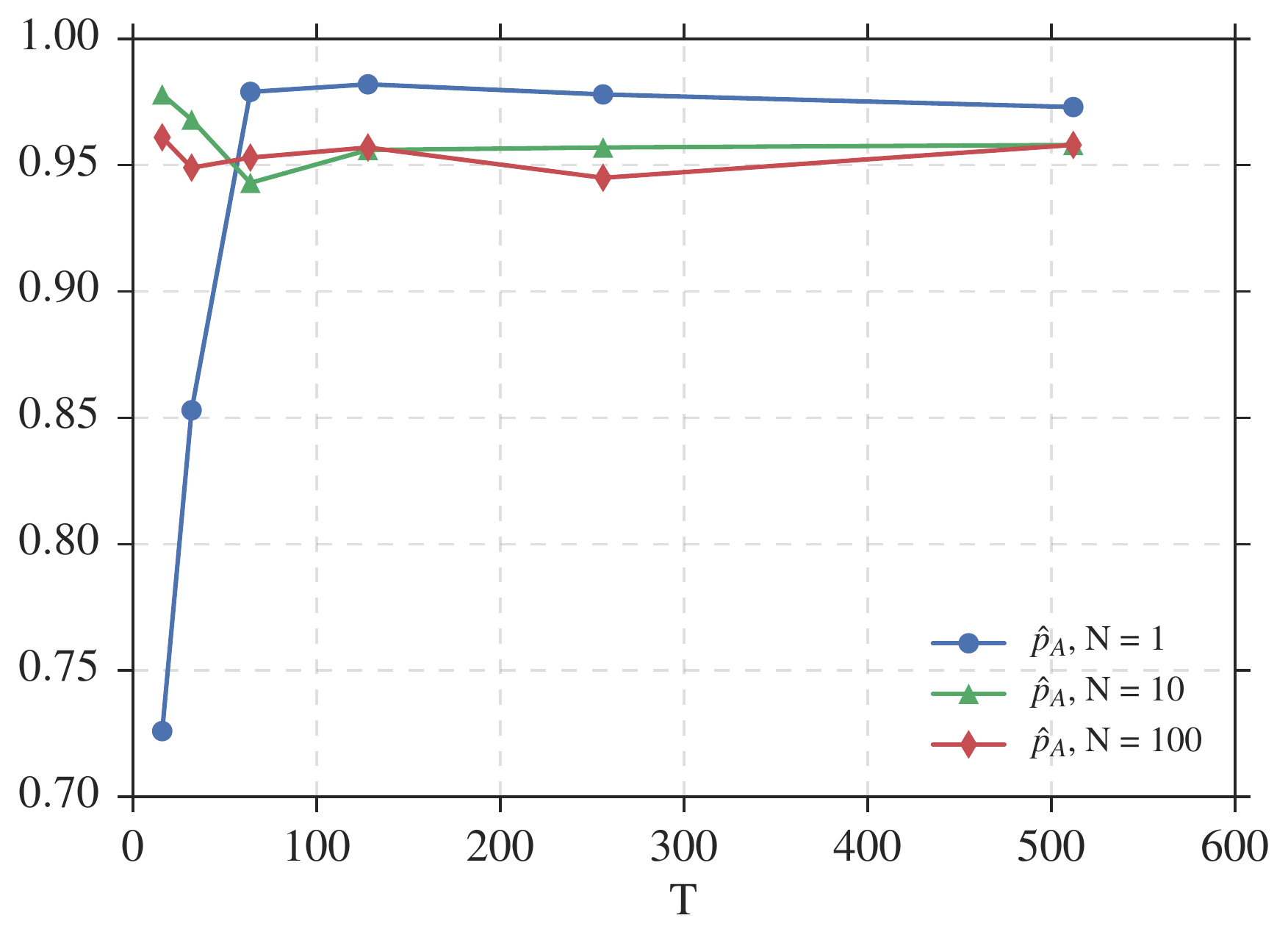}}}
\end{subfigure}
\begin{subfigure}[b]{\sbasefigwidth\textwidth}
\caption{\footnotesize  Estimation Error in $B$}
\centerline{\includegraphics[width=\columnwidth]{{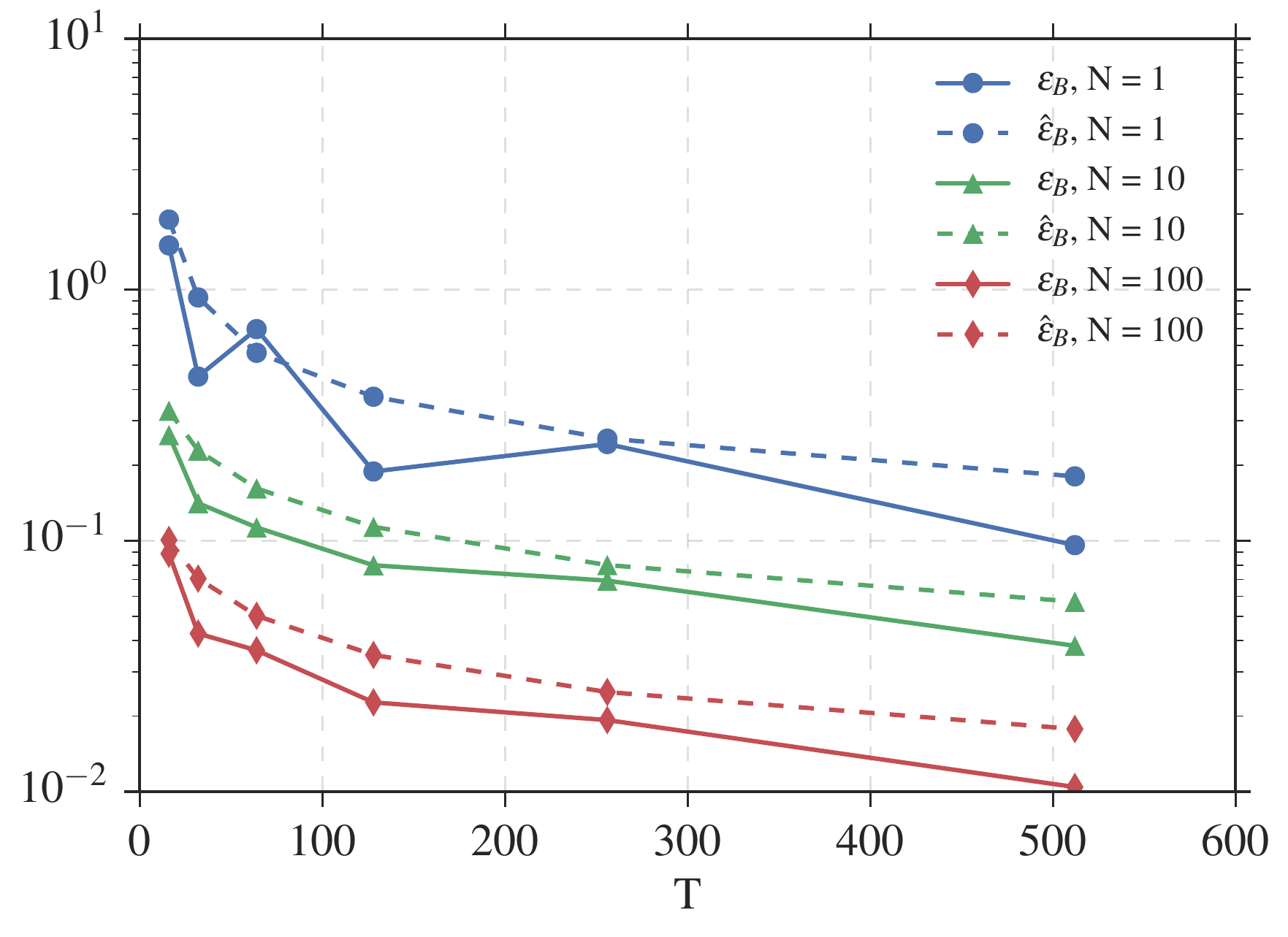}}}
\end{subfigure}
\begin{subfigure}[b]{\sbasefigwidth\textwidth}
\caption{\footnotesize  Correctness of Bootstrap Estimate}
\centerline{\includegraphics[width=\columnwidth]{{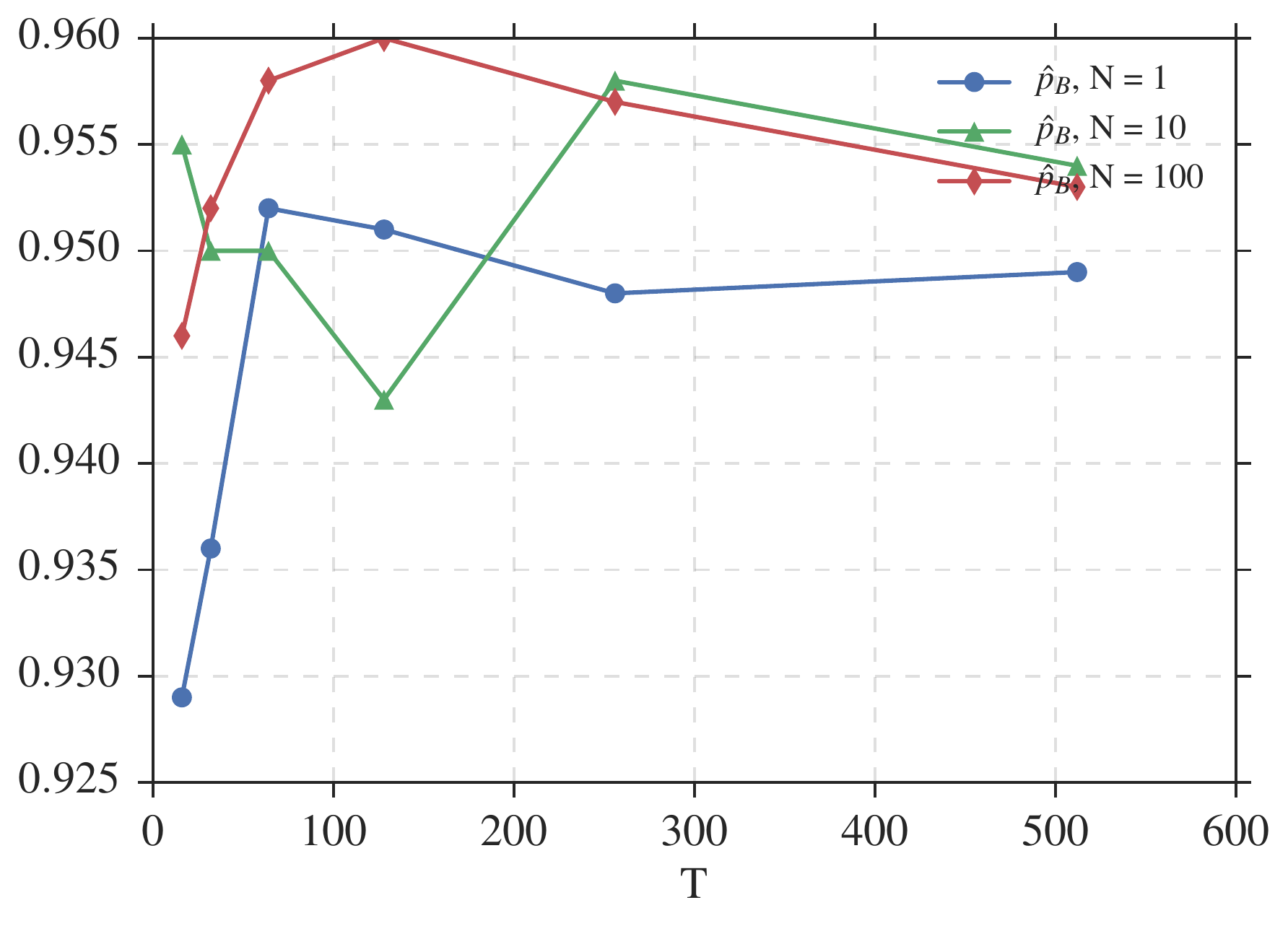}}}
\end{subfigure}
\caption{\small In these simulations: $\statedim = 6$, $\inputdim = 2$, $\rho = 1.01$, and $M = 2000$. In (a), the spectral distances to $\trueA$ are compared with the bootstrap estimates. In (b), the probability $\trueA$ lies in $B_{\widehat{A}}(\widehat{\epsilon}_A)$ estimated from $2000$ trials. In (c), the spectral distances to $\trueB$ are compared with the bootstrap estimates. In (d), the probability $\trueB$ lies in $B_{\widehat{B}}(\widehat{\epsilon}_B)$ estimated from $2000$ trials.}
\label{fig:bootstrap_AB2}
\end{figure}

\section{ Experiments with Varying Rollout Lengths}\label{app:eps_v_trial_figs}

Here we include results of experiments in which we fix the number of trials ($N=6$) and vary the rollout length.
Figure \ref{fig:eps_v_trial} displays the estimation errors. The estimation errors on $\trueA$ decrease more quickly than in the fixed rollout length case, consistent with the idea that longer rollouts of easily excitable systems allow for better identification due to higher signal to noise ratio. Figure \ref{fig:perf_v_trial} shows that stabilizing performance of the nominal is somewhat better than in the fixed rollout length case (Figure \ref{fig:perf_v_rollout}). This fact is likely related to the smaller errors on the estimation of $\trueA$ (Figure \ref{fig:eps_v_trial}).

\begin{figure}[h]
\centering
\begin{subfigure}[b]{\abasefigwidth\textwidth}
\caption{\footnotesize  Least Squares Estimation Errors}
\centerline{\includegraphics[width=\columnwidth]{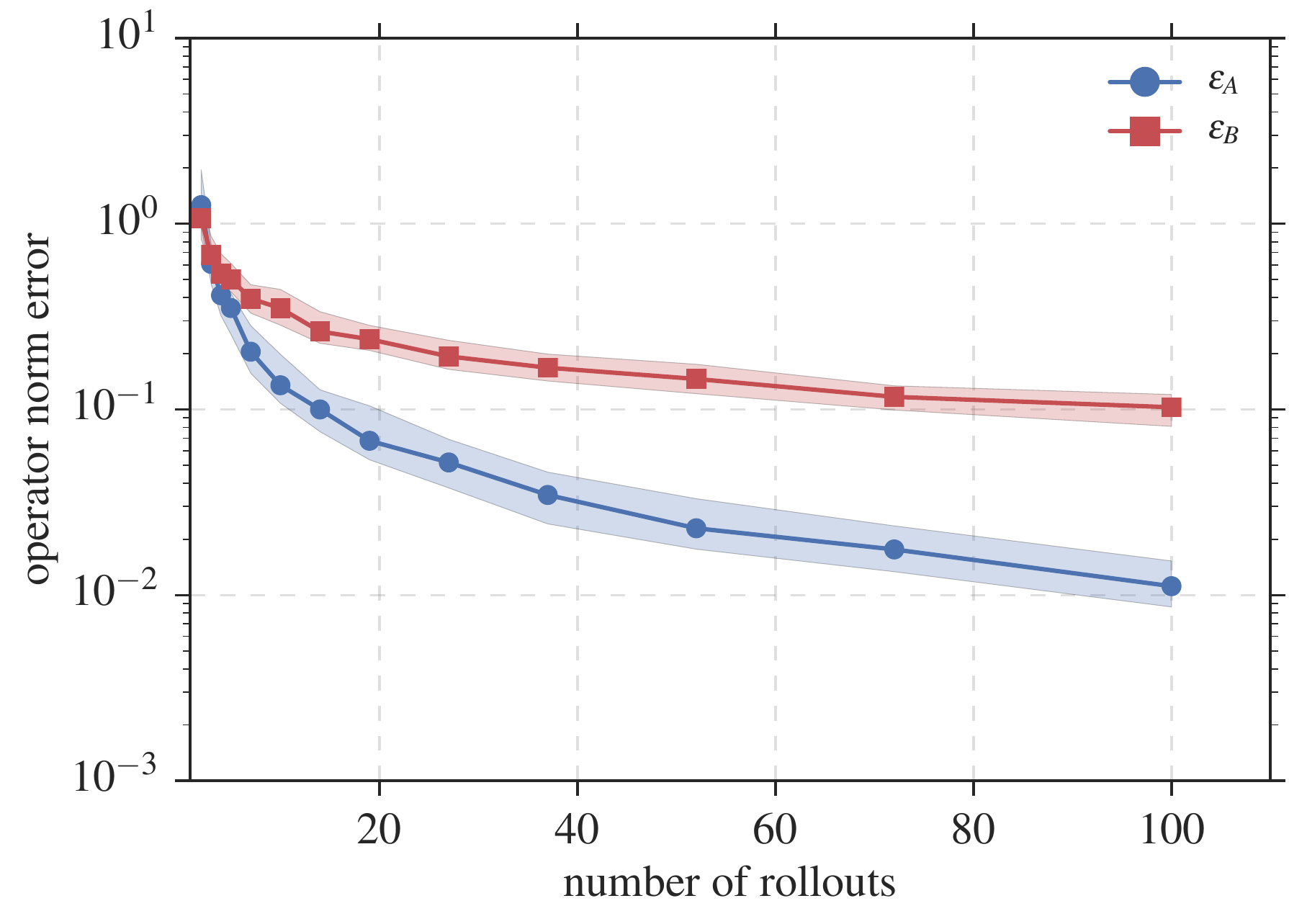}}
\end{subfigure}
\begin{subfigure}[b]{\abasefigwidth\textwidth}
\caption{\footnotesize  Accuracy of Bootstrap Estimates}
\centerline{\includegraphics[width=\columnwidth]{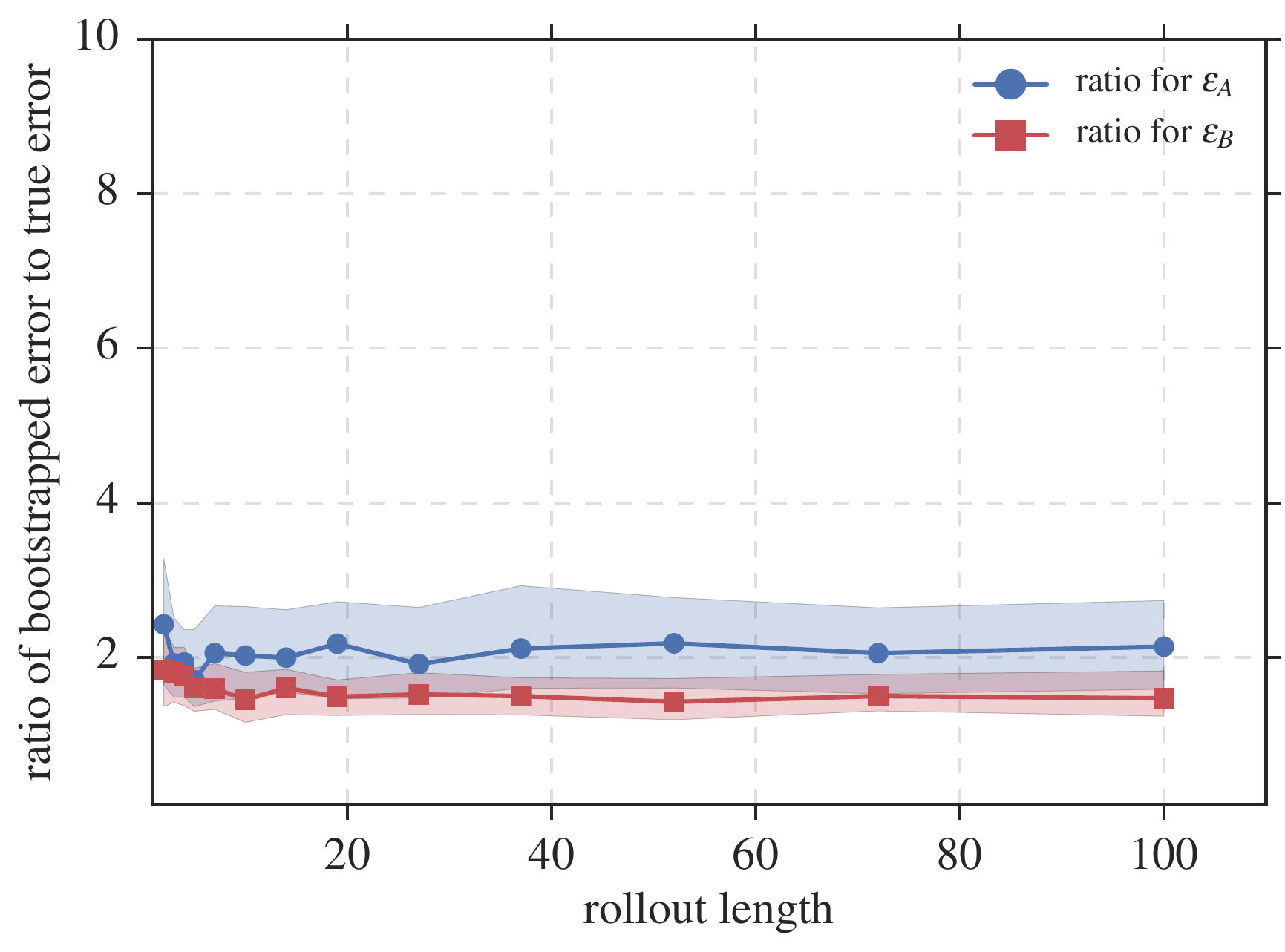}}
\end{subfigure}
\caption{\small The resulting errors from 100 identification experiments with with a total of $N=6$ rollouts is plotted against the length rollouts. In (a), the median of the least squares estimation errors decreases with $T$. In (b), the ratio of the bootstrap estimates to the true estimates. Shaded regions display quartiles.}
\label{fig:eps_v_trial}
\end{figure}

\begin{figure}[h]
\centering
\begin{subfigure}[b]{\abasefigwidth\textwidth}
\caption{\footnotesize  LQR Cost Suboptimality}
\centerline{\includegraphics[width=\columnwidth]{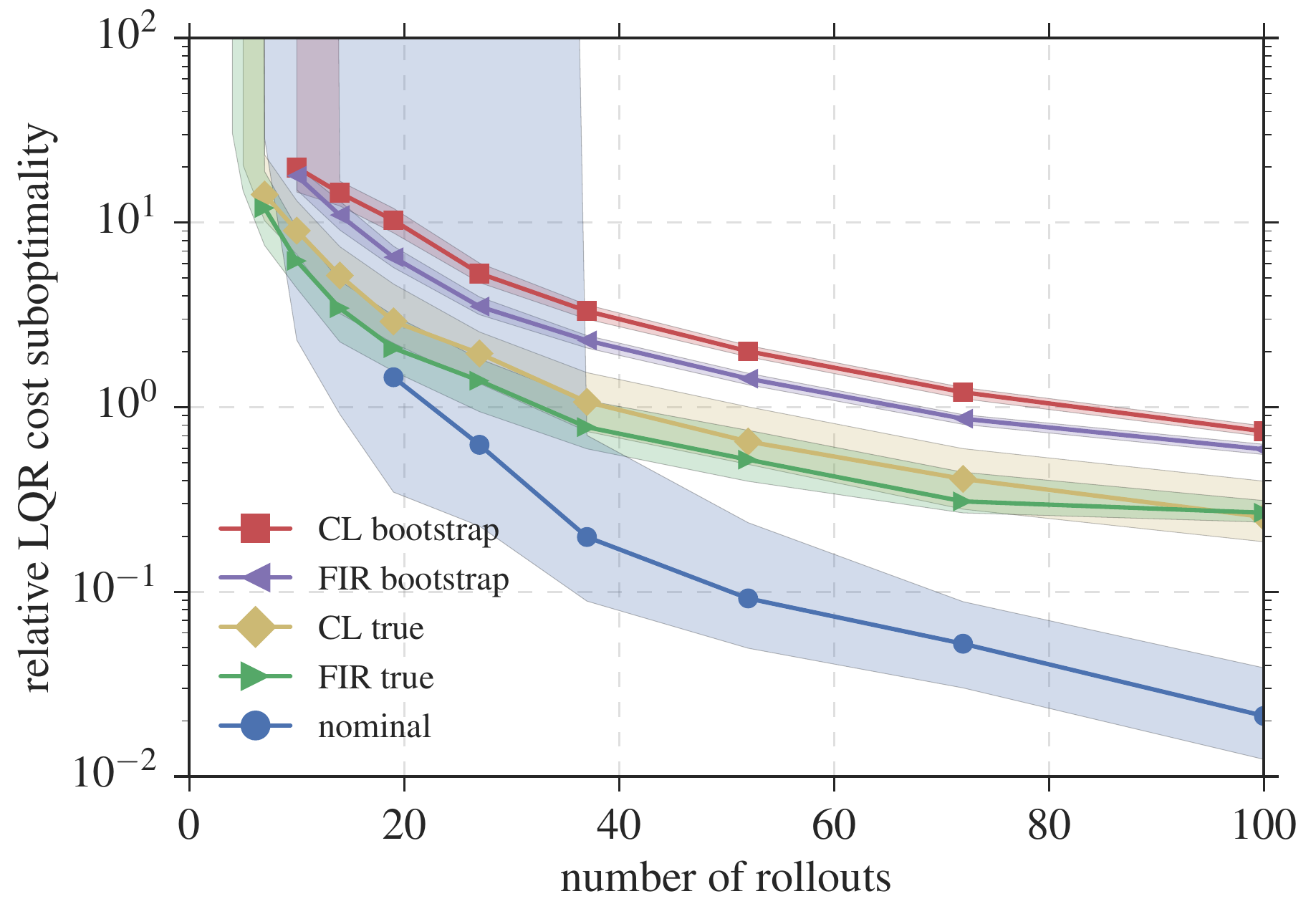}}
\end{subfigure}
\begin{subfigure}[b]{\abasefigwidth\textwidth}
\caption{\footnotesize  Frequency of Stabilization}
\centerline{\includegraphics[width=\columnwidth]{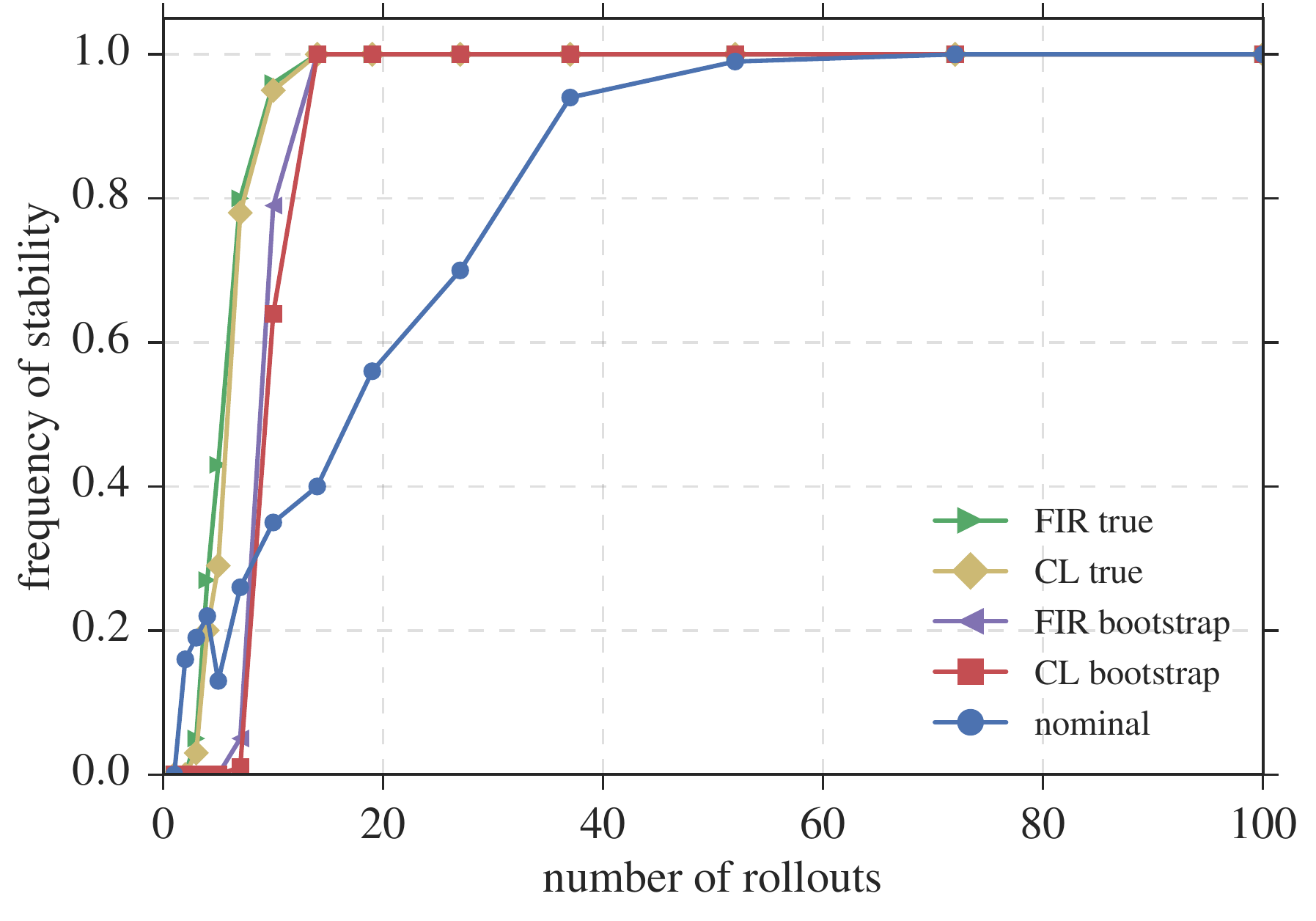}}
\end{subfigure}
\caption{\small  The performance of controllers synthesized on the results of the 100 identification experiments is plotted against the length of rollouts. In (a), the median suboptimality of nominal and robustly synthesized controllers are compared, with shaded regions displaying quartiles, which go off to infinity when stabilizing controllers are not frequently found. In (b), the frequency synthesis methods found stabilizing controllers.}
\label{fig:perf_v_trial}
\end{figure}

\end{document}